\documentclass[letterpaper,10pt]{article}
\usepackage{paralist, amsmath, amsthm, amsfonts, amssymb, xy,}
\usepackage[margin=1in]{geometry} 
\usepackage{hyperref}
\usepackage[T1]{fontenc}
\usepackage{calligra}

\input xy
\xyoption{all}

\newcommand{\NN}{\mathbb{N}}
\newcommand{\ZZ}{\mathbb{Z}}
\newcommand{\SM}{\mathbf{Sm}_k}
\newcommand{\SCH}{\mathbf{Sch}_k}
\newcommand{\LCI}{\mathbf{Lci}_k}

\newcommand{\srarrow}{\twoheadrightarrow}
\newcommand{\irarrow}{\hookrightarrow}

\newcommand{\flag}{\mathcal{F}\ell}
\newcommand{\Proj}{\mathbb{P}}
\newcommand{\Laz}{\mathbb{L}}
\newcommand{\trecd}{\cdot\cdot\cdot}
\newcommand{\tred}{\ldots}
\newcommand{\unddot}{_\textbf{\textbullet}}
\newcommand{\variables}{[x_1,\ldots,x_n,y_1,\ldots,y_n]}
\newcommand{\spec}{{\rm Spec\,}}

\newtheorem{theorem}{Theorem}[subsection]
\newtheorem{lemma}[theorem]{Lemma}
\newtheorem{proposition}[theorem]{Proposition}
\newtheorem{corollary}[theorem]{Corollary}
\newtheorem{definition}[theorem]{Definition}

\theoremstyle{definition} \newtheorem{remark}[theorem]{Remark}
\theoremstyle{definition} \newtheorem{example}[theorem]{Example}
\theoremstyle{definition} 
\theoremstyle{definition} \newtheorem{question}[theorem]{Question}
\date{}

\begin{document}

\begin{center}
{\LARGE Thom-Porteous formulas in algebraic cobordism}

\vspace{0.3 cm}

{\Large Thomas Hudson}

\end{center}

\tableofcontents

%\introduction
\section{Introduction}%\addcontentsline{toc}{chapter}{Introduction}
The motivating question that represented the starting point of this thesis can be phrased as follows: ``Is there any analogue for algebraic cobordism of the Thom-Porteous formula with values in the Chow ring?''. Given a morphism of vector bundles $h: E\rightarrow F$ over a pure dimensional Cohen-Macaulay scheme $X$ such that the degeneracy locus 
$$D_n(h):=\{x\in X\ |\ {\rm rank}(h(x):E(x)\rightarrow F(x))\leq n\}$$
has the expected codimension in $X$, the Thom-Porteous formula allows one to write the Chow ring-valued fundamental class $[D_n(h)]_{CH}$ as a determinant in the Chern classes of the two bundles.
On the other hand the theory of algebraic cobordism $\Omega^*$ was established by Levine and Morel as an algebraic geometric analogue of complex cobordism. From our point of view the key feature of algebraic cobordism is that it represents the universal oriented cohomology theory on smooth schemes. This in particular implies that it can be seen as a powerful generalization of the Chow ring: to be able to find such a formula in the context of algebraic cobordism would have consequences for all other oriented cohomology theories. 

Following the work of Fulton in \cite{FlagsFulton}, we have decided to restrict our attention to degeneracy loci of morphisms of vector bundles endowed with full flags and in particular to the universal case represented by the full flag bundle $\flag(V)$ over a scheme $X$. In this setting the degeneracy loci are the Schubert varieties $\Omega_\omega$ and Fulton has showed that their fundamental classes are given by double Schubert polynomials evaluated at the Chern roots of the defining bundles. From this special case he then recovers the general case by pulling back to the base the fundamental class of the appropriate Schubert variety, therefore providing a description of the fundamental class of the degeneracy loci in terms of double Schubert polynomials. 
  
Later, in \cite{PieriFulton} Fulton and Lascoux considered once again the universal case but this time they aimed at giving a description of the fundamental classes of Schubert varieties in the Grothendieck ring of vector bundles. The formula they found, which expresses the fundamental classes in terms of the double Grothendieck polynomials of Lascoux and Sch\"{u}tzenberger, formally resembles the one in the Chow ring case and it is proved following essentially the same pattern. Even though they are not explicitly mentioned, in both proofs a central role is played by the Bott-Samelson resolutions: it is the push-forward of their fundamental classes that can be naturally described by double Schubert and Grothendieck polynomials. On the other hand Bott-Samelson resolutions also happen to be desingularizations of Schubert varieties, it is this fact that allows to bring back into the picture the fundamental classes $[\Omega_\omega]$. 

The study of the Grothendieck ring case was finally completed by Buch in \cite{GrothendieckBuch}, where he manages to express the fundamental class of a general degeneracy locus by means of Grothendieck polynomials. 
 
In view of these results we wondered if the method designed by Fulton could also be used  in the framework of algebraic cobordism. As we have already mentioned, Levine and Morel have showed that algebraic cobordism is the universal oriented cohomology theory and as such it generalizes both the Chow ring and the Grothendieck ring.  Even though this last fact alone would justify our interest in the problem, there is another aspect which is worth underlining: the universality of algebraic cobordism makes it possible to study the question in many oriented cohomology theories at once, highlighting what conditions the theory has to satisfy so that the different steps of the proof go through. In some sense even the goals can change according to the theory one considers.

Let us give an easy illustration of this phenomenon. As we have already mentioned, Fulton's approach in the original setting consists of two main parts: computing the classes associated to the Bott-Samelson resolutions and relating them to the fundamental classes of Schubert varieties. In case one considers algebraic cobordism already at this very primitive stage the final goal has to be modified: in algebraic cobordism only local complete intersection schemes have a well defined notion of fundamental class, so it is not possible to associate a fundamental class to each Schubert variety. On the other hand it is well possible that there exist other theories, less general than $\Omega^*$, in which fundamental classes are defined and within those theories one can still try to carry on the second part of the computation.

The first successful attempts of solving this kind of problem in the context of algebraic cobordism were carried out by Hornbostel and Kiritchenko in \cite{SchubertHornbostel} and by Calmes, Petrov and Zainoulline in \cite{SchubertCalmes}. In particular, Hornbostel and Kiritchenko gave an explicit description of the push forward map along $\Proj^1$-bundles which they used to compute the push-forward classes of Bott-Samelson resolutions in the case of the flag manifold or, in other words, when the base scheme $X$ is a point. By making use of their computations we have succeeded in extending their result to a general flag bundle, hence allowing any smooth base $X$. 

At this point it is important to mention that there are many Bott-Samelson resolutions associated to the same Schubert variety. In the two classical cases this fact did not play any role because taking the push-forward had the effect of making the different classes equal. On the other hand, when dealing with algebraic cobordism this coincidence is not guaranteed anymore. One way out of this situation is to consider a more restrictive oriented cohomology theory for which the push-forward classes have to coincide. One possible choice, which still generalizes both the Chow ring and the Grothendieck ring, is to consider connected $K$-theory. When the formula obtained for cobordism is translated in this setting, not only we recover the equality as in the original cases, but we also manage to provide a geometric interpretation to the double $\beta$-polynomials defined in \cite{GrothendieckFomin} by Fomin and Kirillov for combinatorial purposes.     

Let us now outline the internal organization of our work. In section \ref{ch 2} we recall the necessary background material on algebraic cobordism and its relations with other oriented cohomology theories, in particular with connected $K$-theory. We also perform some computations with Chern classes that will be used in section \ref{ch 3}.

 In section \ref{ch 1} we introduce the geometric entities that represent the object of our study and we provide a detailed presentation of the method used by Fulton in the Chow ring case. In this section we also present the double Schubert, Grothendieck and $\beta$-polynomials together with the results of Fulton-Lascoux and Buch in the case of the Grothendieck ring of vector bundles.

In section \ref{ch 3}, after presenting the results of Hornbostel and Kiritchenko on the flag manifold, we compute the push-forward classes of Bott-Samelson resolutions in the algebraic cobordism of the flag bundle. We then specialize our formula to connected $K$-theory, hence giving a geometric interpretation to the $\beta$-polynomials of Fomin and Kirillov.

\textbf{Acknowledgements: }
My deepest gratitude goes to Prof. M. Levine for his patient support, for all the things that he managed to teach me and for those that I managed not to learn: without him this thesis would have never been possible.
I am as well greatly indebted to Prof. J. Weyman: his direct and energetic guidance saved me from quite a few pitfalls and dead ends.
I would like to thank Prof. A. S. Buch and Prof. J. Hornbostel for their careful reading of an early version of this work. Their comments and suggestions not only improved the manuscript, they also allowed me to see things from a different angle. In particular, I am indebted to Prof. A. S. Buch for bringing my attention to $\beta$-polynomials.

I might not have started my graduate studies had it not been for the much needed encouragement and active help of Prof. A. Collino and of Prof. A. Mori. 
I am grateful to Prof. V. Lakshmibai, Prof. A. Martsinkovsky, Prof. D. B. Massey, Prof. A. Suciu, Prof. G. Todorov and Prof. V. Toledano-Laredo for making my time at Northeastern University a fruitful and enjoyable experience, both inside and outside the classrooms.
Finally, I would like to thank all my fellow graduate students and the postdocs at Northeastern University and at the Universit\"{a}t Duisburg-Essen for providing me with a friendly and pleasant studying environment.

% \chapter{Algebraic cobordism and oriented cohomology theories}\label{ch 2}

\section{Algebraic cobordism and oriented cohomology theories}\label{ch 2}
The main goal of this section is to present the notions of oriented cohomology theory and oriented Borel-Moore homology theory and to describe the construction of algebraic cobordism. We moreover illustrate the relations existing between algebraic cobordism and other oriented Borel-Moore homology theories.

\subsection{The Lazard ring and the universal formal group law}
%\section{The Lazard ring and the universal formal group law}

In this subsection we recall the notion of formal group law and we introduce the universal such law on the Lazard ring. 
%We begin by introducing some terminology that will be needed to define oriented cohomology theories   

\begin{definition}\label{def FGL}

A commutative formal group law of rank one with coefficients in $R$ is a pair $(R,F)$, where $R$ is a commutative ring and $F(u,v)=\sum a_{i,j}u^iv^j\in R[[u,v]]$ is a formal power series satisfying the following conditions:  
\begin{enumerate} 
\item $F(u,0)=F(0,u)=u\in R[[u]];$
\item $F(u,v)=F(v,u)\in R[[u,v]];$
\item $F(u,F(v,w))=F(F(u,v),w)\in R[[u,v,w]].$
\end{enumerate}
A morphism of formal group laws $\phi:(R,F)\rightarrow (R',F')$ consists of a ring homomorphism $\varPhi:R\rightarrow R'$ such that $[\varPhi(F)](u,v):=\sum\varPhi(a_{i,j})u^i v^j$ equals $F'(u,v)$.
\end{definition}

%\begin{remark}\label{rem obvious}
%It is important to underline that condition (1) in the previous definition imply the %following equalities: $a_{1,0}=a_{0,1}=1$ and $a_{i,0}=a{0,i}=0$ for $i >1$.
%\end{remark}

\begin{definition}
Given a commutative formal group law $(R,F)$ there exists a unique power series $\chi_F(u)\in R[u]$ such that 
$$F(u,\chi_F(u))=0\ .$$
We will refer to $\chi_F(u)$ as the inverse for the formal group law $F$. 
\end{definition}

\begin{example}
Let $R$ be a commutative ring. Two elementary examples of formal group laws and their inverses are given by the \textit{additive formal group law} $$F_a(u,v)=u+v\  ,\ \chi_{F_a}(u)=-u$$ and by the \textit{multiplicative formal group law} 
$$F_m(u,v)=u+v-buv \ ,\ \ \chi_{F_m}(u)=\frac{-u}{1-bu}$$ for some choice of $b\in R$. One sees immediately that the additive formal group law can be recovered from the multiplicative one by setting $b=0$. A multiplicative formal group law is said \textit{periodic} if the element $b\in R$ is a unit.
\end{example}

We will now describe the construction of the Lazard ring.
Let $A=\{A_{i,j}\ |\ i,j\in \NN\setminus\{0\}\}$ be a set of variables and define $\mathbb{\tilde{L}}$ as the polynomial ring over $\ZZ$ generated by $A$. On this ring one defines the formal power series $\tilde{F}(u,v)=\sum_{i,j}A_{i,j} u^i v^j\in \mathbb{\tilde{L}}[[u,v]]$. The next step is to quotient $\tilde{\mathbb{L}}$ by the ideal $I$ generated by the relations obtained by forcing $\tilde{F}$ to satisfy conditions $(1), (2)$ and $(3)$ from definition  \ref{def FGL}. The quotient  ring $\mathbb{\tilde{L}}/I$, usually denoted $\Laz$, is called the {\em Lazard ring}. $\Laz$ is in fact a polynomial ring with integer coefficients on a countable set of variables $x_i$, $i\ge1$ (see for example  \cite[pp. 64-74]{StableAdams}, \cite[pp. 26-30]{FormalHazewinkel} or \cite[pp. pp. 357-360, 368-369]{ComplexRavenel})). The image of $\tilde{F}$ in $\Laz[[u,v]]$ via the quotient map $p:\tilde{\Laz}\rightarrow \Laz$ will be denoted by $F_\Laz$ and we will write $a_{i,j}$ for $p(A_{i,j})$. In order to make $\Laz$ into a graded ring, one possible choice is to assign degree $1-i-j$ to the coefficient~$a_{i,j}$. It is worth mentioning that this choice gives $\text{deg} (x_i)=-i$. We will denote this graded ring by $\Laz^*$. Another option for the grading of $\Laz$ is to set $\text{deg} (a_{i,j})=i+j-1$: we will write $\Laz_*$ for the resulting graded ring. There is a canonical choice for the variable $x_1$, namely the coefficient of $uv$ in the universal formal group law $F(u,v)$, however, the remaining variables $x_i$, $i\ge 2$ are only canonical modulo decomposable elements in the previous variables.

%The quotient ring $\mathbb{\tilde{L}}/I$, usually denoted by $\mathbb{L}$, is the Lazard ring: it was proven in \cite{Lazard} that it is a polynomial ring with integer coefficients on a countable set of variables $x_i, \ i\geq 1$. 

Let us now state the universal property of the Lazard ring.

\begin{proposition} 
$(\Laz, F_\Laz)$ is the universal commutative formal group law of rank one: for every formal group law $(R,F)$ there exist a unique ring homomorphism $\varPhi_F:\Laz\rightarrow R$ such that $\varPhi_F(F_\Laz)=F$. 
\end{proposition}

%\begin{theorem}
%The Lazard ring L is isomorphic to a polynomial ring Z[t1 , t2 , . . .], where each ti has degree 2i.
%\end{theorem}

\begin{example}
Let us consider first the additive formal group law $(R,F_a)$. The ring homomorphism $\varPhi_{F_a}$ arising from the universal property is the composition of the homomorphism $\Laz=\ZZ[\mathbf{x}]\rightarrow R[\mathbf{x}]$ (coming from the canonical morphism $\ZZ\rightarrow R$) together with the homomorphism $R[\mathbf{x}]\rightarrow R$ setting all variables equal to 0. Here by $R[\mathbf{x}]$ we mean the polynomial ring with coefficient in $R$ on the variables $x_i,\ i\geq 1$.

On the other hand, in order to obtain $\varPhi_{F_m}$ for a multiplicative formal group law $(R,F_m)$, one has to modify the second map so that  $x_1$ is mapped to $-b$.  
\end{example}

%\section{Oriented cohomology theories and oriented Borel-Moore homology theories}
\subsection{Oriented cohomology theories and oriented Borel-Moore homology theories}

In this subsection we recall the notions of oriented cohomology theory and Borel-Moore oriented homology theory. All notations and definitions are taken from \cite[Chapter 1 and 5]{AlgebraicLevine} with only minor modifications.

 We will denote by $\SCH$ the category of separated schemes of finite type over $\spec k$, with $k$ an arbitrary field. $\SM$ will then represent the full subcategory of $\SCH$ consisting of schemes smooth and quasi-projective over $\spec k$. In general by smooth morphism we will always mean smooth and quasi-projective. 
%We now reproduce the notation and the definition necessary

\begin{definition} 
Let $\mathcal{V}$ be a full subcategory of $\SCH$. $\mathcal{V}$ is said \textit{admissible} if it satisfies the following conditions

\begin{enumerate}
\item {\rm Spec}$\, k$ and the empty scheme $\emptyset$ are in $\mathcal{V}$.
\item If $Y\rightarrow X$ is a smooth quasi-projective morphism in $\SCH$ with $X\in \mathcal{V}$, then $Y\in \mathcal{V}$.
\item If $X$ and $Y$ are in $\mathcal{V}$, then so is the product $X\times_{{\rm Spec}\, k}Y$.
\item If $X$ and $Y$ are in $\mathcal{V}$, so is $X \coprod Y$. 
\end{enumerate}

\end{definition}

It follows immediately from conditions $1$ and $2$ that $\SM$ is contained in every admissible subcategory $\mathcal{V}$: ${\rm Spec}\, k$ is in $\mathcal{V}$ and for every $X\in \SM$ the structural morphism $\tau_X$ is smooth and quasi-projective. In this work $\mathcal{V}$ will mainly be either $\SCH$ or $\SM$. 

\begin{definition}
For $z\in Z\in\SM $ denote by dim$_k(Z,z)$ the dimension over $\spec k$ of the connected component of $Z$ containing $z$. 
Given an integer $d\in\ZZ$, a morphism $f:Y\rightarrow X$ in $\SM$ has \textit{relative dimension }$d$ if, for each $y\in Y$, we have dim$_k(Y, y)-$dim$_k(X,f(y))=d$.
\end{definition}

\begin{definition}
Let $f:X\rightarrow Z$, $g:Y\rightarrow Z$ be morphisms in an admissible subcategory $\mathcal{V}$ of $\SCH$. We say that $f$ and $g$ are {\rm transverse} in $\mathcal{V}$ if 
\begin{enumerate} 
\item ${\rm Tor }^{\mathcal{O}_Z}_q(\mathcal{O}_Y,\mathcal{O}_X)=0$ for all $q>0$.
\item The fiber product $X\times_Z Y$ is in $\mathcal{V}$.
\end{enumerate}
If $\mathcal{V}=\SM$ we just say that $f$ and $g$ are transverse; if $\mathcal{V}=\SCH$ we will say that $f$ and $g$ are {\rm Tor-independent}.  
\end{definition}

In the following definition $\mathbf{R}^*$ will denote the category of commutative, graded rings with unit. Let us also recall that a functor $A^*:\mathcal{V}^{op}\rightarrow \mathbf{R}^*$ is said to be additive if $A^*(\emptyset)=0$ and for any pair $(X,Y)\in \mathcal{V}^2$  the canonical ring map $A^*(X\coprod  Y)\rightarrow A^*(X)\times A^*(Y)$ is an isomorphism.

\begin{definition}
Let $\mathcal{V}$ be an admissible subcategory of $\SCH$. An oriented cohomology theory on $\mathcal{V}$ is given by 
\begin{list}{(D\arabic{enumi}).}{\usecounter{enumi}}
\item An additive functor $A^*:\mathcal{V}^{{\rm op}}\rightarrow \mathbf{R}^*$.
\item For each projective morphism $f:Y\rightarrow X$ in $\mathcal{V}$ of relative codimension $d$, a homomorphism of graded $A^*(X)$-modules:
$$f_*:A^*(Y)\rightarrow A^{*+d}(X)\ .$$
Observe that the ring homomorphism $f^*:A^*(X)\rightarrow A^*(Y)$ gives $A^*(Y)$ the structure of an $A^*(X)$-module. 

\end{list}
These satisfy
\begin{list}{(A\arabic{enumi}).}{\usecounter{enumi}}
\item One has $({\rm Id}_X)_*={\rm Id}_{A^*(X)}$ for any $X\in \mathcal{V}$. Moreover, given projective morphisms $f:Y\rightarrow X$ and $g:Z\rightarrow Y$ in $\mathcal{V}$, with $f$ of relative codimension $d$ and $g$ of relative codimension $e$, one has 
$$(f\circ g)_*=f_*\circ g_*: A^*(Z)\rightarrow A^{*+d+e}(X)\ .$$

\item Let $f:X\rightarrow Z$, $g:y\rightarrow Z$ be transverse morphisms in $\mathcal{V}$, giving the cartesian square 
$$
\xymatrix{
  W\ar[r]^{g'} \ar[d]_{f'}&X\ar[d]^{f} \\
  Y\ar[r]^{g}& Z     }
$$
Suppose that $f$ is projective of relative dimension $d$ (thus so is $f'$). Then $g^*f_*=f'_*g'^*$.
\end{list}

\begin{list}{(PB).}{}
\item Let $E\rightarrow X$ be a rank $n$ vector bundle over some $X$ in $\mathcal{V}$, $O(1)\rightarrow \Proj (E)$ the canonical quotient line bundle with zero section $s:\Proj (E)\rightarrow O(1)$. Let $1\in A^0(\Proj (E))$ denote the multiplicative unit element. Define $\xi\in A^1(\Proj (E))$ by
$$\xi:=s^*(s_*(1))\ .$$
Then $A^*(\Proj (E))$ is a free $A^*(X)$-module, with basis $(1,\xi,\tred, \xi^{n-1})$.
\end{list}

\begin{list}{(EH).}{}
\item Let $E\rightarrow X$ be a vector bundle over some $X$ in $\mathcal{V}$, and let $p: V\rightarrow X $ be an $E$-torsor. Then $p^*:A^*(X)\rightarrow A^*(V)$ is an isomorphism.

\end{list}

A morphism of oriented cohomology theories on $\mathcal{V}$ is a natural transformation of functors $\mathcal{V}^{\rm{op}}\rightarrow \mathbf{R}^*$ which commutes with the maps $f_*$.

\end{definition}

In the previous definition the abbreviations $(PB)$ and $(EH)$ stands respectively for \textit{projective bundle formula} and \textit{extended homotopy property}. The morphisms $f^*$ are called \textit{pull-backs}, while the morphisms $f_*$ are called \textit{push-forwards}. 

\begin{example} \label{ex OCT}
Two fundamental examples of oriented cohomology theories on $\SM$ are given by the Chow ring $X\mapsto CH^*(X)$ and by a graded version of the Grothendieck group of locally free coherent sheaves  $X\mapsto K^0(X)$. More precisely, in order to obtain a graded ring out of $K^0(X)$ one first considers the multiplication law given by the tensor product of sheaves and then adds a graded structure by tensoring over $\ZZ$ with the ring of Laurent polynomials $\ZZ[\beta,\beta^{-1}]$ with $\beta$ in degree -1. We will denote by $K^0[\beta,\beta^{-1}]$ the functor corresponding to the assignment $X\mapsto K^0(X)\otimes_\ZZ \ZZ[\beta,\beta^{-1}]$.

 It is important to notice that both the pull-back and push-forward maps for $K^0[\beta,\beta^{-1}]$ are defined by adding the right power of $\beta$ to the corresponding maps in $K^0$. For a smooth morphism $f:Y\rightarrow X$ one sets
$$f^*([\mathcal{E}]\cdot\beta^n)=[f^*(\mathcal{E})]\cdot \beta^n\ ,$$ 
where $\mathcal{E}$ is a locally free coherent sheaf on $X$ and $n\in \ZZ$.
In order to be able to describe the push-forwards we first need to recall that for $X\in\SM$ it is possible to identify $K^0(X)$ with the Grothendieck group of coherent sheaves $G_0(X)$. In view of this identification, for a projective morphism $f:Y\rightarrow X$ of pure codimension $d$ one can set
$$f_*([\mathcal{E}]\cdot\beta^n)= \sum_{i=0}^{\infty} (-1)^i [R^i f_*(\mathcal{E})]\cdot \beta^{n-d}\in K_0[\beta,\beta^{-1}](X)\ ,$$
where $n\in\ZZ$ and $\mathcal{E}$ is a locally free coherent sheaf on $Y$.
 %See \cite[Example 1.1.5]{AlgebraicLevine} for a more detailed account on this subject.   
%and for a projective $f:Y\rightarrow X$
 %It is also worth recalling that for $X$ smooth $K^0(X)$ can be identified with the Grothendieck group of coherent sheaves $G_0(X)$: this identification is necessary to be able to define the push-forward morphisms. See \cite[Example 1.1.5]{AlgebraicLevine} for a more detailed account on this subject.   
\end{example}

We now want to introduce the notion of oriented Borel-Moore homology theory and in order to do this we first need to recall the definitions of regular embedding and local complete intersection morphisms.

\begin{definition}
A closed immersion $i:Z\rightarrow X$ is said to be a \textit{regular embedding} if the ideal sheaf $\mathcal{I}_Z$ of $Z$ in $X$ is locally generated by a regular sequence. 
\end{definition}

\begin{definition}
A morphism $f:X\rightarrow Y$ between flat $k$-schemes of finite type is said to be a local complete intersection morphism (an \textit{l.c.i. morphism}) if it admits a factorization as $f=q\cdot i$, where $i:X\rightarrow P$ is a regular embedding and $q:P\rightarrow Y$ is a smooth, quasi-projective morphism.

We will call a scheme whose structural morphism is l.c.i. an {\rm l.c.i. scheme} and we will denote by $\LCI$ the full subcategory of $\SCH$ whose objects are l.c.i. schemes.  
\end{definition}

\begin{remark} \label{rem l.c.i. trans}
It is important to underline that both classes of morphisms are closed under composition (see \cite[Remarks 5.1.2 (2)-(3)]{AlgebraicLevine}) and to point out that given two Tor-independent morphisms $f:X\rightarrow Y$ and $g:Z\rightarrow Y$ in $\SCH$, knowing that $f$ is l.c.i. allows to conclude that also $pr_2:X\times_Y Z\rightarrow Z$ is an l.c.i. morphism.
\end{remark}

\begin{definition}
Let $\mathcal{V}$ be an admissible subcategory of $\SCH$. An \textit{ oriented Borel-Moore homology theory} on $\mathcal{V}$ is given by 
\begin{list}{(D\arabic{enumi}).}{\usecounter{enumi}}
\item An additive functor $A_*:\mathcal{V}'\rightarrow \mathbf{Ab}_*$.
\item For each l.c.i. morphism $f:Y\rightarrow X$ in $\mathcal{V}$ of relative dimension $d$, a homomorphism of graded groups:
$$f^*:A_*(X)\rightarrow A_{*+d}(Y)\ .$$
\item An element $1\in A_0(\spec k)$ and, for each pair $(X,Y)$ of objects in $\mathcal{V}$, a bilinear graded pairing 
\begin{align*}
A_*(X)\otimes A_*(Y)&\rightarrow A_*(X\times_{\spec k}Y)\\
u\otimes v \quad& \mapsto \quad u\times v
\end{align*}
called the{ \rm external product}, which is associative, commutative and admits 1 as unit element. 
\end{list}
These satisfy 
\begin{list}{(BM\arabic{enumi}).}{\usecounter{enumi}}
\item One has ${\rm Id}_X^*={\rm Id}_{A_*(X)}$ for any $X\in \mathcal{V}$. Moreover, given l.c.i. morphisms $f:Y\rightarrow X$ and $g:Z\rightarrow Y$ in $\mathcal{V}$, of pure relative dimension, one has 
$(f\circ g)_*=f_*\circ g_*$.

\item Let $f:X\rightarrow Z$, $g:y\rightarrow Z$ be transverse morphisms in $\mathcal{V}$. Suppose that $f$ is projective and that $g$ is an l.c.i. morphism, giving the cartesian square 
$$
\xymatrix{
  W\ar[r]^{g'} \ar[d]_{f'}&X\ar[d]^{f} \\
  Y\ar[r]^{g}& Z     }
$$
Note that $f'$ is projective and $g'$ is an l.c.i. morphism. Then $g^*f_*=f'_*g'^*$.

\item Let $f:X'\rightarrow X$ and $g:Y'\rightarrow Y$ be morphisms in $\mathcal{V}$. If $f$ and $g$ are projective, then for $u'\in A_*(X')$ and $v'\in A(Y')$ one has 
$$(f\times g)_*(u'\times v')=f_*(u')\times g_*(v')\ .$$
If $f$ and $g$ are l.c.i. morphisms, then for $u\in A_*(X)$ and $v\in A_*(Y)$ one has
$$(f\times g)^*(u\times v)=f^*(u)\times g^*(v)\ . $$
\end{list}

\begin{list}{(PB).}{}
\item For $L\rightarrow Y$ a line bundle on $Y\in \mathcal{V}$ with zero-section $s:Y\rightarrow L$, define the operator 
$$\widetilde{c_1}(L):A_*(Y)\rightarrow A_{*-1}(Y)$$
by $\widetilde{c_1}(\eta)=s^*(s_*(\eta))$. Let $E$ be a rank $n+1$ vector bundle on $X\in \mathcal{V}$, with projective bundle $q:\Proj(E)\rightarrow X$ and canonical quotient line bundle $O(1)\rightarrow \Proj(E)$. For $i\in \{0,\tred, n\}$, let 
$$\xi^{(i)}:A_{*+i-n}(X)\rightarrow A_*(\Proj(E))$$
be the composition of $q^*:A_{*+i-n}(X)\rightarrow A_{*+i}(\Proj(E))$ with $\widetilde{c_1}(O(1))^i:A_{*+1}(\Proj(E)\rightarrow A_*(\Proj(E)))$. Then the homomorphism 
$$\sum_{i=0}^{n}\xi^{(i)}:\bigoplus_{i=0}^n A_{*+i-n}\rightarrow A_*(\Proj(E))$$
is an isomorphism.
\end{list}

\begin{list}{(EH).}{}
\item Let $E\rightarrow X$ be a vector bundle of rank $r$ over $X\in\mathcal{V}$, and let $p: V\rightarrow X $ be an $E$-torsor. Then $p^*:A_*(X)\rightarrow A_{*+r}(V)$ is an isomorphism.
\end{list}

\begin{list}{(CD).}{}
\item For integers $r,N>0$, let $W=\Proj^N\times_{\spec k}\tred \times_{\spec k} \Proj^N$ ($r$ factors), and let $p_i:W\rightarrow \Proj^N$ be the $i$-th projection. Let $X_0, \tred, X_N$ be the standard homogeneous coordinates on $\Proj^N$, let $n_1,\tred,n_r$ be non negative integers, and let $i:Z\rightarrow W$ be the subscheme defined by $\prod_{i=1}^r p_i^*(X_N)^{n_i}=0$. Suppose that $Z$ is in $\mathcal{V}$. Then $i_*:A_*(Z)\rightarrow A_*(W)$ is injective. 

\end{list}
A morphism of oriented Borel-Moore homology theories on $\mathcal{V}$ is a natural transformation of functors $\mathcal{V}'\rightarrow \mathbf{Ab}_*$ which respects the element 1 and commutes with both the maps $f^*$ and  the external product $\times$. 

\end{definition}

\begin{example}
Two examples of oriented Borel-Moore homology theories on $\SCH$ are given by the Chow group functor $X\mapsto CH_*(X)$ and by a graded version of the Grothendieck group of coherent sheaves $X\mapsto G_0(X)$. Exactly as for the case of $K^0$ in example \ref{ex OCT}, the graded structure is added by tensoring $G_0(X)$ with $\ZZ[\beta,\beta^{-1}]$. The only difference lies in the grading of $\ZZ[\beta,\beta^{-1}]$: in this case the degree of $\beta$ is set equal to 1. We will denote the resulting functor $X\mapsto G_0\otimes_\ZZ\ZZ[\beta,\beta^{-1}]$ by $G_0[\beta,\beta^{-1}]$. For the precise details concerning the definitions of external product, push-forwards and pull-back maps see \cite[Examples 2.2.5]{AlgebraicLevine}. 
\end{example}

We now present a lemma which states a set of sufficient conditions under which axiom $(CD)$ holds. 

\begin{lemma}\label{lem CD}
Suppose to be given a functor $A_*:\SCH\rightarrow\mathbf{Ab}_*$, a family of homomorphisms $\{f^*\}$, an element $1$ and an external product $\times$ as in $(D1)-(D3)$ of the previous definition, satisfying all the axioms with the possible exception of $(CD)$. If for every closed embedding $i:Z\rightarrow X$ with complement $j:U\rightarrow X$ the sequence
$$A_*(Z)\stackrel{i_*}\longrightarrow A_*(X)\stackrel{j^*}\longrightarrow A_*(U)$$
is exact, then axiom $(CD)$ is satisfied.
\end{lemma}

\begin{proof}
See \cite[Lemmas 5.2.11 and 5.2.10]{AlgebraicLevine}
\end{proof}

%In case one decides to resctrict his attention to the case in which  $\mathcal{V}$ is $\SM$
We now want to illustrate how, provided one sets $\mathcal{V}=\SM$, it is possible to construct a functor $A^*:\SM^{op}\rightarrow \mathbf{R}^*$ out of an oriented Borel-Moore homology theory $A_*$. First of all for a pure $d$-dimensional $X\in\SM$ one sets $A^n(X):=A_{d-n}(X)$ and the definition is then extended to a general $X$ by additivity over the connected componets. On $A^*(X)$ the multiplication $\cup_X$  is defined by relying on the fact that for $X\in \SM$ the diagonal morphism $\delta_X:X\rightarrow X\times X$ is a regular embedding and hence an l.c.i morphism: for $a\in A^n(X)$ and $b\in A^m(X)$ one sets
$$a\cup_X b:=\delta_X^*(a\times b)\in\Omega^{n+m}(X) \ .$$
Since the external product is commutative and associative and, by axiom $(BM3)$, is compatible with l.c.i. pull-backs, we have that the multiplication $\cup_X$ turns $A^*(X)$ into a commutative graded ring with $\tau^*_X(1)$ as a unit. Concerning the morphisms, the first thing to note is that all morphisms between smooth schemes are l.c.i. and as a consequence for any morphism $f$ in $\SM$ one obtains a graded group homomorphism $f^*$. It is an immediate consequence of axioms $(BM1)$ and $(BM3)$ that $f^*$ is actually a graded ring homomorphism. One is finally left to verify the functoriality with respect to composition but this is granted by axiom $(BM1)$. 

One can actually say more: $A^*$ is not just a functor, it is an oriented cohomology theory. Moreover, the construction can also be reversed and from an oriented cohomology theory one can obtain an oriented Borel-Moore homology theory.  One in fact has the following result (\cite[Proposition 5.2.1]{AlgebraicLevine}), which describes the relationship between the two kinds of theories.
\begin{proposition}\label{prop OCT OBM}
Sending $A_*$ to $A^*$ as described above defines an equivalence between the category of oriented Borel-Moore homology theories on $\SM$ and the category of oriented cohomology theories on $\SM$.  
\end{proposition}

%\subsection{Fundamental classes}\label{sect classes}

\subsubsection{Fundamental classes} \label{sect classes}

The existence of a multiplicative structure in a oriented cohomology theory $A^*$ leads to the notion of the fundamental class of a scheme $X$. If one interprets the multiplication in $A^*(X)$ as an algebraic version of the geometric operation of intersecting two schemes, then the class representing the whole space has to act as a identity element. For this reason one defines the fundamental class of $X$ to be $1_X\in A^*(X)$. Given this definition, the compatibility of fundamental classes with respect to pull-back maps is an immediate consequence of the obvious observation that ring homomorphisms respect the identity element. This in particular implies that one can re-interpret the fundamental classes as pull-backs along the structural morphisms of the identity element in the coefficient ring $A^*(\spec k)$. 

The main advantage of this approach is that it can also be used in the context of oriented Borel-Moore homology theories, where the multiplicative structure is not available. Moreover, the fundamental classes defined in this way coincide, for smooth schemes, with those one obtains through proposition \ref{prop OCT OBM}: to a theory $A_*$ on some admissible subcategory $\mathcal{V}$ one associates a theory $A^*$ on $\SM$ by applying the proposition to the restriction of $A_*$ to $\SM$. Since for $X\in\SM$ the groups $A^*(X)$ and $A_*(X)$ coincide, it is possible to refer to the fundamental class of $X$ in both contexts. Let us now state the precise definitions.

\begin{definition}
Let $A^*$ be an oriented cohomology theory on an admissible subcategory $\mathcal{V}$. For $X\in \mathcal{V}$, we define the fundamental class of $X$, denoted  $[X]_{A^*}\in A^0(X)$, by setting
$$[X]_{A^*}:=\tau_X^*(1)\ ,$$
where $\tau_X$ is the structural morphism of $X$ and $1$ represents the identity element in the coefficient ring $A^*(\spec k)$. These classes are functorial with respect to pull-back morphisms: for every $f:Y\rightarrow X$ in $\mathcal{V}$ one has $f^*[X]_{A^*}=[Y]_{A^*}$.
\end{definition}

\begin{definition}\label{def fund}
Let $A_*$ be an oriented Borel-Moore homology theory on an admissible subcategory $\mathcal{V}$. For an l.c.i. scheme $X\in\mathcal{V}$, we define the fundamental class of $X$, denoted $[X]_{A_*}\in A_*(X)$ as 
$$[X]_{A_*}:=\tau_X^*(1)\ ,$$
where $\tau_X$ is the structural morphism of $X$ and $1$ represents the identity element in the coefficient ring $A_*(\spec k)$. These classes are functorial with respect to pull-back maps associated to l.c.i. morphisms: for every l.c.i. morphism $f:Y\rightarrow X$ in $\mathcal{V}$ with $Y,X\in\LCI$ one has the equality $f^*[X]_{A_*}=[Y]_{A_*}$.
\end{definition}

\begin{remark}
In both cases the compatibility between pull-back maps and fundamental classes is due to the functoriality of pull-back morphisms: $(f\circ g)^*=g^*f^*$. While for an oriented cohomology theory this descends from the fact that $A^*$ is a functor, for an oriented Borel-Moore homology theory the equality is just axiom $(BM1)$. 
\end{remark}

We now present a lemma which illustrates the compatibility between  fundamental classes and push-forward morphisms. 

\begin{lemma} \label{lem transverse}
Let $A_*$ be an oriented Borel-Moore homology theory on $\SCH$.
Let $f:X\rightarrow Y$ be a projective morphism in $\SCH$, with $X\in \LCI$ and let $g:Z\rightarrow Y$ be an l.c.i. morphism in $\SCH$ such that $f$ and $g$ are {\rm Tor}-independent. Then
\begin{enumerate}
\item $W:=Z\times_Y X$ is an l.c.i. scheme;
\item $pr_{2*}([W]_{A_*})=g^*(f_*([X]_{A_*}))\ .$
\end{enumerate}
\end{lemma}

\begin{proof}
The proof of (1) essentially follows from remark \ref{rem l.c.i. trans}. First one observes that since $f$ and $g$ are Tor-independent and $g$ is an l.c.i. one has that $pr_1: W\rightarrow X$ is l.c.i.; the statement then follows since $\tau_W=\tau_X \circ pr_1$ and l.c.i. morphisms are closed under composition. 

For (2), as we have already proven that $W\in\LCI$ and that $pr_1$ is an l.c.i. morphism, it suffices to recall the functoriality of fundamental classes with respect to l.c.i. morphisms and axiom $(BM2)$: 
$$pr_{2*}([W]_{A_*})=pr_{2*}(pr_1^*([X]_{A_*}))=g^*f_*([X]_{A_*})\ .\qedhere$$
\end{proof}

Let us now consider more in detail the definition of fundamental classes in the two most important examples of oriented Borel-Moore homology theory: the Chow group $CH_*$ and the Grothendieck group of coherent sheaves $G_0[\beta,\beta^{-1}]$. While our general definition gives us a notion of fundamental class only for l.c.i. schemes, in these two theories it is possible to extend the definition so that it includes all equi-dimensional schemes in $\SCH$. We consider first the case of the Chow group.

\begin{definition}
Let $X\in \SCH$ be an equi-dimensional scheme with irreducible components $X_1,\tred,X_n$. The Chow group fundamental class of $X$ in $CH_d(X)$ is defined as
$$[X]_{CH_*}:=\sum_{i=1}^n m_i[X_i]\ ,$$
where the coefficients $m_i$ are set equal to $l(\mathcal{O}_{X,X_i})$, the length of the local ring $\mathcal{O}_{X,X_i}$ viewed as a module over itself.   
\end{definition}

\begin{remark} \label{rem comp CH}
It is important to point out that this last definition of fundamental class is compatible with l.c.i pull-backs. To show this one first makes use of the functoriality of l.c.i. pull-back maps to reduce to two different cases: smooth morphisms and regular embeddings.  The first case follows immediately from the definition of flat pull-backs in the Chow group (see \cite[Section 1.7]{IntersectionFulton}). For what it concerns regular embeddings one has to work explicitly with the definition of the Gysin morphism. For a proof see \cite[Example 6.2.1]{IntersectionFulton}.
\end{remark}

An immediate consequence of the previous remark is that the definition we just gave for the Chow group extends the general one. It suffices to observe that the two definitions trivially agree on $\spec k$ and recall the compatibility with respect to l.c.i. pull-back morphisms to conclude that for l.c.i. schemes the two notions of fundamental class actually coincide. 

%Now that we have two different definitions of fundamental class, it is necessary to verify that they coincide on l.c.i. schemes. To achieve this, as they trivially agree on $\spec k$, it is sufficient to prove that l.c.i. pull-back maps preserves both kinds of fundamental classes. As we have already mentioned, for the general definition this property is encoded in axiom $(BM1)$. For what it concerns the other case, again thanks to the functoriality of l.c.i. pull-back maps, one can treat separatly two cases: smooth morphisms and regular embeddings. 

Let us now state the analogue of lemma \ref{lem transverse} in the case of the Chow group: in this context the result can be extended to equi-dimensional schemes. 

\begin{lemma} \label{lem transverse CH}
Let $f:X\rightarrow Y$ and $g:Z\rightarrow Y$ be Tor-independent morphisms in $\SCH$ which are respectively projective and l.c.i.. Suppose furthermore that $X$ is an equi-dimensional scheme, then one has
 $$pr_{2*}([W]_{CH_*})=g^*(f_*([X]_{CH_*}))$$ 
where $W:=Z\times_Y X$.  
\end{lemma}

\begin{proof}
Same as for lemma \ref{lem transverse}.
\end{proof}

Let us now consider the case of the Grothendieck group of coherent sheaves $G_0[\beta,\beta^{-1}]$.

\begin{definition}
Let $X\in\SCH$ be an equi-dimensional scheme. We define the fundamental class of $X$ in $G_0[\beta,\beta^{-1}](X)$ as
$$[X]_{G_0[\beta,\beta^{-1}]}:=[\mathcal{O}_X]\cdot\beta^d\ ,$$
where $d$ is the dimension of $X$.
\end{definition} 

\begin{remark} \label{rem comp G_0}
A direct application of the definition of the pull-back morphisms for $G_0$ yields the equality $f^*[\mathcal{O}_X]=[\mathcal{O}_Y]\in G_0(Y)$ for any morphism $f:X\rightarrow Y$. In particular the equality still holds if we restrict to the case of l.c.i. morphisms and we take into account the correct power of $\beta$, so to adjust to the definition in $G_0[\beta,\beta^{-1}]$. We therefore have that $[X]_{G_0[\beta,\beta^{-1}]}$ is functorial with respect to l.c.i. pull-back maps and that for l.c.i. schemes it coincides with the fundamental class arising from the general definition. 
%Let us recall that the 

%It is sufficient to apply the definition  of the pull-back maps in $G_0$ to obtain the 

%It is sufficient to consider the definition of the pull-back maps in $G_0$ to conclude that for every morphism $f:X\rightarrow Y$ one has .
%Since for every morphism $f:X\rightarrow Y$ one has the equality , 

%When applied to the class of the stru

%It follows readily 

%It is sufficient to 
%Directly from the definition of the pull-back morphisms

%By definition the pull-back maps in $G_0$ 
%It follows directly from the definition of the pull-back maps in $G_0$ that the
%class of the structural morphism is 
% class fundamental class we have just defined. In fact for a morphism %$f:X\rightarrow Y$ one has .

%If one applies the d
%Since for every morphism $f:X\rightarrow Y$ one has 

%, directly from the definition of pull-back map in $G_0$ the equality %$f^*[\mathcal{O}_X]=[\mathcal{O}_Y]\in G_0(Y)$

%It follows directly from the definition of pull-back maps in $G_0$ that they preserve the fundamental class we have just defined. In fact for a morphism $f:X\rightarrow Y$ one has $f^*[\mathcal{O}_X]=[\mathcal{O}_Y]$.

%Also in this case
\end{remark}

We complete our discussion on fundamental classes by stating the analogue of lemma \ref{lem transverse CH} for $G_0[\beta,\beta^{-1}]$.

\begin{lemma}\label{lem transverse G_0}
Let $f:X\rightarrow Y$ and $g:Z\rightarrow Y$ be Tor-independent morphism in $\SCH$ which are respectively projective and l.c.i.. Suppose furthermore that $X$ is an equi-dimensional scheme, then one has
 $$pr_{2*}([W]_{G_0[\beta,\beta^{-1}]})=g^*(f_*([X]_{G_0[\beta,\beta^{-1}]}))$$ 
where $W:=Z\times_Y X$.
\end{lemma}

\begin{proof}
Same as for lemma \ref{lem transverse}.
\end{proof}

%These classes satisfy the following properties:
%\begin{enumerate}

%\item If $X\in \SM$, then $1_X=[id_X:X\rightarrow X]\in\Omega^0(X)$ .

%\item For every cobordism cycle $[f:Y\rightarrow X]\in\Omega^*(X)$ with %$X\in\SCH$ one has $[f:Y\rightarrow X]=f_*(1_Y)$. 

%\item Let $f:Y\rightarrow X$ be an l.c.i. morphism with $Y,X\in \LCI$. Then  %$f^*(1_X)=1_Y\ .$
%\end{enumerate}

%\subsection{Chern classes and Chern class operators}

\subsubsection{Chern classes and Chern class operators}
Suppose now that $A^*$ is an oriented cohomology theory and that $E\rightarrow X$ is a vector bundle of rank $n$. To define the Chern classes of $E$ one can make use of Grothendieck's method from \cite{ChernGrothendieck}: it is a direct consequence of $(PB)$ that there exist unique elements $\alpha_i\in A^i(X)$, $i\in\{0,\tred,n-1\}$ such that
\begin{align*}
\xi^n=\sum_{i=0}^{n-1}\alpha_i \xi^i\ . 
\end{align*}
Starting from these $\alpha_i$'s one can define elements $c_i(E)\in A^i(X),\ i\in\{0,\tred, n\}$ which, provided one sets $\mathcal{V}=\SM$, enjoy the formal properties expected from Chern classes. To achieve this one sets $c_0(E)=1$ and  $c_i(E)=(-1)^{i+1} \alpha_{n-i}$ for $i\in\{1,\tred, n\}$ so that they satisfy the defining equation 
\begin{align}
\sum_{i=0}^n (-1)^i c_i(E)\xi^{n-i}=0\ .\label{eq def chern}
\end{align}
\vspace{0.2 cm}

\textbf{Notation:}  Let $E\rightarrow X$ be a vector bundle of rank $n$. From the Chern classes of $E$ one defines the \textit{Chern polynomial} by setting $c_t(E)=\sum_{i=0}^n c_i(E) t^i \in A^*(X)[t]$. We will refer to the leading coefficient of this polynomial as the \textit{top Chern class}.   

\begin{proposition}\label{prop Chern}
Let $A^*$ be an oriented cohomology theory on $\SM$. The Chern classes $\{c_i(E)\}_{0\leq i\leq n}$ satisfy the following properties:
\begin{enumerate}
\item For any line bundle $L$ over $X\in \SM$, $c_1(L)$ equals $s^*s_*(1)\in A^1(X)$, where $s:X\rightarrow L$ denotes the zero section and $1\in A^*(X)$ is the multiplicative unit element.

\item For any morphism $f:Y\rightarrow X\in \SM$, and any vector bundle $E$ over $X$, one has for each $i\geq 0$
$$c_i(f^*E)=f^*(c_i(E))\ .$$

\item (Whitney formula) Given the exact sequence of vector bundles
$$0\rightarrow E'\rightarrow E\rightarrow E''\rightarrow 0$$
then one has 
$$c_t(E)=c_t(E')c_t(E'')\ .$$
\end{enumerate}
Moreover Chern classes are characterized by these properties.
\end{proposition}

\begin{proof}
In \cite[Proposition 4.1.15]{AlgebraicLevine} one can find the proof for the case of Chern class operators $\tilde{c}_i(E)$ in an oriented Borel-Moore weak homology theory. The result then follows because every oriented cohomology theory on $\SM$ defines an oriented Borel-Moore weak homology theory on  $\SM$ and the relationship between $c_i(E)$ and $\tilde{c}_i(E)$ for a vector bundle $E\rightarrow X$ is given by the equality $c_i(E)=\tilde{c}_i(E)(1_X)$. 
\end{proof}
%(see \cite[Propositions 5.2.1 and 5.2.6]{AlgebraicLevine})
\begin{remark}
For the definition of oriented Borel-Moore weak homology theory see \cite[Definition 4.1.9]{AlgebraicLevine}. The relationship existing between these theories and oriented Borel-Moore homology theories is described by proposition 5.2.6 in \cite{AlgebraicLevine}. There it is shown that every oriented Borel-Moore homology theory on an admissible subcategory $\mathcal{V}$ defines an oriented Borel-Moore weak homology theory. As a consequence, in view of proposition \ref{prop OCT OBM} one is able to associate an oriented Borel-Moore weak homology theory to every oriented cohomology theory on $\SM$.
\end{remark}

Unlike what happens for $CH^*$, in a general oriented cohomology theory it is not always true that for two line bundles $L$ and $M$ over the same base one has 
$$c_1(L\otimes M)=c_1(L)+c_1(M)\ .$$
Instead, the relation existing between the first Chern class of a tensor product of line bundles and the first Chern class of the factors is described by means of a formal group law. More precisely, let us recall a result from \cite[Lemma 1.1.3]{AlgebraicLevine}.

\begin{lemma} \label{lem 1.1.3}
Let $A^*$ be an oriented cohomology theory on $\SM$. Then for any line bundle $L$ on $X\in \SM$ the class $c_1(L)^n$ vanishes for $n$ large enough. Moreover, there is a unique power series
$$F_A(u,v)=\sum_{i,j} a_{i,j}u^i v^j\in A[[u,v]]$$
with $a_{i,j}\in A^{1-i-j}(k)$, such that, for any $X\in \SM$ and any pair of line bundles $L,\ M$ on $X$, we have
$$F_A(c_1(L),c_1(M))=c_1(L\otimes M)\ .$$
In addition, the pair $(A^*(k),F_A)$ is a commutative formal group law of rank one. 
\end{lemma}

The fact that every oriented cohomology theory $A^*$ has an associated formal group law $(A^*(\spec k),F_A)$ also gives, by the universal property of the Lazard ring, a homomorphism $\varPhi_A:\Laz\rightarrow A^*(\spec k)$. It can be checked that this is actually a homomorphism of graded rings $\varPhi_A:\Laz^*\rightarrow A^*(\spec k)$.

\begin{example}
For $A^*=CH^*$, as it was implicitly mentioned earlier, the formal group law obtained by applying lemma 1.1.3. is the additive formal group law over $CH^*(\spec k)=\ZZ$. 

For $A^*=K^0[\beta,\beta^{-1}]$ one has $F_{K^0[\beta,\beta^{-1}]}(u,v)=u+v-\beta uv\in K^0[\beta,\beta^{-1}](\spec k)[[u,v]]=\ZZ[\beta,\beta^{-1}][[u,v]]$ and therefore $F_{K^0[\beta,\beta^{-1}]}$ is a multiplicative formal group law. 
\end{example}

Let us now consider the more general case of an oriented Borel-Moore homology theory over an admissible subcategory $\mathcal{V}$. Also in this context it is possible to define Chern classes, not in the form of actual classes but as operators. In view of axiom $(PB)$, for any vector bundle $E\rightarrow X$ of rank $n$ with $X\in \mathcal{V}$ it is possible to define the homomorphisms 
$$\widetilde{c_i}(E):\Omega_*(X)\rightarrow \Omega_{*-i}(X)$$
with $i\in\{0,\tred, n\}$ and $\widetilde{c_0}(E)=1$, as the unique solution of the equation
$$\sum_{i=0}^n (-1)^i\xi^{(n-i)}\widetilde{c_i}(E)=0\ ,$$ 
which represents the analogue of (\ref{eq def chern}). Since for line bundles we already have a notion of first Chern class operator, it is necessary to check that the two definitions actually coincide. This is in fact the case as one can verify by setting $n=0$ in axiom $(PB)$. One last point worth mentioning is associated to the relationship between the Chern classes $c_i(E)$ and the Chern class operators $\widetilde{c_i}(E)$: the link between the two notions, assuming $X$ to be a smooth scheme, is given by the formula 
\begin{align}
c_i(E)=\widetilde{c_i}(E)(1_X)\ . \label{eq chern op}
\end{align}
In view of the Whitney formula, which holds for the operators as well as for the Chern classes, one only has to consider the case of line bundles (see \cite[Proposition 5.2.4]{AlgebraicLevine}).

%\section{Some computations using Chern classes}

\subsection{Some computations using Chern classes}

In this subsection we recall some basic facts concerning Chern classes in an oriented cohomology theory $A^*$ on $\SM$. We will denote by $F$ the formal group law associated to $A^*$ and by $\chi$ its inverse. All schemes are assumed to be objects in $\SM$ with $k$ an arbitrary field.

We begin by verifying the vanishing of the first Chern class of a trivial line bundle and by relating, using the formal group law, the first Chern class of a line bundle with the one of its dual.
 
\begin{lemma}\label{lem trivial}
Let $O_X$ be the trivial line bundle over a scheme $X$. Then $c_1(O_X)=0\in A^*(X)$.
\end{lemma}

\begin{proof}
Since by property $1$ of proposition \ref{prop Chern} we have $c_1(O_X)=s^* s_*(1)$, it suffices to show that $s_*(1)=0$. To do this, one takes  a non-zero section $s'$ transverse to $s$ and considers the following cartesian square.  
$$
\xymatrix{
  \emptyset\ar[r]^{j} \ar[d]_{j}&X\ar[d]^{s} \\
  X\ar[r]^{s'}& \mathbb{A}_X^1     }
$$
Since $s$ and $s'$ are transverse in $\SM$, by $(A2)$ one has $s'^*s_*=j_*j^*$ and this last composition has to be 0 as it factors through $A^*(\emptyset)=0$. This is enough to complete the proof: it is a consequence of the extended homotopy property that $t^*$ is an isomorphism for every section $t:X\rightarrow E$ of a vector bundle $E$. In particular this applies to~$s'^*$ and we can conclude 
$$s_*(1)=(s'^*)^{-1}(j_*j^*(1))=(s'^*)^{-1}(0)=0\ . \qedhere$$ 
%We start by proving the result for the special case $X=\spec k,\ O_X=\mathbb{A}^1$. By property $1$ of proposition \ref{prop Chern} we have that $c_1(O_X)=s^* s_*(1)$, so it suffices to show that $s_*(1)=0$. To do this, one takes  a non-zero section $s'$ and considers the following cartesian square.  
%$$
%\xymatrix{
 % \emptyset\ar[r]^{j} \ar[d]_{j}&\spec k\ar[d]^{s} \\
 % \spec k\ar[r]^{s'}& \mathbb{A}^1     }
%$$
%Since $s$ and $s'$ are transverse in $\SM$ by $(A2)$ one has $s'^*s_*=j_*j^*$ and this last composition has to be 0 as it factors through $A^*(\emptyset)=0$. This is enough to complete the proof: it is a consequence of the extended homotopy property that $t^*$ is an isomorphism for every section $t:X\rightarrow E$. 
%The general case follows easily from the one we just considered. One only needs to apply property $2$ of proposition \ref{prop Chern} to the structural morphism $\pi_X$:
%$$c_1(O_X)=c_1(\pi^*(\mathbb{A}^1))=\pi^*(c_1(\mathbb{A}^1))=\pi^*(0)=0\ . \qedhere$$
\end{proof}

%It now suffices to consider the equality one obtains in degree 1: since $a_{1,0}=a_{0,1}=1$ (see remark \ref{rem obvious}) one has $c_1(O_X)=2c_1(O_X).$  

\begin{lemma} \label{lem dual}
Let $L\rightarrow X$ be a line bundle. Then 
$$c_1(L^\vee)=\chi(c_1(L))\ .$$
\end{lemma}

\begin{proof}
First one shows, using lemma \ref{lem 1.1.3} and lemma \ref{lem trivial} that  
$$F(c_1(L),c_1(L^\vee))=c_1(L\otimes L^\vee)=c_1(O_X)=0\ .$$
The needed equality is then recovered by making use of the properties of the formal group law and its inverse:
$$c_1(L^\vee)=F(c_1(L^\vee),0)=F(c_1(L^\vee),F(c_1(L),\chi(c_1(L))))=$$
$$=F(F(c_1(L^\vee),c_1(L)),\chi(c_1(L)))=F(0,\chi(c_1(L)))=\chi(c_1(L))\ .\qedhere$$
\end{proof}

The next lemma introduces the concept of \textit{Chern roots}: if a bundle $E$ is equipped with a full flag, either of quotient bundles or of subbundles, $c_t(E)$ can be factored as a product of Chern polynomials of certain line bundles one constructs using the flag (for the details see definition \ref{def linear factors}). The first Chern classes of these line bundles are called the Chern roots of $E$. More precisely, suppose we are given a full flag of quotient bundles $E\unddot$ and that we denote by $x_1,\tred,x_n$ the Chern roots associated to this flag. One then has $c_t(E)=\prod_{i=1}^n (1+x_i t)$. It follows from this factorization that the $i$-th Chern class of $E$ is the $i$-th elementary symmetric function in the Chern roots. 

\begin{remark}
It is important to point out that the standard definition of Chern classes differs from the one we just gave: in some sense ours is a restriction to bundles equipped with full flags. In the usual setting the Chern roots of a vector bundle $E\rightarrow X$ are defined as the first Chern classes of the line bundles associated to the universal full flag over $\flag(E)$, the full flag bundle of $E$. As a consequence, the Chern roots belong to $A^*(\flag(E))$ and the factorization of the Chern polynomial takes place in $A^*(\flag(E))[t]$. The link between the two different definitions is given by the universal property of $\flag(E)$: a full flag $E\unddot$ of $E$ produces a section $s_{E\unddot}: X\rightarrow \flag(E)$ whose associated pull-back morphism $s^*_{E\unddot}:A^*(\flag(E))\rightarrow A^*(X)$ maps the usual Chern roots to the ones given by our definition.
\end{remark}

\begin{lemma}\label{lem quot bundle}
Let $E\rightarrow X$ be a vector bundle of rank $n$ and let $E\unddot=(E=E_n\srarrow E_{n-1}\srarrow \tred \srarrow E_1)$ be a full flag of quotient bundles. For $i\in\{1,\tred,n\}$ set $x_i=c_1({\rm Ker}(E_i\srarrow E_{i-1}))$. Then the Chern polynomial and the top Chern class of $E$ are given by the following formulas:
$$c_t(E)=\prod_{i=1}^{n}(1+x_i t)\quad,\quad c_n(E)=\prod_{i=1}^n x_i\ .$$
In other words, the $x_i$'s form a set of Chern roots of $E$.
\end{lemma}

\begin{proof}
First of all let us observe that the formula expressing the top Chern class is a direct consequence of the one which involves the Chern polynomial: by definition the top Chern class is the leading coefficient of $c_t(E)$. It is therefore sufficient to prove the first equality.
 
The proof is done by induction on $n$ and the case $n=1$, the basis of the induction, is tautologically true.   
In order to prove the inductive step, let us consider the following short exact sequence of vector bundles: 
$$0\rightarrow {\rm Ker}(E_n\srarrow E_{n-1})\rightarrow E_n\rightarrow E_{n-1}\rightarrow 0\ .$$
We can now finish the proof by applying first the Whitney formula (proposition \ref{prop Chern}) and then the inductive hypothesis. 
$$c_t(E_n)=c_t({\rm Ker}(E_n\srarrow E_{n-1}))c_t(E_{n-1})= (1+x_n t) \prod_{i=1}^{n-1}(1+x_i t)=\prod_{i=1}^{n}(1+x_i t)\ .\qedhere$$

\end{proof}

In the next two lemmas we compute the Chern roots of a dual bundle and of a tensor product of bundles and, as a consequence, their Chern polynomials. 

\begin{lemma} \label{lem Chern dual}
Let $E\rightarrow X$ be a vector bundle of rank $n$ and let $E\unddot=(E_1\subset E_2\subset ... \subset E_n=E)$ be a full flag of subbundles. Set $y_i=c_1(E_i/E_{i-1})$ for $i\in\{1,\tred,n\}$. Then the Chern polynomial and the top Chern class are given by:
$$c_t(E^\vee)=\prod_{i=1}^{n}(1+\chi(y_i) t)\quad, \quad c_n(E^\vee)=\prod_{i=1}^{n}\chi(y_i)\ .$$
\end{lemma}
%The $\chi(y_i)$'s, therefore form a set of Chern roots of $E^\vee$. 

\begin{proof}
We begin by observing that dualizing the flag $E\unddot$ returns a full flag of quotient bundles $(E^\vee\srarrow E^\vee_{n-1}\srarrow \tred \srarrow E_1^\vee)$ and that the linear factors ${\rm Ker}(E^\vee_i\srarrow E^\vee_{i-1})$ are isomorphic to  $(E_i/E_{i-1})^\vee$. One can then finish the proof by applying lemma \ref{lem quot bundle} and \ref{lem dual}:
$$c_t(E^\vee)=\prod_{i=1}^{n}(1+c_1((E_i/E_{i-1})^\vee)t)=\prod_{i=1}^{n}(1+\chi(c_1(E_i/E_{i-1}))t)=\prod_{i=1}^{n}(1+\chi(y_i)t)\ . \qedhere$$
\end{proof}

\begin{lemma} \label{lem Chern tensor}
Let $E$ and $F$ be two vector bundles over $X$ of rank $n$ and $m$ respectively. Let $E\unddot=(E=E_n\srarrow E_{n-1}\srarrow \tred \srarrow E_1)$ and   $F\unddot=(F=F_m\srarrow F_{m-1}\srarrow \tred \srarrow F_1)$ be full flags of quotient bundles of $E$ and $F$ respectively. Set $y_j=c_1(E_j/E_{j-1})$ and  $x_i=c_1({\rm Ker}(F_i\srarrow F_{i-1}))$ for $j\in\{1,\tred, n\}$, $i\in\{1,\tred, m\}$. Then the Chern polynomial and the top Chern class of $E\otimes F$ are given by:
$$c_t(E\otimes F)=\prod_{i=1}^m\prod_{j=1}^n (1+F(x_i,y_j)t) \quad, \quad c_{nm}(E\otimes F)=\prod_{i=1}^m\prod_{j=1}^n F(x_i,y_j)\ .$$
\end{lemma}

\begin{proof}
We begin by constructing a filtration of $E\otimes F$ by means of $E\unddot$ and $F\unddot$. In order to get a filtration one first needs to tensor the linear factors $\text{Ker}(F_i\srarrow F_{i-1})$ with the filtration $E\unddot$. In this way one obtains a filtration for each $E\otimes \text{Ker}(F_i\srarrow F_{i-1})$. These filtrations are then assembled together to produce a full flag of quotient bundles of $E\otimes F$, whose linear factors are of the form $\text{Ker}(E_j\srarrow E_{j-1})\otimes \text{Ker}(F_i\srarrow F_{i-1})$. We are finally able to apply lemma \ref{lem quot bundle} and then finish the proof using lemma 1.1.3:
$$c_t(E\otimes F)=\prod_{i=1}^m\prod_{j=1}^n (1+c_1(\text{Ker}(E_j\srarrow E_{j-1})\otimes \text{Ker}(F_i\srarrow F_{i-1}))t)=\prod_{i=1}^m\prod_{j=1}^n (1+F(x_i,y_j)t)\ .\qedhere$$
\end{proof}

\begin{corollary} \label{cor Chern}
Let $E$ and $F$ be two vector bundles over $X$ respectively of rank $n$ and $m$. Let $E\unddot=(E_1\subset E_2\subset ... \subset E_n=E)$ and   $F\unddot=(F=F_m\srarrow F_{m-1}\srarrow \tred \srarrow F_1)$ be full flags of $E$ and $F$ respectively. Set $y_j=c_1(E_j/E_{j-1})$ and  $x_i=c_1({\rm Ker}(F_i\srarrow F_{i-1}))$ for $j\in\{1,\tred, n\}$, $i\in\{1,\tred, m\}$. Then the Chern polynomial and the top Chern class of $E^\vee\otimes F$ are given by:
$$c_t(E^\vee\otimes F)=\prod_{i=1}^m\prod_{j=1}^n (1+F(x_i,\chi(y_j))t) \quad, \quad c_{nm}(E^\vee\otimes F)=\prod_{i=1}^m\prod_{j=1}^n F(x_i,\chi(y_j))\ .$$
\end{corollary}

\begin{proof}
As it was noticed in the proof of lemma \ref{lem Chern dual}, the full flag of quotient bundles $(E^\vee\srarrow E^\vee_{n-1}\srarrow \tred \srarrow E_1^\vee)$ has linear factors isomorphic to $(E_i/E_{i-1})^\vee$ whose first Chern class is given by $\chi(y_i)$. One can therefore apply lemma \ref{lem Chern tensor} and finish the proof. 
\end{proof}

%We prooceed by induction on $n$ and the base of the induction $n=1$ is lemma \ref{lem dual}. To prove the inductive step one first has to notice that by dualizing $E\unddot$ one obtaines a full flag of quotient bundles $(E^\vee\srarrow E^\vee_{n-1}\srarrow \tred \srarrow E_1^\vee)$. 

%\section{Algebraic cobordism}
\subsection{Algebraic cobordism}
In this subsection we recall the definition and main properties of algebraic cobordism. Our goal is to present the material contained in \cite{AlgebraicLevine} that will be necessary for our purposes. In this subsection with the exception of 2.4.1, in which $k$ can be arbitrary, we will assume the base field to have characteristic 0.

%\subsection{The construction of $\Omega_*$}

\subsubsection{The construction of $\Omega_*$}
The first step for defining algebraic cobordism as an additive functor $\Omega^*:\SM^{{\rm op}}\rightarrow \mathbf{R}^*$, consists of constructing an additive functor $\Omega_*:\SCH'\rightarrow \mathbf{Ab}_*$. Here $\mathbf{Ab}_*$ denotes the category of abelian groups, while $\SCH'$ stands for the subcategory of $\SCH$ which has the same objects but with only projective morphisms. This functor is enriched with extra-structures: pull-backs morphisms for smooth morphisms, first Chern class operators for line bundles and an external product. Our ultimate goal will be to establish $\Omega_*$ as a oriented Borel-Moore homology theory on $\SCH$ and to use proposition \ref{prop OCT OBM} to obtain $\Omega^*$.

We begin the construction by introducing the notion of cobordism cycle. 
 
\begin{definition}
Let $X$ be a $k$-scheme of finite type. A cobordism cycle over $X$ is a family $(f:Y\rightarrow X,L_1,\tred, L_r)$ where $f:Y\rightarrow X$ is a projective morphism with $Y\in\SM$ and integral, while $(L_1,\tred, L_r)$ is a (possibly empty) finite sequence of $r$ line bundles over $Y$. The dimension of $(f:Y\rightarrow X,L_1,\tred, L_r)$ is ${\rm dim}_k(Y)-r$.
%\begin{enumerate}
%\item a projective morphism $f:Y\rightarrow X$ with $Y\in\SM$ and integral;
%\item a (possibly empty) finite sequence of $(L_1,\tred, L_r)$ of $r$ line bundles %over $Y$.
%\end{enumerate}
\end{definition}

In order to simplify the notation, whenever in a formula the number of line bundles of a cobordism cycle is clear from the context and it is not modified, we will write $\mathbf{L}$ to denote the sequence $(L_1,\tred,L_r)$.

 We now introduce the notion of isomorphism of cobordism cycles and we construct the functor $\mathcal{Z}_*:\SCH'\rightarrow \mathbf{Ab}_*$ which represents the first step towards the definition of $\Omega_*$.

\begin{definition}
An isomorphism $(\phi: Y\rightarrow Y',\sigma, (\psi_1,\tred, \psi_r))$ between the cycles $(Y\rightarrow X,L_1,\tred, L_r)$ and $(Y'\rightarrow X,L'_1,\tred, L'_{r})$ consists of  an isomorphism of $X$-schemes $\phi$, a permutation $\sigma\in S_r$ and isomorphisms of line bundles $\psi_i:L_i\cong \phi^*(L'_{\sigma(i)})$.
\end{definition}

\begin{definition}
Let $\mathcal{Z}(X)$ be the free abelian group generated by the isomorphism classes of cobordism cycles over $X$. This group can be graded by means of the dimension of cobordism cycles, giving rise to the abelian graded group $\mathcal{Z}_*(X)$. We will denote by $[f:Y\rightarrow X,L_1,\tred, L_r]$ the  image of $(f:Y\rightarrow X,L_1,\tred, L_r)$ in $\mathcal{Z}_*(X)$.  
\end{definition}

Suppose $Y\in\SM$ and denote its irreducible components by $Y_\alpha$. For a projective morphism $f:Y\rightarrow X$, we define $[Y\rightarrow X]$ to be the sum of the classes $[f\circ i_\alpha :Y_\alpha\rightarrow X]$, where $i_\alpha$ is the inclusion of $Y_\alpha$ into $Y$. In case $X\in\SM$, it is possible to consider the class $[id_X:X\rightarrow X]$ which we will denote by $1_X$. We will refer to this class as the \textit{fundamental class} of $X$.

The next series of definition describes the pull-back, push-forward and first Chern class homomorphisms for $\mathcal{Z}_*$.    
 
\begin{definition}
Let $g:X\rightarrow X'$ be a projective morphism in $\SCH$.  The push-forward along $g$ is defined as 
\begin{align*}
g_*:\mathcal{Z}_*(X) &\longrightarrow \mathcal{Z}_*(X') \\
[f:Y\rightarrow X,\boldsymbol{L}] &\longmapsto  [g\circ f : Y\rightarrow X',\boldsymbol{L}]
\end{align*}
and it is a map of graded groups.
\end{definition}

\begin{definition}
Let $g:X\rightarrow X'$ be a smooth equidimensional morphism of relative dimension $d$. The pull-back homomorphism along $g$ is defined as 
\begin{align*}
g^*:\mathcal{Z}_*(X')&\longrightarrow \mathcal{Z}_{*+d}(X)\\
[f:Y\rightarrow X,\boldsymbol{L}]&\longmapsto [p_2 :(Y\times_X X')\rightarrow X',p_1^*(\boldsymbol{L})]
\end{align*}
and it is a map of graded groups. Here by $p_1^*(\boldsymbol{L})$ we mean the sequence of line bundles one obtains by pulling back $\boldsymbol{L}$ along the first projection of $Y\times_X X'$.
\end{definition}

\begin{definition}
For $X\in\SM$ and $L$ a line bundle on $X$ let us define the first Chern class homomorphism of $L$ as the graded group homomorphism
\begin{align*}
\tilde{c}_1(L):\mathcal{Z}_*(X)&\longrightarrow \mathcal{Z}_{*-1}(X)\\
[f:Y\rightarrow X,L_1,\tred, L_r]&\longmapsto [f :Y\rightarrow X,L_1,\tred, L_r,f^*(L)]
\end{align*}
\end{definition}

On the functor $\mathcal{Z}_*$ it is also possible to define an external product.

\begin{definition}\label{def ext}
Let us denote by $\alpha$ the cycle $[f:X'\rightarrow X,L_1,\tred, L_r]\in\mathcal{Z}_*(X)$ and by $\beta$ the cycle  $[g:Y'\rightarrow Y, M_1,\tred,M_s]\in\mathcal{Z}_*(Y)$. Let $p_1^*(\boldsymbol{L})$ and $p_2^*(\boldsymbol{M})$ be the two sequences one obtains by pulling back the sequences $\boldsymbol{L}$ and $\boldsymbol{M}$ along the two projections of $X'\times Y'$. Then one sets
\begin{align*}
\times:\mathcal{Z}_*(X)\times\mathcal{Z}_*(Y)&\longrightarrow \mathcal{Z}_*(X\times Y)\\
(\alpha ,\beta)\qquad&\longmapsto [f\times g :X'\times Y'\rightarrow X\times Y,p_1^*(\boldsymbol{L}),p_2^*(\boldsymbol{M})]
\end{align*}
and $\times$ is associative and commutative.  
\end{definition}

It is important to observe that such a product gives $\mathcal{Z}_*(k)$ the structure of a commutative graded ring (the unit being $[id_{\spec k}]\in \mathcal{Z}_0(k)$) and therefore every graded group $\mathcal{Z}_*(X)$ has an $\mathcal{Z}_*(k)$-module structure. 

As a graded group, algebraic cobordism is obtained from $\mathcal{Z}_*$ by successively imposing three families of relations. These relations are such that taking the quotient with respect to them will not affect the extra-structures we have defined on the functor $\mathcal{Z}_*$. For more details see \cite[Section 2.1.5]{AlgebraicLevine}.
% In particular the quotient functor will still have an external product and this will allow us to interpret it as a functor of graded modules over the coefficient ring.

 The first family of relations forces every composition of Chern classes homorphisms to vanish once the dimension of the base scheme is exceeded. More precisely one requires algebraic cobordism to satisfy the following axiom:
 
\begin{list}{(Dim).}{}
\item  For any $Y\in \SM$ and any family $(L_1,\tred, L_n)$ of line bundles on $Y$ with $n>{\rm dim_k}(Y)$, one has 
$$\tilde{c}_1(L_1)\circ\trecd \circ\tilde{c}_1(L_n)(1_Y)=0\in\Omega_*(Y)\ .$$
\end{list}

The second family of relations establishes a link between the first Chern class homomorphism associated to a line bundle and the fundamental class of the zero-subscheme of its sections:  
\begin{list}{(Sect).}{}
\item  For any $Y\in \SM$, any line bundle $L$  on $Y$ and any section $s$ of $L$ which is transverse to the zero-section of $L$, one has
$$\tilde{c}_1(L)(1_Y)=i_*(1_Z)\ ,$$
where $i:Z\rightarrow Y$ is the closed immersion of the zero-subscheme of s.
\end{list}

The last family of relations endows $\Omega_*(\spec k)$ with a formal group law by forcing to hold the analogue of the equality in lemma \ref{lem 1.1.3}.

\begin{list}{(FGL).}{}
\item  Suppose given a fixed graded ring homomorphism $\varPhi: \Laz_*\rightarrow \Omega_*(k)$, denote by $F\in\Omega_*(k)[[u,v]]$ the image of the universal formal group law $F_\Laz\in\Laz_*[[u,v]]$ via $\varPhi$. Then for any $Y\in\SM$ and any pair $(L,M)$ of line bundles on $Y$ one has 
$$F(\tilde{c}_1(L),\tilde{c}_1(M))(1_Y)=\tilde{c}_1(L\otimes M)(1_Y)\in \Omega_*(Y)\ .$$
\end{list}

\begin{remark}It is worth noticing that the order in which this relations are imposed matters: in order for the statement of (FGL) to make sense one uses (Dim) to ensure that $F(\tilde{c}_1(L),\tilde{c}_1(M))$ is a well defined element in $\Omega_*(k)$. 
\end{remark}

Before we start imposing the relations on $\mathcal{Z}_*$ it can be helpful to say a few words about how this procedure works in general. We will use $\mathcal{Z}_*$ to examplify the procedure but the same observations will of course hold for any other functor endowed with the same structure, as the ones one builds as intermediate stages in the construction of $\Omega_*$. Suppose we are given for each $X\in\SCH$ a set of homogeneous elements $\mathcal{R}_*(X)\subset \mathcal{Z}_*(X)$. In order to ensure the compatibility of the quotient with the pull-back, push-forward and first Chern class homomorphisms, one has to define a subgroup $\langle \mathcal{R}_*\rangle (X)$ generated not just by  $\mathcal{R}_*(X)$ but by all elements of the form 
$$f_*\circ \tilde{c}_1(L_1)\circ\tred\circ\tilde{c}_1(L_r)\circ g^*(\rho)$$
with $f:Y\rightarrow X$ in $\SCH'$, $(L_1,\tred,L_r)$ a sequence of line bundles over Y, $g:Y\rightarrow Z$ smooth and equi-dimensional and $\rho \in \mathcal{R}_*(Z)$. In this way one ensures that the set of generators of the subgroup is closed under pull-backs, push-forwards and the action of first Chern classes. In this way one ensures that the quotient is still endowed with these extra-structures. One last word should be said about the external product. For the quotient to be endowed with an external product compatible with the projection map, one requires the sets $\mathcal{R}_*(Z)$ to satisfy the following condition: given elements $\rho\in\mathcal{Z}_*(X)$ and $\sigma\in\mathcal{Z}_*(Y)$ one has
\begin{align}\label{cond product}
\big( \rho\in\mathcal{R}_*(X)\vee\sigma\in\mathcal{R}_*(T) \big)\Rightarrow  \rho\times\sigma \in \mathcal{R}_*(X\times Y)\ .
\end{align}
Even though strictly speaking the expression $[f:Y\rightarrow X,L_1,\tred,L_r]$ represents an element in $\mathcal{Z}_*(X)$, we will abuse notation and we will also use it to denote its image in $\mathcal{Z}_*(X)/\langle\mathcal{R}_*(X)\rangle$ and in the successive quotients as well. In particular it will also denote an element in $\Omega_*(X)$.

Let us now see in detail how one imposes the relations (Dim), (Sect) and (FGL). For what concerns (Dim) one defines a subset $\mathcal{R}_*^{Dim}(X)\subset \mathcal{Z}_*(X)$ for every irreducible $X\in \SM$: it consists of all elements of the form 
$$[Y\rightarrow X,L_1,\tred,L_r]\ ,$$
 where ${\rm dim}_k Y<r$. The subgroup $\langle\mathcal{R}_*^{Dim}\rangle (X)$ is then explicitly described by the following result (see \cite[Lemma 2.4.2]{AlgebraicLevine}).

\begin{lemma}
Let $X$ be a finite type $k$-scheme. Then $\langle \mathcal{R}_*^{Dim}\rangle (X)$ is the subgroup of $\mathcal{Z}_*(X)$ generated by standard cobordism cycles of the form:
$$[Y\rightarrow X, \pi^*(L_1),\tred,\pi^*(L_r),M_1,\tred,M_s]\ ,$$
where $\pi:Y\rightarrow Z$ is a smooth quasi-projective equi-dimensional morphism, $Z$ is a smooth quasi-projective irreducible $k$-scheme, $(L_1,\tred,L_r)$ are line bundles on $Z$ and $r> {\rm dim}_k (Z)$.    
\end{lemma}

It is now evident from the construction that if we define  $\mathcal{\underline{Z}}_*(X):= \mathcal{Z}_*(X)/\langle \mathcal{R}_*^{Dim}\rangle (X)$, then the functor  
$\underline{\mathcal{Z}}_*:\SCH'\rightarrow \mathbf{Ab}_*$ satisfies the axiom (Dim). At this point one applies the same procedure to $\underline{\mathcal{Z}}_*$  so to make (Sect) hold. In this case for every irreducible $X\in\SM$ one defines $\mathcal{R}_*(X)$ as the subset consisting of all elements of the form 
$$\tilde{c}_1(L)-[Z\rightarrow X]\ ,$$
where $L$ is a line bundle over $X$, $s:X\rightarrow L$ is a section transverse to the zero section and $Z\rightarrow Y$ is the zero subscheme of $s$. Again one can give an explicit description of the generators of $\langle \mathcal{R}_*^{Sect}\rangle (X)$ (see \cite[Lemma 2.4.7]{AlgebraicLevine}).

\begin{lemma}
Let $X$ be a finite type $k$-scheme. Then $\langle \mathcal{R}_*^{Sect}(X) \rangle$ is the subgroup of $\mathcal{\underline{Z}}_*(X)$ generated by elements of the form;
$$[Y\rightarrow X,L_1,\tred,L_r]-[Z\rightarrow X, i^*(L_1), \tred, i^*(L_{r-1})]$$
with $r>0$ and $i:Z\rightarrow Y$ the closed immersion of the subscheme defined by the vanishing of a transverse section $s:Y\rightarrow L_r$. 
\end{lemma}
The functor $\underline{\Omega}_*:\SCH'\rightarrow \mathbf{Ab}_*$ obtained by setting $\underline{\Omega}_*(X)=\underline{\mathcal{Z}}_*(X)/\langle\mathcal{R}_*^{Sect}(X) \rangle$ is called algebraic pre-cobordism and it satisfies both (Dim) and (Sect). 

In order to complete the construction of algebraic cobordism by enforcing (FGL), one needs to have a ring homomorphism $\varPhi$ from $\Laz_*$ to the coefficient ring. For this reason one replaces $\underline{\Omega}_*$ with $\Laz_*\otimes_\ZZ \underline{\Omega}_*$ and it can be checked that this substitution preserves the validity of both (Dim) and (Sect).  Then  for $X\in\SM$ irreducible, the elements of $\mathcal{R}_*^{FGL}(X)$ are given by
$$F(\tilde{c}_1(L),\tilde{c}_1(M))(1_X)-\tilde{c}_1(L\otimes M)(1_X)$$
where $L$ and $M$ are line bundles over $X$. In view of (Dim), $F(\tilde{c}_1(L),\tilde{c}_1(M))$ is simply a polynomial in $\tilde{c}_1(L)$ and $\tilde{c}_1(M)$ and it can therefore be viewed as an endomorphism of $\Laz_*\otimes_\ZZ \underline{\Omega}_*(X)$. It is a direct consequence of the grading of $\Laz_*$ that this endomorphism  decreases the degree by 1 and this last fact implies that all the elements of $\mathcal{R}_*^{FGL}(X)$ are homogenenous: the two summands of each element have both degree  $\rm{deg}(1_X)-1$.

Unlike what was happening for the other two families, in this case one cannot use directly $\mathcal{R}_*^{FGL}$: one first has to force condition (\ref{cond product}) to hold. For this reason one replaces $\mathcal{R}_*^{FGL}$ with $\Laz_* \mathcal{R}_*^{FGL}$. For a given $X$, $\Laz_* \mathcal{R}_*^{FGL}(X)$ is the subset of $\Laz_*\otimes_\ZZ \underline{\Omega}_*(X)$ whose elements are of the form $a\otimes\rho$ with $a\in\Laz_*$ and $\rho\in\underline{\Omega}_*(X)$. Exactly as for the previous cases, it is possible to give an explicit description of the generators of $\langle\Laz_*\mathcal{R}^{FGL}\rangle(X)$ (see \cite[remark 2.4.11]{AlgebraicLevine}).

\begin{lemma}
Let $X$ be a finite type $k$-scheme. Then $\langle \Laz_*\mathcal{R}_*^{FGL}\rangle(X)$ as an $\Laz_*$-submodule of $\Laz_*\otimes_\ZZ \underline{\Omega}_*(X)$ is generated by elements of the form 
$$f_*\big(\tilde{c}_1(L_1)\tred\circ\tilde{c}_1(L_n)(\rho)\big)\ ,$$
where $f:Y\rightarrow X$ is in $\SCH'$, $L_1,\tred, L_n, L$ and $M$ are line bundles on $Y\in\SM$ and $\rho$ belongs to $\mathcal{R}_*^{FGL}$.  
\end{lemma}

We are now finally able to give the definition of algebraic cobordism.

\begin{definition}
Algebraic cobordism $\Omega_*:\SCH'\rightarrow \mathbf{Ab}_*$ is defined as the additive functor arising from the quotient of $\Laz_*\otimes_\ZZ \underline{\Omega}_*$ with respect to  $\Laz_*\mathcal{R}_*^{FGL}$, 
$$\Omega_*:=\Laz_*\otimes_\ZZ\underline{\Omega}_*/\langle\Laz_* \mathcal{R}_*^{FGL}\rangle\ .$$  
\end{definition}

As a consequence of the construction one has that $\Omega_*$ is endowed with pull-back morphisms $f^*$ for smooth morphisms, first Chern class operators $\tilde{c}_1$ for line bundles, an external product $\times$ and a graded ring homomorphism $\varPhi:\Laz_*\rightarrow\Omega_*(k)$ giving rise to a formal group law $F$. It is worth underlying that the interplay of the external product and of $\varPhi$ gives to all graded groups $\Omega_*(X)$ an $\Laz_*$-module structure. Moreover, this structure is compatible with the other operations as they all happen to be $\Laz_*$-linear. As it was mentioned earlier, we will abuse notation and interpret the cobordism cycles $[Y\rightarrow X,L_1,\tred,L_r]$ as elements of $\Omega_*(X)$.

%\subsection{The projective bundle formula and the extended homotopy  property}

\subsubsection{The projective bundle formula and the extended homotopy  property}

Before we proceed further with the construction of $\Omega_*$, let us introduce an important technical property enjoyed by algebraic cobordism: the right-exact localization sequence (see \cite[Section 3.2 and theorem 3.2.7]{AlgebraicLevine}). In this subsection, as well as in the remainder of the section, we will assume that the base field $k$ has characteristic 0. 
 
\begin{theorem}\label{th loc seq}
%Suppose that $k$ admits resolution of singularities.
 Let $X$ be a finite type $k-$scheme, $i:Z\rightarrow X$ a closed subscheme and $j:U\rightarrow X$ the open complement. Then the sequence 
$$\Omega_*(Z)\stackrel{i_*}\longrightarrow \Omega_*(X)\stackrel{j^*}\longrightarrow \Omega_*(U)\longrightarrow 0 \ ,$$
is exact.
\end{theorem}

This theorem is used to show that both the projective bundle formula and the extended homotopy property hold for $\Omega_*$. 
Let us first recall the notations necessary to express the projective bundle formula. Let $X\in \SCH$ and let $p:E\rightarrow X$ be a vector bundle of rank $n+1$. Denote by $q:\Proj (E)\rightarrow X$ the $\Proj^n$-bundle arising from $E$ and recall that this bundle is equipped with a canonical quotient line bundle $O(1)$: we will write $\xi$ for the group homomorphism
$$\widetilde{c_1}(O(1)):\Omega_*(\Proj(E))\longrightarrow\Omega_{*-1}(\Proj (E))\ . $$
In this setting we define the group homomorphism
$$\sum_{i=0}^n\xi^{(i)}:\bigoplus_{i=0}^n\Omega_{*-n+i}(X)\longrightarrow \Omega_*(\Proj(E))$$
as the sum of the family of group homomorphism $\{\xi^{(i)}\}_{i\in\{0,\tred,n\}}$ given by $\xi^{(i)}:=\widetilde{c_1}(O(1))^i\circ q^*$. 

We are now able to state both the projective bundle formula and the extended homotopy property for $\Omega_*$. For the proofs see (\cite[Theorems 3.5.4 and 3.6.3]{AlgebraicLevine}). 

\begin{theorem}\label{th proj bundle}
Let $X\in \SCH$ and let $E$ be a rank $n+1$ vector bundle on $X$. Then
$$\sum_{i=0}^n\xi^{(i)}:\bigoplus_{j=0}^n\Omega_{*-n+j}(X)\rightarrow \Omega_*(\Proj(E))$$
is an isomorphism.
\end{theorem}

\begin{theorem} \label{th ext homotopy}
 Let $E\rightarrow X$ be a vector bundle over some $X$ in $\SCH$, and let $p: V\rightarrow X $ be an $E$-torsor. Then 
$$p^*:\Omega_*(X)\rightarrow \Omega_*(V)$$ 
is an isomorphism.
\end{theorem}

%\subsection{Gysin and l.c.i. pull-back morphisms}

\subsubsection{Gysin and l.c.i. pull-back morphisms}

At this stage the only structure still missing on $\Omega_*$ is represented by the family of pull-back maps for l.c.i. morphisms: so far these maps have been defined for smooth morphisms only. 
%In order to complete the construction of $\Omega_*$ a 
%In order to complete the  from $\Omega_*$ an additive functor $\Omega^*:\SM^{\text{op}}\rightarrow \mathbf{R}^*$,  there are still two main points one has to deal with: the definition of the multiplication on $\Omega^*(X)$ for $X\in \SM$ and the extension of pull-back morphisms from smooth quasi-projective morphisms to arbitrary morphisms in $\SM$. Both goals are achieved by defining pull-backs for local complete intersections morphisms through the so called Gysin morphism for regular embeddings. 
The approach used by Levine and Morel to overcome this difficulty is essentially based on the method introduced by Fulton in \cite{IntersectionFulton}. First one deals with the intersection with Cartier divisors, which is later used, by making use of the deformation to the normal cone, to define pull-back maps for regular embeddings (i.e. the Gysin morphisms).  Finally, the case of l.c.i. morphisms is considered: they are factored into the composition of a regular embedding with a smooth morphism, as for these kinds of morphisms the pull-back map already exists.

Since a more detailed exposition of the construction of the Gysin morphism would force us to a significant detour and given that our use of it will be essentially limited to the formal properties related to functoriality, we will simply assume that Gysin morphisms can be defined and refer the interested reader to sections 6.1-6.5 in \cite{AlgebraicLevine} for a complete treatment of the subject. More specifically, for the next proposition see \cite[Proposition 6.5.4 and theorem 6.5.11]{AlgebraicLevine}.

\begin{proposition}
To every regular embedding $i:Z\rightarrow X$ it is possible to associate a graded group homomorphism $i^*:\Omega_*(X)\rightarrow \Omega_{*-d}(Z)$ called the Gysin morphism. This homomorphism satisfies the following properties.
\begin{enumerate}
\item For every morphism $f:Y\rightarrow X$ {\rm Tor}-independent to $i$ giving rise to the cartesian diagram
$$
\xymatrix{
  Z\times Y \ar[r]^{i'} \ar[d]_{f'}&Y\ar[d]^{f} \\
  Z\ar[r]^{i}& X     }
$$

\quad i) if $f$ is projective, then $i^*f_*=f'\circ i'^*$;

\vspace{0.1 cm}

\quad ii) if $f$ is smooth and quasi-projective, then $f'^*i^*=i'^*f^*$.

\item For every regular embedding $i':Z'\rightarrow Z$ one has $i'^*i^*=(i\circ  i')^*$.
\end{enumerate}

\end{proposition}

In order to define pull-backs for l.c.i. morphism we still need one more lemma (\cite[Lemma 6.5.9]{AlgebraicLevine}) to guarantee that different factorizations of the same morphism give rise to the same map.

\begin{lemma}\label{lem l.c.i. factorization}
Let  $f:X\rightarrow Y$ be an l.c.i. morphism. If we have factorizations $f=q_1\circ i_1=q_2\circ i_2$, with $i_j:X\rightarrow P_j$ regular embeddings and $q_j\rightarrow Y$ smooth and quasi-projective, then 
$$i_1^*\circ q_1^*=i_2^*\circ q_2^*\ .$$ 
\end{lemma}

Let us finally provide the definition of pull-back morphism for local complete intersection morphisms together with the results that illustrate its functoriality (\cite[Theorem 6.5.11]{AlgebraicLevine}) and its compatibility with both the external product (\cite[Theorem 6.5.13]{AlgebraicLevine}) and projective push-forwards (\cite[Proposition 6.5.12]{AlgebraicLevine}). Note in particular that these results guarantee that $\Omega_*$ satisfies axioms $(BM1)-(BM3)$.

\begin{definition}
Let $f:X\rightarrow Y$ be an l.c.i. morphism in $\SCH$ of relative dimension $d$. Define $f^*: \Omega_*(Y)\rightarrow \Omega_*(X)$ as $i^*\circ q^*$, where $f=q\circ i$ is a factorization of $f$ with $i$ a regular embedding and $q$ smooth and quasi-projective. 
\end{definition}

\begin{theorem}\label{th l.c.i. functoriality}
Let $f_1:X\rightarrow Y$, $f_2:Y\rightarrow Z$ be l.c.i. morphisms in $\SCH$. Then 
$$(f_2\circ f_1)^*=f_1^* f_2^*\ .$$
\end{theorem}

\begin{proposition}\label{prop l.c.i. ext}
Let $f_i:X_i\rightarrow Y_i$, $i=1,2$ be l.c.i. morphisms in $\SCH$. Then for $\eta_i\in\Omega_*(Y_i)$, $i=1,2$, we have 
$$(f_1\times f_2)^*(\eta_1\times\eta_2)=f_1^*(\eta_1)\times f_2^*(\eta_2)\ .$$
\end{proposition}

\begin{theorem}\label{th l.c.i. push}
Let $f:X\rightarrow Y$, $g:Z\rightarrow Y$ be {\rm Tor}-independent morphisms in $\SCH$, giving the cartesian diagram 
$$
\xymatrix{
  X\times Z \ar[r]^{f'} \ar[d]_{g'}&Z\ar[d]^{g} \\
  X\ar[r]^{f}& Y}
$$
Suppose that $f$ is an l.c.i. morphism and that $g$ is projective. Then
$$f^*g_*=g'_* f'^*\ .$$
\end{theorem}

%\subsection{Universality and fundamental classes}\label{sec universal}

\subsubsection{Universality and fundamental classes}\label{sec universal}

Now that the pull-back morphisms have been extended to l.c.i. morphisms, we are finally able to prove that $\Omega_*$ is an oriented Borel-Moore homology theory on $\SCH$.

\begin{theorem}
Algebraic cobordism $X\rightarrow \Omega_*(X)$ is an oriented Borel-Moore homology theory on $\SCH$ and it is universal among such theories: given an oriented Borel-Moore homology theory $A_*$ on $\SCH$, there exists a unique morphism of oriented Borel-Moore homology theories
$$\vartheta_{A_*}:\Omega_*\rightarrow A_*\ .$$
\end{theorem}

\begin{proof}
One first has to verify that $\Omega_*$ satisfies all the axiom of oriented Borel-Moore homology theory. As we already pointed out, axioms $(BM1)-(BM3)$ corresponds respectively to  theorems \ref{th l.c.i. functoriality}, \ref{th l.c.i. push} and proposition \ref{prop l.c.i. ext}. On the other hand axioms $(PB)$ and $(EH)$ are satisfied due to theorems \ref{th proj bundle} and \ref{th ext homotopy}. One is therefore left to verify axiom $(CD)$ which, in view of the right-exact localization sequence (theorem \ref{th loc seq}), follows from lemma \ref{lem CD}. For the universality see \cite[Theorem 7.1.3 (1)] {AlgebraicLevine}.  
\end{proof}
% define the functor $\Omega^*:\SM^{\text{op}}\rightarrow \mathbf{R}^*$. First of all for a pure $d$-dimensional $X\in\SM$ we set 
%$\Omega^n(X):=\Omega_{d-n}(X)$ and we extend to a general $X$ by additivity over the connected componets. Then one defines the multiplication $\cup_X$ by relying on the fact that for $X\in \SM$ the diagonal morphism $\delta_X:X\rightarrow X\times X$ is a regular embedding: for $a\in\Omega^n(X)$ and $b\in\Omega^m(X)$ one sets
%$$a\cup_X b:=\delta_X^*(a\times b)\in\Omega^{n+m}(X) \ .$$
%Since the external product is commutative and associative (recall definition \ref{def ext}) and it is compatible with l.c.i. pull-backs (see proposition $\ref{prop l.c.i. ext})$, we have that the multiplication $\cup_X$ turns $\Omega^*(X)$ into a commutative graded ring with $1_X$ as a unit.
% We still need to deal with the morphisms. First one observes that all morphisms between smooth schemes are l.c.i. and hence one has at least a graded group homomorphism. It is easy to see, by making use of theorem \ref{th l.c.i. functoriality} and proposition \ref{prop l.c.i. ext}, that it is actually a graded ring homomorphism. We are finally left to check the functoriality with respect to composition but for this it is sufficient to invoke once more theorem \ref{th l.c.i. functoriality}. This establishes $\Omega^*$ as a contravariant functor from $\SM$ to $\mathbf{R}^*$ and we can now show that it has the structure of an oriented cohomology theory.   

Now that we have established $\Omega_*$ as an oriented Borel-Moore homology theory, it is possible to construct a functor $\Omega^*:\SM^{op}\rightarrow \mathbf{R}^*$ which, thanks to proposition \ref{prop OCT OBM}, is an oriented cohomology theory . One can actually prove more, that $\Omega^*$ is the universal oriented cohomology theory on $\SM$.

\begin{theorem}\label{th universality}
Algebraic cobordism $X\mapsto\Omega^*(X)$ is an oriented cohomology theory on $\SM$ and it is universal among such theories: given an oriented cohomology theory $A^*$ on $\SM$, there exists a unique morphism of oriented cohomology theories
$$\vartheta_{A^*}:\Omega^*\rightarrow A^*\ .$$ 
Moreover, the classifying map $\varPhi_{\Omega}:\Laz^*\rightarrow \Omega^*(\spec k)$ associated to the formal group law $(\Omega^*(\spec k),F_\Omega)$ is an isomorphism. 
\end{theorem}

\begin{proof}
%In order show that $\Omega^*$ is an oriented cohomology one first has to prove that for every projective morphism $f:Y\rightarrow X$ the push-forward $f_*$ is actually a homomorphism of graded $\Omega^*(X)$-modules and this is essentially achieved by using theorem \ref{th l.c.i. push} (see \cite[(1) within the proof of proposition 5.2.1]{AlgebraicLevine}). At this point one is left with checking that all the axioms are satisfied: $(PB)$ and $(EH)$ follow from theorems \ref{th proj bundle} and \ref{th ext homotopy},  $(A2)$ is theorem \ref{th l.c.i. push} and $(A1)$ is simply the functoriality of push-forwards in $\Omega_*$. 
For the universality see \cite[Theorem 7.1.3 (2)]{AlgebraicLevine}. Concerning the last statement, the formal group law $(\Omega^*(\spec k),F_\Omega)$ arises from lemma \ref{lem 1.1.3}, while the isomorphism between the Lazard ring and the coefficient ring of algebraic cobordism is proven in \cite[Theorem 4.3.7]{AlgebraicLevine}.
\end{proof}

\begin{remark}
It is worth pointing out that, due to the uniqueness of lemma \ref{lem 1.1.3}, $(\Omega^*(\spec k),F_\Omega)$ has to coincide with $(\Omega^*(\spec k),F)$, the formal group law we obtained by imposing axiom (FGL) in the construction of $\Omega_*$.
\end{remark}
%Observe that if we apply lemma \ref{lem 1.1.3} to $\Omega^*$, we obtain a formal group law $(\Omega^*(\spec k),F_\Omega)$: in view of the uniqueness it has to coincide with $(\Omega^*(\spec k),F)$, the formal group law we obtained by imposing axiom (FGL) in the construction of $\Omega_*$.   

%First of all let us recall the following result on the coefficient ring of algebraic cobordism (\cite[Theorem 4.3.7]{AlgebraicLevine}).
%\begin{theorem}\label{th Lazard}
%Let $k$ be a field admitting resolution of singularities and weak factorization. Then the natural map $\varPhi:\Laz^*\rightarrow \Omega^*(\spec k)$ is an isomorphism. 
%\end{theorem}

We now want to specialize our general definition of fundamental classes for oriented Borel-Moore homology theories to the specific case of algebraic cobordism. In particular we are interested in illustrating how fundamental classes relates to cobordism cycles.   
%Our next task will be to extend the notion of fundamental class from $\SM$ to $\LCI$. Recall that for a smooth scheme $X$, the fundamental class $1_X$ is represented by the cobordism class $[id_X:X\rightarrow X]$. Now that the multiplicative structure of algebraic cobordism has been established, it is worth stressing that the identity element of $\Omega^*(X)$ is $1_X$. This fact follows from two simple observations: $[id_X:X\rightarrow X]$ is the pull-back of $[id_{\spec k}:\spec k\rightarrow \spec k]\in \Omega^0(\spec k)$ under the structural morphism $\tau_X$ and $\tau^*_X:\Omega^*(\spec k)\rightarrow \Omega^*(X)$ is a ring homomorphism.

% Even though in general for $X\in \LCI$, $\Omega_*(X)$ does not have a multiplicative structure, it is still possible to define fundamental classes by pulling back the identity element along the structural morphisms. Since from lemma \ref{lem l.c.i. factorization} we know that smooth pull-backs are just special cases of l.c.i. pull-back, in this way one obtains an extension of the notion of fundamental class for smooth schemes. Moreover, in view of theorem \ref{th l.c.i. functoriality}, we can also conclude that fundamental classes are functorial with respect to l.c.i. morphisms. We now present the definition together with the precise statements of the properties we have just outlined.

%This last theorem allows us to extend our notion of fundamental class from the category of smooth schemes $\SM$ to the category of local complete intersection schemes $\LCI$. 

\begin{definition}\label{def fund cob }
Let $X\in \LCI$. We define the fundamental class of $X$, denoted $[X]_{\Omega_*}\in \Omega_*(X)$ by setting 
$$[X]_{\Omega_*}:=\tau_X^*(1)\ ,$$
where $1$ represents the identity element in the coefficient ring $\Omega_*(\spec k)$.  

These classes satisfy the following properties:
\begin{enumerate}

\item Let $f:Y\rightarrow X$ be an l.c.i. morphism with $Y,X\in \LCI$. Then  $f^*([X]_{\Omega_*})=[Y]_{\Omega_*}\ .$

\item If $X\in \SM$, then $[X]_{\Omega_*}=1_X=[id_X:X\rightarrow X]\in\Omega^0(X)$ .

\item For every cobordism cycle $[f:Y\rightarrow X]\in\Omega^*(X)$ with $X\in\SCH$ one has $[f:Y\rightarrow X]=f_*(1_Y)$.

\end{enumerate}
\end{definition}
\begin{remark}
In the previous definition property (1) is a direct consequence of the functoriality of l.c.i. pull-back maps. For property (2) one needs only to observe that from the definition of smooth pull-backs one has the equality of cobordism cycles $[id_X:X\rightarrow X]=\tau^*_X([id_{\spec k}:\spec k\rightarrow \spec k])$. (3) follows once the push-forward map $f_*$ is applied to the equality in (2). 
\end{remark}

\begin{remark}
A question that arises quite naturally at this point is whether or not the notion of fundamental class can further be extended so as to enclose a more general family of schemes. In particular one may hope that it is possible to define on all of $\SCH$ fundamental classes which are functorial with respect to l.c.i. morphisms. A partial answer to this question was given by Levine in \cite{FundamentalLevine}. There he exhibits examples of reduced projective Cohen-Macaulay schemes for which it is not possible to define fundamental classes satisfying the required functoriality, hence ruling out the possibility of the existence of a good notion of fundamental class for the whole of $\SCH$. 
\end{remark}

We finish our general discussion on algebraic cobordism with a lemma that will play an important role in our computations: it will allow us to express the top Chern class of a bundle as a cobordism class over the base.

\begin{lemma} \label{lem top Chern}
Let $p:E\rightarrow X$ be a vector bundle of rank $d$ on $X\in\text{\textbf{Sch}}_k$.

1. Suppose that $E$ has a section $s:X\rightarrow E$ such that the zero-subscheme of $s$, $i:Z\rightarrow X$ is a regularly embedded closed subscheme of codimension $d$. Then $\widetilde{c_d}(E)=i_* i^*$.  

2. Suppose furthermore that $X,Z\in\SM$. Then $c_d(E)=[i:Z\rightarrow X]$.
\end{lemma}

\begin{proof}
For (1) see \cite[Lemma 6.6.7]{AlgebraicLevine}. For (2) one first recalls the functoriality of fundamental classes with respect to l.c.i. morphisms to obtain $$1_Z=[Z]_{\Omega_*}=i^*([X]_{\Omega_*})=i^*(1_X)\ .$$ Since,  $c_d(E)=\widetilde{c_d}(E)(1_X)$, we can apply part (1) and write
$$c_d(E)=i_* i^*(1_X)=i_*(1_Z)=[i:Z\rightarrow X]\ .\qedhere$$
\end{proof}

\subsection{Relations with other theories: $CH_*$, $G_0[\beta,\beta^{-1}]$ and $CK_*$}

We begin this subsection by illustrating how scalar extension can be used to produce new oriented oriented Borel-Moore homology theories with chosen formal group law. Afterwards we make use of this construction to describe the relations existing between algebraic cobordism and the other theories which we will consider in our study. Throughout this subsection we will again assume the base field $k$ to have characteristic 0.

\begin{definition}
Let $(R,F)$ be a commutative formal group law with $R\in\mathbf{R}^*$. We will denote by $\Omega_*^{(R,F)}$ the functor 
\begin{align*}
\SCH&\longrightarrow \quad \mathbf{Ab}_*\\
X\ \ &\mapsto \Omega_*(X)\otimes_{\Laz _*}R
\end{align*}
where the $\Laz$-module structure is given on $R$ by the ring homomorphism $\varPhi_F:\Laz^*\rightarrow R$ associated to the formal group law $F$ and on $\Omega_*(X)$ by the external product. In case the formal group law $(R,F)$ arises from an oriented Borel-Moore homology theory $A_*$, we will sometimes write $\Omega_*^{A_*}$ instead of $\Omega_*^{(R,F)}$.
\end{definition} 
% $\tau_X^*:\Laz_*\cong \Omega^*( \spec k)\rightarrow \Omega^*(X)$

It is easy to check that the functor $\Omega_*^{(R,F)}$, together with the induced external product and the obvious family of pull-back morphisms, satisfies all the axioms of an oriented Borel-Moore homology theory and that its formal group law is precisely $(R,F)$. Moreover, it follows from the universality of algebraic cobordism that $\Omega_*^{(R,F)}$ is universal among the oriented Borel-Moore homology theories which have $(R,F)$ as associated formal group law. Suppose $A_*$ to be such a theory, then for every $X\in\SCH$ one can define the bilinear morphism
\begin{align*}
 \Omega_*(X)\times R&\longrightarrow \  A_*(X)\\
(\alpha,a)\ &\longmapsto a\times (\vartheta_{A_*}(X)(\alpha))
\end{align*}
where $\times$ stands for the external product in $A_*$ and represents the scalar multiplication in the $R$-module structure on $A_*(X)$. As a consequence for every scheme $X\in\SCH$ one obtains from the universal property of tensor product a unique morphism $\Omega_*^{(R,F)}(X)\rightarrow A_*(X)$ and it is possible to check that as a whole these morphisms form a morphism of oriented Borel-Moore homology theories $\vartheta_{A_*}^{(R,F)}:\Omega_*^{(R,F)}\rightarrow A_*$. The uniqueness of $\vartheta_{A_*}^{(R,F)}$ then follows from the universal properties of tensor product and of $\Omega_*$.

This construction, in view of proposition \ref{prop OCT OBM}, has an analogue in the context of oriented cohomology theories on $\SM$: the functor $\Omega^*_{(R,F)}:=\Omega^*\otimes_{\Laz^*} R$ represents the universal oriented cohomology theory on $\SM$ with $(R,F)$ as associated formal group law. We will denote by $\vartheta_{A^*}^{(R,F)}$ the canonical map $\Omega^*_{(R,F)}\rightarrow A^*$.

\begin{remark}
It is important to point out that fundamental classes of l.c.i. schemes are preserved under morphisms of oriented cohomology theories, as well as under morphisms of oriented Borel-Moore homology theories: this follows from the compatibility of both kinds of morphism with l.c.i. pull-back maps.  
\end{remark}

\begin{remark}\label{rem double ext}
Suppose to be given a morphism of formal group laws $\phi:(R,F)\rightarrow (R',F')$. It follows immediately from the universal property of $(\Laz,F_{\Laz})$ that the unique morphisms $\phi_F$ and $\phi_{F'}$ satisfy the equality $\phi_{F'}=\phi\circ\phi_F$ and hence the two functors   $\Omega_*^{(R',F')}$ and $\Omega_*^{(R,F)}(-)\otimes_R R'$ are isomorphic. 
\end{remark}

We will now present a series of results which identify the universal oriented Borel-Moore homology theories and the universal oriented cohomology theories associated to the additive and periodic multiplicative formal group laws.  
We consider first the case of the additive formal group law $(\ZZ, F_{a})$. 

\begin{theorem} \label{th CH}
The canonical map $$\vartheta^{(\ZZ,F_{a})}_{CH_*}:\Omega^{(\ZZ,F_{a})}_*\rightarrow CH_*$$
 of oriented Borel-Moore homology functors on $\SCH$ is an isomorphism. Moreover, once it is restricted to $\SM$, it induces on the associated oriented cohomology theories the isomorphism
$$\vartheta^{(\ZZ,F_{a})}_{CH^*}:\Omega^*_{(\ZZ,F_{a})}\rightarrow CH^*\ .$$    
\end{theorem}

\begin{proof}
See \cite[Theorems 4.5.1 and 7.1.4 (2)]{AlgebraicLevine}
\end{proof}

For what it concerns the periodic multiplicative formal group law $(\ZZ[\beta,\beta^{-1}],F_m)$, Levine and Morel proved the following result.

\begin{theorem}\label{th K^0}
The canonical map 
$$\vartheta_{(\ZZ[\beta,\beta^{-1}],F_m)}^{K^0[\beta,\beta^{-1}]}:\Omega_{(\ZZ[\beta,\beta^{-1}],F_m)}^*\rightarrow K^0[\beta,\beta^{-1}]$$
is an isomorphism of oriented cohomology theories on $\SM$.
\end{theorem}

\begin{proof}
See \cite[Theorems 4.2.10 and 7.4.1 (1)]{AlgebraicLevine}.
\end{proof}

This result was later extented to the case of oriented Borel-Moore homology theories by Dai.

\begin{theorem} \label{th G_0}
The canonical map
$$\vartheta_{(\ZZ[\beta,\beta^{-1}],F_m)}^{G_0[\beta,\beta^{-1}]}:\Omega_*^{(\ZZ[\beta,\beta^{-1}],F_m)}\rightarrow G_0[\beta,\beta^{-1}]$$
is an isomorphism of Borel-Moore homology theories on $\SCH$.
\end{theorem}

\begin{proof}
See \cite[Theorem 2.2.3]{ThesisDai}
\end{proof}

%\begin{remark}
%An important consequence of theorems \ref{th}
%\end{remark}

The next example of formal group law that can be considered is the multiplicative formal group law $(\ZZ[\beta],F_m)$ which gives rise to the so-called connected $K$-theory. We will denote the resulting oriented Borel-Moore homology theory $\Omega_*^{(\ZZ[\beta],F_m)}$ by $CK_*$. Since the multiplicative formal group law can be restricted to both the additive law (by setting $\beta$ equals to 0) and the periodic multiplicative law (by setting $\beta$ equal to an invertible element), in view of remark \ref{rem double ext} one can see  that connected $K$-theory specializes to both $CH_*$ and $G_0[\beta,\beta^{-1}]$. 

\begin{corollary}\label{cor special CH}
The canonical map 
$$\vartheta_{CH_*}^{CK_*}:CK_*\otimes_{\ZZ[\beta]}\ZZ\rightarrow CH_*$$
is an isomorphism of Borel-Moore homology theories on $\SCH$. Moreover, once it is restricted to $\SM$, it induces on the associated oriented cohomology theories the isomorphism
$$\vartheta^{CK^*}_{CH^*}:CK^*\otimes_{\ZZ[\beta]}\ZZ\rightarrow CH^*\ .$$     
\end{corollary}

\begin{proof}
The statement follows from theorem \ref{th CH} and remark \ref{rem double ext}.
\end{proof}

\begin{corollary}\label{cor special G_0}
The canonical map 
$$\vartheta^{CK_*}_{G_0[\beta,\beta^{-1}]}:CK_*\otimes_{\ZZ[\beta]}\ZZ[\beta,\beta^{-1}]\rightarrow G_0[\beta,\beta^{-1}]$$
is an isomorphism of Borel-Moore homology theories on $\SCH$. Moreover, once it is restricted to $\SM$, it induces on the associated oriented cohomology theories the isomorphism 
$$\vartheta^{CK_*}_{K^0[\beta,\beta^{-1}]}: CK^*\otimes_{\ZZ[\beta]}\ZZ[\beta, \beta^{-1}]\rightarrow K^0[\beta,\beta^{-1}]\ .$$
 
\end{corollary}

\begin{proof}
The statement follows from theorems \ref{th G_0}, \ref{th K^0} and remark \ref{rem double ext}.
\end{proof}

In view of these results it seems natural to try to investigate whether or not the common properties of $CH_*$ and $G_0[\beta,\beta^{-1}]$ can be extended to $CK_*$. In particular, we have seen in subsection \ref{sect classes} that for both $CH_*$ and $G_0[\beta,\beta^{-1}]$ it is possible to extend the notion of fundamental class to all equi-dimensional schemes in $\SCH$, that this extension is functorial with respect to l.c.i. morphisms (remarks \ref{rem comp CH} and \ref{rem comp G_0}) and that it is compatible with push-forwards (lemmas \ref{lem transverse CH} and \ref{lem transverse G_0}).
One can thererefore ask the following  question.

\begin{question}
Can one extend the definition of fundamental class arising from the structure of oriented Borel-Moore homology theory on $CK_*$ to all equi-dimensional schemes in $\SCH$, so that properties $(1)-(3)$ below are satisfied?
\begin{enumerate}
\item For every l.c.i morphism $f:X\rightarrow Y$ between equi-dimensional schemes $X,Y\in\SCH$ one has
$$[X]_{CK_*}=f^*[Y]_{CK_*}\ .$$ 
\item For every pair of Tor-independent morphisms $f:X\rightarrow Y$ and $g:Z\rightarrow Y$ in $\SCH$, with $f$ projective, $g$ l.c.i. and $X$ equi-dimensional one has 
$$pr_{2*}([W]_{CK_*})=g^*(f_*([X]_{CK_*}))\ ,$$ 
where $W:=Z\times_Y X$.
\item For every equi-dimensional scheme $X\in\SCH$ one has
$$\vartheta^{CK_*}_{CH_*}([X]_{CK_*})=[X]_{CH_*}\quad ,\quad \vartheta^{CK_*}_{G_0[\beta,\beta^{-1}]}([X]_{CK_*})=[X]_{G_0[\beta,\beta^{-1}]}\ .$$
\end{enumerate}

\end{question}

\begin{remark}
While properties (1) and (2) represent the obvious analogues of the compatibilities between the fundamental classes in $CH_*$ and $G_0[\beta,\beta^{-1}]$ and the pull-back and push-forward maps, property (3) requires the extension of the fundamental class to be compatible with the specializations of corollaries \ref{cor special CH} and \ref{cor special G_0}. 
\end{remark}

%\subsection{Birational invariance for connected $K$-theory}
\subsubsection{Birational invariance for connected $K$-theory}
We end this section by presenting some consequences that can be drawn from a universal property enjoyed by connected $K$-theory. Let us first state the following theorem (\cite[Theorem 4.3.9]{AlgebraicLevine}), which illustrates the nature of the universal property. 

\begin{theorem}
 Let $k$ be a field admitting resolution of singularities and weak factorization. Then $CK_*$ is the universal oriented Borel-Moore homology theory on $\SM$ which has ``birational invariance'' in the following sense: given a birational projective morphism $f:Y\rightarrow X$  between smooth irreducible varieties, then $f_*[Y]_{CK_*}=[X]_{CK_*}$. 
\end{theorem}

In this context for a field $k$ to admit resolution of singularities will mean that the conclusion of the following theorem is valid for varieties over $k$.

\begin{theorem}
Let $k$ be a field of characteristic zero, and let $f:Y\rightarrow X$ be a rational map of reduced $k$-schemes of finite type. Then there is a projective birational morphism $\mu: Y'\rightarrow Y$ such that
\begin{enumerate}
\item $Y'$ is smooth over $k$.
\item The induced birational map $f\circ \mu : Y'\rightarrow X$ is a morphism.
\item The morphism $\mu$ can be factored as a sequence of blow-ups of Y along smooth centers lying over \rm{Sing}$f$.
\end{enumerate}
\end{theorem}

An important consequence of birational invariance is that, together with resolution of singularities, it allows to associate to every $X\in\SCH$ a unique class in $CK_*(X)$ which represents the push-forward of the fundamental class of any of the smooth schemes birationally isomorphic to $X$. Given a non-smooth integral scheme $Y\in\SCH$, one can apply resolution of singularities to $id_Y$ to obtain $r:R\rightarrow Y$ birational and projective, with $R\in\SM$. As a consequence one can consider the class $[r:R\rightarrow Y]\in\Omega^*(Y)$. In general this assignment is not well defined as there could be different resolutions of $Y$ giving rise to different cobordism classes but, as it is shown in the next proposition, all these classes have to coincide once they are mapped to connected $K$-theory.

\begin{proposition}\label{prop birational}
Let $r:R\rightarrow X$ and $r ':R'\rightarrow X$ be two projective birational morphisms. Then $\vartheta_{CK_*}([r:R\rightarrow X])=\vartheta_{CK_*}([r ': R'\rightarrow X])\in CK_*(X)$.
\end{proposition}

\begin{proof}
Let us consider the rational map $\rho:=r^{-1}\circ r': R'\rightarrow R$. Thanks to resolution of singularities there exists a projective birational morphism $\mu:R''\rightarrow R'$ with $R''\in\textbf{Sm}_{k}$ such that $\rho\circ \mu$ is a morphism. Let us observe that the birational invariance of $CK_*$ implies that $\mu_*[R'']_{CK_*}=[R']_{CK_*}$ and therefore that 
$$\vartheta_{CK_*}([r'\circ \mu:R''\rightarrow X])=
r'_*\mu_*\vartheta_{CK_*}(1_{R''})=r'_*\mu_*[R'']_{CK_*}=r'_*[R']_{CK_*}
=\vartheta_{CK_*}([r':R'\rightarrow X])\ .$$
%\vartheta_{CK_*}(r'_*[\mu:R''\rightarrow R'])=r'_*\mu_*(1_{R''})=r'_*1_{R'}=
%\vartheta_{CK_*}([r':R'\rightarrow X])\ .$$
Moreover, since the composition $r\circ(r^{-1}\circ r')\circ \mu$ is a morphism and equals $r'\circ \mu$, we also have that $[r\circ \rho\circ \mu:R''\rightarrow X]=[r'\circ \mu:R'' \rightarrow X]$ with $\rho\circ \mu$ birational and projective. It now suffices to invoke again the birational invariance of $CK^*$ to conclude that 
$$\vartheta_{CK_*}([r\circ\rho\circ \mu:R''\rightarrow X])=r_*\rho_*\mu_*[R'']_{CK_*}=r_*[R]_{CK_*}=\vartheta_{CK_*}([r:R\rightarrow X])\ . \qedhere$$
\end{proof}

Thanks to this result we are now able to associate to every integral scheme $X$ a class in $CK_*(X)$. This class will represent the push-forward of the fundamental class of any of the smooth scheme birationally isomorphic to the scheme $X$.
%We are now in the position to give a definition 

\begin{definition}
Let $Y\in\SCH$ be an integral scheme and let $r:R\rightarrow Y$ be any resolution of singularities of $Y$. We associate to $Y$ the following class in $CK_*(Y)$: $$\eta_Y:=\vartheta_{CK_*}([r:R\rightarrow Y])\ .$$ 
%Thanks to the previous proposition $1_Y$ is well defined. 
\end{definition}

\begin{remark}
It is worth noticing that if $Y$ is a smooth scheme, then one can take $id_Y$ as a resolution of singularities and therefore the class we just defined coincides with its fundamental class. 
% arising from the usual definition.
\end{remark}

%\chapter{Degeneracy loci and Schubert varieties}\label{ch 1}

\section{Degeneracy loci and Schubert varieties}\label{ch 1}
In this section we will present the geometric objects that motivate our study: degeneracy loci, Schubert varieties and Bott-Samelson resolutions. We will also illustrate the method used by Fulton in \cite{FlagsFulton} to express the fundamental classes of both degeneracy loci and Schubert varieties by means 
of certain families of polynomials.
 Throughout this section $k$ will be an arbitrary field.

%\section{Notations and definitions for the symmetric group}

\subsection{Notations and definitions for the symmetric group}

For $i\in \{1,\tred ,n-1\}$ we will denote by $s_i$ the permutation $(i\  i+1)$ and we will refer to the elements of this family as \textit{fundamental transpositions}. By decomposition of a permutation $\omega\in S_n$ we will mean an $l$-tuple $I=(i_1,\tred,i_l)$ such that $s_I:=s_{i_1}\tred s_{i_l}=\omega$. We will write $\emptyset$ to refer to the empty decomposition of the identity of $S_n$. If $I$ is an $l$-tuple, $(I,i_{l+1})$ will refer to the $(l+1)$-tuple obtained from $I$ by adding $i_{l+1}$ at the end.   

 Since the set of all elementary transpositions generates $S_n$, every $\omega$ admits a decomposition. Among the decompositions of a given element $\omega$, the ones with the fewest elementary transpositions are said to be \textit{minimal}. $l(\omega)$, the length of $\omega$, is then defined as the number of elements appearing in any minimal resolution.  

Among all elements of $S_n$ a special role is played by 
$ 
w_0=
\left( \begin{array}{cccc}
      1 & 2 & \tred & n  \\ n & n-1 &\tred & 1
\end{array} \right)
$, the permutation that achieves the maximum  of the length function $l$: $\frac{n(n-1)}{2}$.  

%\section{Degeneracy loci associated to morphisms of vector bundles}\label{sec degeneracy}

\subsection{Degeneracy loci associated to morphisms of vector bundles}\label{sec degeneracy}
 
Given a morphism between vector bundles, a degeneracy locus is a closed subscheme of the base obtained by selecting the points over which the map  induced between the fibers satisfies some requirements called $\textit{rank conditions}$. In order to be able to define the degeneracy locus both as a set and as a scheme, it is convenient to recall the notion of zero scheme of a section of a vector bundle.

\begin{definition}
Let $p:E \rightarrow X$ be a vector bundle and $s_E$ its zero section. Given a section $s:X\rightarrow E$ one defines $Z(s)$, the zero scheme of $s$, as the pull-back of $s$ along $s_E$.  Diagrammatically one has
$$
\xymatrix{
  Z(s)\ar[r]^{j=i} \ar[d]_{i}&X\ar[d]^{s} \\
  X\ar[r]^{s_E}& E     }
$$
and the fact that both $s_E$ and $s$ are sections of $E$ forces $i$ and $j$ to coincide.
\end{definition}

\begin{remark}\label{rem zeroscheme}
It is not difficult to prove, by a repeated use of the universal property of fiber products, that the construction of the zero scheme of a section commutes with pull-backs. More precizely, given a vector bundle $p:E\rightarrow X$, a section $s:X\rightarrow E$ and a morphism $\varphi:Y\rightarrow X$ one has $\varphi^{-1}(Z(s))=Z(\varphi^*s)$, where $\varphi^*s:Y\rightarrow \varphi^*E$ is the section naturally induced by $s$. 
\end{remark}

\begin{remark} \label{rem alt}
The zero scheme $Z(s)$ can be also defined in the following equivalent way. Suppose that the affine open sets $\{U_i\}_{i\in I}$ form a trivializing cover of $X$ and denote by $s_i:U_i\rightarrow \mathbb{A}_{U_i}^{{\rm rank}E}$ the restriction of $s$ to $U_i=\text{Spec}\, R_i$. Then $Z(s)\cap U_i$ is defined by the ideal $(s_{i_1},\tred,s_{i_{{\rm rank} E}})$ where the $s_{i_j}\in R_i$ are given by the different components of $s_i$. 
\end{remark}

Before considering the more general case that will be needed for our purposes, let us first define the degeneracy locus associated to a single rank condition.

\begin{definition} \label{def degen locus}
Let $E$ and $F$ be two vector bundles over $X$ of rank $e$ and $f$ respectively. Given $k\in\NN$ with $0\leq k \leq \text{min}(e,f)$ and a morphism of vector bundles $h:E\rightarrow F$, we define the $k$-th degeneracy locus 

$$D_k(h):=Z(\wedge^{k+1}h)=\{x\in X\ |\ {\rm rank}(h(x):E(x)\rightarrow F(x))\leq k\}\ ,$$ 
where $\wedge^{k+1}h$ is the morphism induced by $h$ on the $(k+1)$-th exterior powers (viewed as a section of the bundle $Hom(\wedge^{k+1}E,\wedge^{k+1}F)$) and $h(x)$ is the restriction of $h$ to the fiber over $x$.
\end{definition}  

\begin{remark}\label{rem degen locus}
Since for vector bundles the exterior power functor and the pull-back functor  commute, in view of remark \ref{rem zeroscheme}, one is able to conclude that $k$-th degeneracy loci are preserved under pull-back. In other words, with the notations of the previous definition one has $\varphi^{-1}(D_k(h))=D_k(\varphi^*h)$ for all $\varphi:Y\rightarrow X$.

\end{remark}

\begin{remark} \label{rem minors}
If one considers the alternative definition of zero scheme given in remark \ref{rem alt}, one can actually see what are the local equations defining $D_k(h)$. It is possible to show that the elements $s_{i_j}\in R_i$ are given by the  $(k+1)$-minors of the $e$ by $f$ matrix describing the morphism $h_{|U_i}:\mathbb{A}_{U_i}^e\rightarrow \mathbb{A}_{U_i}^f$.
\end{remark}

We are now in the position to generalize the previous construction to the case of a morphism of vector bundles endowed with full flags. One important feature of these kind of bundles is that they come equipped with a filtration into linear factors.
%starting from the full flag it is possible to obtain a into linear factors 
%But before we turn our attention to this matter, let us first recall how from every full flag of a given bundle it is possible to obtain a filtration with linear factors.  

\begin{definition} \label{def linear factors}
Let $V\rightarrow X$ be vector bundle of rank $n$ and let $W\unddot=(V=W_n\srarrow \trecd \srarrow W_1)$ and $U\unddot=(U_1\subset \trecd \subset U_n=V)$ be full flags of respectively quotient and subbundles of $V$. To these full flags we associate two families of $n$ line bundles $\{L^{W\unddot}_i\}_{i\in\{1,\tred,n\}}$ and $\{L^{U\unddot}_i\}_{i\in\{1,\tred,n\}}$ by setting
$$L^{W\unddot}_i:={\rm Ker\,}(W_i\srarrow W_{i-1})\quad,\quad L^{U\unddot}_i:=U_i/U_{i-1}\ .$$
\end{definition}

 Let us fix some notation. Given $h:E \rightarrow F$  a morphism of vector bundles (respectively of rank $e$ and $f$) over a scheme $X$ it is not restrictive, thanks to the splitting principle, to assume that $E$ and $F$ come equipped with full flags $E_{\textbf{\textbullet}}=(E_1\subset \trecd \subset E_e=E)$ and $F_\textbf{\textbullet}=(F=F_f\srarrow \trecd \srarrow F_1)$. We will denote by $h_{ij}$ the composition of the restriction of $h$ to $E_i$ with the projection onto $F_j$.

 In this setting a set of rank conditions is the assignment of an integer $r_{ij}$ to every map $h_{ij}$. It is therefore possible to interpret it as a function $r:\{1,\tred,e \}\times\{1,\tred, f\}\rightarrow \NN$  such that $r(i,j)=r_{ij}$.

\begin{definition}
  Let $r$ be a set of rank conditions. With the above notations the degeneracy locus of $h$ associated to $r$ is defined as
$$\Omega_r(E_{\textbf{\textbullet}},F_{\textbf{\textbullet}},h):=\bigcap_{(i,j)} D_{r_{ij}}(h_{ij})=\{x\in X \ |\ {\rm rank}(h_{ij}(x):E_i(x)\rightarrow F_j(x))\leq r(i,j)\ \forall i,j \}\ ,$$
where $h_{ij}(x)$ is the restriction of $h_{ij}$ to the fiber over $x$. In case no confusion can arise about which morphism and which flags are considered, we will write $\Omega_r$ instead of the more precise $\Omega_r(E_{\textbf{\textbullet}},F_{\textbf{\textbullet}},h)$.
\end{definition}

\begin{remark}\label{rem pull}
As scheme intersection is defined in terms of fiber products, it follows from remark \ref{rem degen locus} that also $\Omega_r(E_{\textbf{\textbullet}},F_{\textbf{\textbullet}},h)$ is preserved under pull-backs: for $\varphi:Y\rightarrow X$ one has
$\varphi^{-1}(\Omega_r(E_{\textbf{\textbullet}},F_{\textbf{\textbullet}},h))=\Omega_r(\varphi^*E_{\textbf{\textbullet}},\varphi^*F_{\textbf{\textbullet}},\varphi^* h)$.
\end{remark}

In case the two vector bundles have the same rank, it is possible to consider a family of sets of rank conditions associated to permutations.

\begin{definition} 
Suppose $e=f=n$.  Given a  permutation $\omega\in S_n$,  one defines a set of rank conditions $r_\omega$ by setting $$r_\omega (i,j)=|\{k\leq j \  |\ \omega (k)\leq i \}|\ .$$ 

\end{definition}

\begin{definition}
A set of rank conditions $r$ is said permissible if there exists $\omega \in S_n$, with $n\geq\text{max}\{e,f\}$, such that  the restriction of $r_\omega$ to $\{1,\tred ,e\}\times \{1,\tred ,f\}$ coincides with $r$. 
\end{definition}

Permissible rank conditions play an important role since, assuming $h$ generic, they give rise to degeneracy loci  which are locally irreducible. Moreover, as we will see later, if the set of rank conditions arises from a permutation, the degeneracy locus can be defined using a subset of the $n^2$ rank conditions: this leads to the notion of \textit{essential set}. 

%Moreover, if one restricts his attention to permissible rank conditions, then it is possible to symplify the set-up in two ways.

\begin{definition}
Given a permutation $\omega\in S_n$ the essential set $Ess(\omega)$ is defined as follows:
$$Ess(\omega)=\{(i,j)\in\{1,\tred,n-1\}^2\ |\ \omega(i)>j,\ 
\omega(i+1)\leq j, \ \omega^{-1}(j)>i,\  \omega^{-1}(j+1)\leq i\}\ .$$ 
\end{definition}

\begin{example} \label{ex omega0}
It is easy to verify that $Ess(\omega_0)=\{(1,n-1),(2,n-2),\tred,(n-1,1)\}$: one only has to recall that $\omega_0(i)=n+1-i$. This turns the four requirements in: 
$$n+1-i>j\quad,\quad  n-i\leq j \quad,\quad n+1-j>i \quad,\quad n-j\leq i\ .$$ Once they are combined the resulting condition is given by $i+j=n$.
\end{example}
An easy consequence of the definition is the following lemma which shows that the essential set is independent of the ambient symmetric group $\omega$ belongs to. 

\begin{lemma} \label{lem Ess emb}
Let $\omega\in S_n$ and, for $m\geq n$, let $i:S_n\rightarrow S_{m}$ be the canonical inclusion. Then $Ess(\omega)=Ess(i(\omega))$.
\end{lemma}

\begin{proof}
First of all, let us observe that it is sufficient to restrict to the case $m=n+1$: the general case immediately follows by induction. As $\omega$ and $i(\omega)$ coincides on $\{1,\tred,n-1\}^2$, the very definition of essential set implies that $Ess(\omega)=Ess(i(\omega))\cap \{1,\tred, n-1\}^2$. One is therefore left to show that in $Ess(i(\omega))$ there are no elements of the form $(k,n)$ and $(n,l)$. In order for $(k,n)$ to belong to $Ess(i(\omega))$ it should satisfy  the forth of the relations defining $Ess(i(\omega))$, which in this case gives $k\geq [i(\omega)^{-1}](n+1)=n+1$: this is impossible since   by definition $Ess(i(\omega))\subseteq \{1,\tred, n\}^2$. Similarly the second requirement forces $l\geq n+1$, thus showing that no element of the form $(n,l)$ can belong to $Ess(i(\omega))$.
\end{proof}
%the inequality $[i(\omega)^{-1}](n+1)\leq k$,

%Our interest for the notion of essential set of a permutation is mainly due to the following result and its direct consequence.  

\begin{lemma} \label{lem Ess}
For any $\omega\in S_n$ and any $n$ by $n$ matrix $M$ with entries in a commutative ring $R$, the ideal generated by all minors of size $r_\omega(i,j)+1$ taken from the upper left $i$ by $j$ corner of $M$, for all $1\leq i,j \leq n $, is generated by these same minors using only those $(i,j)$ which are in $Ess(\omega)$.
\end{lemma}

\begin{proof}
See  \cite[Lemma 3.10.a]{FlagsFulton}.
\end{proof}

\begin{proposition}\label{prop Ess}
Given a permutation $\omega\in S_n$ one has
$$\Omega_{r_\omega}(E_{\textbf{\textbullet}},F_{\textbf{\textbullet}},h)=\bigcap_{(i,j)\in Ess(\omega)} D_{r_\omega(i,j)}(h_{ij})\ .$$
\end{proposition}

\begin{proof}
Let $\{U_k\}_{k\in I}$ be an affine open cover of $X$ such that over each $U_k=\text{Spec}\, R_k$ all bundles appearing in the two flags are trivial: we will show that the scheme structures of $\Omega_{r_\omega}(E_{\textbf{\textbullet}},F_{\textbf{\textbullet}},h)$ and $\bigcap_{(i,j)\in Ess(\omega)} D_{r_\omega(i,j)}(h_{ij})$ coincide on these open sets. To do this, let us first consider the restriction of $h$ to one of these open sets: $h_{|U_k}:\mathbb{A}_{U_k}^e\rightarrow \mathbb{A}_{U_k}^f $. This morphism can be interpreted as an $e$ by $f$ matrix with entries in $R_k$, in such a way that the restriction of each morphism $h_{ij}$ is given by the upper left $i$ by $j$ corner. Recall that, as it was pointed out in remark \ref{rem minors}, each $D_{r_\omega(i,j)}(h_{ij})$ is locally defined by the vanishing of the $(r_\omega(i,j)+1)$-minors associated to ${h_{ij}}_{|U_k}$. As a consequence, lemma \ref{lem Ess} guarantees that the defining ideal of $\Omega_{r_\omega}(E_{\textbf{\textbullet}},F_{\textbf{\textbullet}},h)\cap U_k$ can be generated using only the minors coming from the rank conditions $r_\omega(i,j)$ with  $(i,j)\in Ess(w)$, thus proving the equality of the two scheme structures.  
\end{proof}

\begin{remark}
One consequence of proposition \ref{prop Ess} is that it allows to express in the form $\Omega_{r_\omega}(E'\unddot,F'\unddot,h')$ all the degeneracy loci $D_l(h)$ arising from a morphism of vector bundles $h:E\rightarrow F$, provided $E$ and $F$ are already equipped with full flags. One only needs to construct a morphism $h':E'\rightarrow F'$ and to select a permutation $\omega$ such that $Ess(\omega)=\{(i,j)\}$ and $D_{r_\omega(i,j)}(h'_{ij})=D_l(h)$.

 This can be achieved as follows. If $E$ and $F$ have rank $e$ and $f$ respectively, one sets $E':=E\oplus \mathbb{A}_X^{f-l}$, $F':=F\oplus \mathbb{A}_X^{e-l}$ and defines $h':E'\rightarrow F'$ by extending $h$ by 0 on $\mathbb{A}^{f-l}$. The flags on $E'$ and $F'$ are obtained by extending the full flags of $E$ and $F$ with trivial line bundles. For the permutation one sets
$$w=
\left( \begin{array}{ccccccccc}
1 & \tred & l & l+1&  \tred &  f & f+1& \tred & e+f-l \\
1 & \tred & l &  e+1& \tred & e+f-l & l+1 & \tred & e
\end{array} \right)\ .$$
It is easy to verify that $Ess(\omega)=\{(e,f)\}$ and that $r_\omega(e,f)=l$. Since from our construction we have $E'_e=E$, $F'_f=F$ and $h_{(ef)}=h$, we can conclude that 
$$D_l(h)=D_{r_\omega(e,f)}(h'_{ef})=\Omega_{r_\omega}(E'\unddot,F'\unddot,h')\ .$$
\end{remark}

We are now going to see how the set-up can be significantly simplified if one restrict his attention to permissible rank conditions. The first step is to show that it is sufficient to consider degeneracy loci associated to morphisms of vector bundles of the same rank. 

\begin{lemma} \label{lem same rk}
Let $r$ be a permissible set of rank conditions, $\omega\in S_n$ the corresponding permutation and $h:E\rightarrow F$ a morphism of vector bundles over $X$. Let $E\unddot$ and $F\unddot$ be full flags of $E$ and $F$ respectively. Then there exists $h':E'\rightarrow F'$ and full flags $E'_{\textbf{\textbullet}}$ and $F'_{\textbf{\textbullet}}$ such that $\Omega_r (E_{\textbf{\textbullet}},F_{\textbf{\textbullet}},h)=
\Omega_{r_\omega}(E'_{\textbf{\textbullet}},F'_{\textbf{\textbullet}},h')$
\end{lemma}

\begin{proof}
Set $E'= E\oplus \mathbb{A}_X^{n-e}$, $F'=F\oplus \mathbb{A}_X^{n-f}$ and define $h'$ by extending $h$ to $\mathbb{A}_X^{n-e}$ with the zero map. The full flags $E'$ and $F'$ are then obtained by extending the flags of $E$ and $F$ by setting $E'_{e+i}=E\oplus\mathbb{A}_X^{i}$ and $F'_{e+j}=F\oplus\mathbb{A}_X^{j}$. We now want to show that the two schemes are locally defined by the same equations. For this purpose let us now consider an affine open cover $\{U_k\}_{k\in K}$ such that over each $U_k$ all bundles appearing in $E\unddot$ and $F\unddot$ are trivial; note that this makes trivial also all the bundles in $E'\unddot$ and $F'\unddot$. If we inspect the two maps  $h_{|U_i}:\mathbb{A}_{U_k}^e\rightarrow \mathbb{A}_{U_k}^f $ and  $h'_{|U_k}:\mathbb{A}_{U_k}^n\rightarrow \mathbb{A}_{U_k}^n $ we see that $h'_{|U_k}$ can be described by an $n$ by $n$ matrix whose upper left $e$ by $f$ corner gives $h_{|U_k}$ and such that all entries outside this submatrix are 0.

 Let us now focus on the rank conditions coming from $(i,j)\in\{1,\tred, e\}\times\{1,\tred,f\}$: the equation they impose are obviously the same for both schemes since we are dealing with the exact same minors. On the other hand, the remaining rank conditions for $\Omega_{r_\omega}(E'_{\textbf{\textbullet}},F'_{\textbf{\textbullet}},h')\cap U_k$ do not provide any new equations. In fact these minors are either 0 (if one is taking the determinant of a matrix not contained in the upper corner defining $h_{|U_k}$) or already present in the list of generators of the defining ideal.
\end{proof}

 %The new rank conditions are automatically satisfied as they are implied by the old ones, which on the other hand are not modified. It therefore follows that 
%$\Omega_r (E_{\textbf{\textbullet}},F_{\textbf{\textbullet}},h)=
%\Omega_{r_\omega}(E'_{\textbf{\textbullet}},F'_{\textbf{\textbullet}},h')$.

The second step consists in reducing to the case in which the morphism $h$ is $id_V$.

\begin{lemma}
\label{lem id}
Let $h:E\rightarrow F$ be a morphism of vector bundles of rank $n$ over $X$. Let $E_{\textbf{\textbullet}}$ and $F_{\textbf{\textbullet}}$ be full flags of $E$ and $F$ respectively. Then there exists a vector bundle $V$ over $X$ with full flags $E'_{\textbf{\textbullet}}$ and $F'_{\textbf{\textbullet}}$, such that for every $\omega\in S_n$ there exists $\omega'\in S_{\text{rank } V}$ for which $\Omega_{r_\omega}(E\unddot,F\unddot,h)=\Omega_{r_{\omega '}}(E'\unddot,F'\unddot,h'=id_V)$.
\end{lemma}

\begin{proof}
One sets $V:=E\oplus F$ and makes the flags of $E$ and $F$ partial flags of $V$ by embedding $E$ into $V$ as the graph of $h$ and by projecting $V$ on $F$ by means of the second projection. One then completes the flags by setting $E'_{n+i}=E\oplus \text{Ker}(F\srarrow F_{n-i})$ and $F'_{n+i}=E/E_{n-i}\oplus F$. Finally, one sets $\omega'$ to be the image of $\omega$ in $S_{2n}$ via the canonical inclusion. In order to show that $\Omega_{r_\omega}(E\unddot,F\unddot,h)=\Omega_{r_{\omega '}}(E'\unddot,F'\unddot,h'=id_V)$, one first makes use of proposition $\ref{prop Ess}$ to write
$$\Omega_{r_\omega}(E\unddot,F\unddot,h)=\bigcap_{(i,j)\in Ess(\omega)} D_{r_\omega(i,j)}(h_{ij})\quad \text{and}\quad 
\Omega_{r_{\omega '}}(E'\unddot,F'\unddot,h')=\bigcap_{(i,j)\in Ess(\omega')} D_{r_{\omega '}(i,j)}(h'_{ij})\ .$$
One then observes that, as a consequence of the set-up, one has $h'_{ij}=h_{ij}$ for $(i,j)\in \{1,\tred,n-1 \}^2$ and therefore to finish the proof it is sufficient to show that $Ess(\omega)=Ess(\omega')$: this is granted by lemma \ref{lem Ess emb}.
\end{proof}

% It follows from the set-up that for $(i,j)\in \{1,\tred,n-1 \}^2$ one has  and therefore these maps will satisfy all the rank conditions imposed by $\omega$. Moreover, if one sets $\omega'$ to be the image of $\omega$ in $S_{2n}$ via the canonical inclusion, one has $\Omega_{r_\omega}(E,F,h)=\Omega_{r_{\omega '}}(E',F',h'=id_V)$, since all remaining conditions are redundant. 

Now that these reductions have been achieved, we will consider the case of degeneracy loci on flag bundles: this will be helpful since the results obtained in this context will later allow us to define a degeneracy class.

%\section{Schubert varieties and Bott-Samelson resolutions}

\subsection{Schubert varieties and Bott-Samelson resolutions}

 Let $p : V\rightarrow X$ be a vector bundle of rank $n$ over a  smooth scheme $X$ and let $V_{\textbf{\textbullet}}=(V_1\subset V_2\subset ... \subset V_n=V)$ be a full flag of subbundles. We will denote by $\pi:\flag (V)\rightarrow X$ the bundle of full flags of quotient bundles of $V$.

\vspace{0.3 cm}

\textbf{Notation:} By its very defining property $\flag (V)$ has a universal full flag of quotient bundles $Q_{\textbf{\textbullet}}=(\pi^*V=Q_n\srarrow Q_{n-1}\srarrow ... \srarrow Q_1)$ such that for every full flag of quotient bundles $W_{\textbf{\textbullet}}=(V=W_n\srarrow W_{n-1}\srarrow ... \srarrow W_1)$ there exist a unique section $s:X\rightarrow \flag (V)$ for which $s^*(Q\unddot)=W\unddot$. We will denote this section by~$i_{W\unddot}$.

 It is possible as well to associate a section to any full flag of subbundles $U_{\textbf{\textbullet}}=(U_1\subset U_2\subset ... \subset U_n=V)$ in a unique way: it suffices to consider $V/U\unddot=(V\srarrow V/U_1\srarrow \tred \srarrow V/U_{n-1})$.  By $i_{U\unddot}$ we will mean $i_{V/U\unddot}$.

\begin{definition}
Let $\omega\in S_n$ be a permutation. We define $\Omega_\omega$, the \textit{Schubert variety} associated to $\omega$, as the vanishing locus  $\Omega_{r_\omega}(\pi^*V_{\textbf{\textbullet}},Q_{\textbf{\textbullet}},h=id_{\pi^* V})$. 
\end{definition}

\begin{remark}
By their very definition the Schubert varieties depend on the choice of the flag $V_{\textbf{\textbullet}}$. 
\end{remark}

\begin{remark} \label{rem smooth}
In general a Schubert variety $\Omega_\omega$ needs not to be an l.c.i. scheme and, as a consequence (see subsection \ref{sec universal}), the inclusion into $\flag (V)$ will not define a class in algebraic cobordism. However, as we will see, $\Omega_{\omega_0}$ is smooth since it is possible to show that it coincides with $i_{V\unddot}(X)$.
\end{remark}

\begin{lemma} \label{lem Omega0}
The Schubert variety $\Omega_{\omega_0}$ can be described as an intersection in the following way: 
$$\Omega_{\omega_0}=\bigcap_{l=1}^{n-1} Z(h_{l,n-l})\ .$$
\end{lemma}

\begin{proof}
 In view of the definition of Schubert varieties and of proposition \ref{prop Ess} we have
$$\Omega_{\omega_0}=\Omega_{r_{\omega_0}}(\pi^* V\unddot,Q\unddot,h=id_{\pi^* V})=\bigcap_{(l,k)} D_{r_{\omega_0}(l,k)}(h_{lk})=\bigcap_{(l,k)\in Ess(\omega_0)} D_{r_{\omega_0}(l,k)}(h_{lk})=
\bigcap_{l=1}^{n-1} D_0(h_{l,{n-l}})\ .$$
The last step follows from example \ref{ex omega0}: one has $Ess(\omega_0)=\{(1,n-1),(2,n-2),\tred, (n-1,1)\}$ and it is an easy computation to check that on this set $r_{\omega_0}$ is constantly 0. To finish the proof it is now sufficient to observe that, by definition, for a section $s$ one has
$$D_0(s)=Z(s^{\wedge 1})=Z(s)\ .\qedhere$$
\end{proof}

We now want to establish a connection between Schubert varieties and vanishing loci of morphisms of vector bundles. 
%This alone will justify our interest in Schubert varieties.

\begin{lemma}\label{lem preimage}
Let $V\rightarrow X$ be a vector bundle of rank $n$, endowed with a full flag of subbundles $V\unddot$ and a full flag of quotient bundles $W\unddot$. Then $i_{W\unddot}^{-1}(\Omega_\omega)=\Omega_{r_\omega}(V\unddot,W\unddot,id_V)$ for every $\omega\in S_n$.
\end{lemma}

\begin{proof}
In this context $\Omega_\omega$ corresponds to $\Omega_{r_\omega}(\pi^*V_{\textbf{\textbullet}},Q\unddot,id_{\pi^* V})$ and therefore the proposition is a consequence of the fact, pointed out in remark \ref{rem pull} that the construction of $\Omega_{r_\omega}$ is preserved under pull-backs:  $i_{W\unddot}^{-1}(\Omega_{r_\omega}(\pi^*V\unddot,Q\unddot,id_{\pi^* V})$ coincides with $\Omega_{r_\omega}(i_{W\unddot}^*(\pi^*V\unddot),i_{W\unddot}^* Q\unddot,id_{i_{W\unddot}^* \pi^* V})=\Omega_{r_\omega}(V\unddot,W\unddot,id_V)$.
\end{proof}

\begin{proposition} 
Let $\Omega_r(E_{\textbf{\textbullet}},F_{\textbf{\textbullet}},h)\subseteq X$ be the vanishing locus associated to a permissible set of rank conditions $r$ and to a morphism of vector bundles $h:E\rightarrow F$. Then there exist a vector bundle $V\rightarrow X$ with a full flag of subbundles $V_{\textbf{\textbullet}}$, together with a section $s:X\rightarrow\flag (V)$ and a permutation $\omega\in S_{\text{rank} \ V} $ such that $s^{-1}(\Omega_\omega)=\Omega_r(E_{\textbf{\textbullet}},F_{\textbf{\textbullet}},h)$.

\end{proposition}

\begin{proof}
Thanks to  lemmas \ref{lem same rk} and \ref{lem id}, it is possible to reduce to the case in which $E=F=V$, $h=id_V$ and $r=r_\omega$ for some $\omega\in S_{\text{rank }V}$: this is precisely the content of lemma \ref{lem preimage}.
\end{proof}

The local properties of Schubert varieties can be deduced from the special case in which the base scheme is a point. If one sets $X=\spec k$, $V$ becomes an affine space $\mathbb{A}^n$ while $\flag (V)$ turns into the flag manifold $\flag (n)$, which has dimension $\frac{n(n-1)}{2}$.
Let us recall the following properties of Schubert varieties in a flag manifold. 

\begin{proposition}\label{prop schubert}
Let $X=\spec k$. For any $\omega\in S_n$ the Schubert variety  $\Omega_\omega$ is integral, Cohen-Macaulay and has codimension $l(\omega)$. 
\end{proposition}

\begin{proof}
See \cite[Lemma 6.1 (a),(c),(d)]{FlagsFulton}.
\end{proof}

\begin{remark}\label{rem point}
It is not difficult to see that in $\flag (n)$ the Schubert variety $\Omega_{\omega_0}$ is just the $k$-point $i_{V\unddot}:\spec k\rightarrow \flag (n)$, describing the full quotient flag $V/V\unddot$. In view of lemma \ref{lem Omega0} one has
$$i_{V\unddot } ^{-1}(\Omega_{\omega_0})=i_{V\unddot}^{-1}(\bigcap_{l=1}^{n-1} Z(h_{l,n-l}))=\bigcap_{l=1}^{n-1} i^{-1}_{V\unddot} (Z(h_{l,n-l}))=\bigcap_{l=1}^{n-1} Z(s^*(h_{l,n-l}))$$
where the morphisms $s^*(h_{l,n-l})$ are nothing but the zero maps $V_l\rightarrow V\srarrow V/V_l$: it follows that all the zero schemes $Z(s^*(h_{l,n-l}))$ actually coincide with $\spec k$.
$$
\xymatrix{
\spec k  \ar[r]^{\varphi} \ar[d]_{id_{\spec k}}&\Omega_{\omega_0}\ar[d] \\
  \spec k \ar[r]^{i_{V\unddot}}& \flag(n)     }
$$
 Therefore, since $\Omega_{\omega_0}$ is integral and of dimension 0, we have that the closed imbedding $i_{V\unddot } ^{-1}(\Omega_{\omega_0})=\spec k\rightarrow\Omega_{\omega_0}$ is actually an isomorphism.  
\end{remark}

This last observation can be used to obtain a generalization for a general base scheme $X$.
  
\begin{lemma}\label{lem regular Schubert}
The section $i_{V\unddot}: X\rightarrow \flag(V)$ maps $X$ isomorphically onto the Schubert variety $\Omega_{\omega_0}$ and is a regular embedding of codimension $\frac{n(n-1)}{2}$. As $X\in\SM$ this implies $\Omega_{\omega_0}\in\SM$.
\end{lemma}

\begin{proof}
In view of the fact that both the flag bundle and the Schubert varieties are preserved under pull-backs, we can check the statement locally. Let us consider an open subset $j:U\rightarrow X$ over which all the bundles in $V\unddot$ are trivial. In other words we have that the full flag bundle $j^*(V\unddot)$ is nothing but the pull-back of a full flag of $\mathbb{A}^n$ via $\tau_U$, the structural morphism of $U$. In particular this implies that $\flag(j^*V)=\flag(n)\times_{\spec k}U$, as one can check that each scheme satisfies the universal properties of the other. A further consequence is that the universal full quotient flag over $\flag(n)$ is pulled-back to the one over $\flag(j^*V)$. 
 Therefore, thanks to remarks \ref{rem pull} and \ref{rem point} we have  
$$\Omega_{\omega_0}=\Omega_{r_{\omega_0}}(j^* (V\unddot),\tau_U^*(Q\unddot),id_{j^* V})=
\tau_U^{-1}(\Omega_{r_{\omega_0}}(\mathbb{A}^n_{\flag(n) \textbf{\textbullet}}, Q\unddot,id_{\mathbb{A}^n_{\flag(n)}}))=\tau_U^{-1}(\Omega_{\omega_0})=\tau_U^{-1}(\spec k)=U\ .$$
Moreover, if one goes through all the equalities, one sees that the isomorphism between $\Omega_{\omega_0}$ and $U$ is given, exactly as it was happening in remark \ref{rem point}, by factoring $i_{j^*(V\unddot)}$ through $\Omega_{\omega_0}$. This happens precisely because the diagram for the case of $\spec k$ pulls-back to 
$$
\xymatrix{
  U \ar[r]^{\varphi} \ar[d]_{id_U}&\Omega_{\omega_0}\ar[d] \\
  U\ar[r]^{i_{j^*(V\unddot)}}& \ \flag(j^* V)     }
$$
and $\varphi$ is the pull-back of the isomorphism between $\spec k$ and the Schubert variety inside of $\flag (n)$. We are only left to show that $i_{V\unddot}$ is regular embedding but this follows from the fact that $i_{V\unddot}$ is a section of the smooth morphism $\pi:\flag(V)\rightarrow X$. 
\end{proof}

The last lemma provides the starting point for the construction of a family of schemes over $\flag (V)$, the so-called Bott-Samelson resolutions, which will allow us to overcome the difficulty outlined in remark \ref{rem smooth}. Each of the members of this family will be smooth over $k$ and will map birationally onto a Schubert variety. Even though this assignment is not unique (the same Schubert variety can be associated to many Bott-Samelson resolutions), it will let us associate algebraic cobordism classes to each Schubert variety.

 To be able to define Bott-Samelson resolutions we first need to introduce a family of flag bundles over $X$. Let $Y_i\rightarrow X$ be the bundle parametrizing the flag bundles one obtains when the rank $i$ bundle is removed from a complete flag. If we denote by $(Q_n\srarrow ...\srarrow Q_{i+1} \srarrow \widehat{Q_{i}} \srarrow Q_{i-1} \srarrow... \srarrow Q_1)$ the universal flag over $Y_i$, then $\flag (V)=\mathbb{P}_{Y_i}(\text{Ker}( Q_{i+1}\srarrow Q_{i-1}))$. 

\begin{remark}\label{rem P^1 bundle}
It is important to stress that this last observation shows that $\varphi_i:\flag (V)\rightarrow Y_{i}$ is a $\mathbb{P}^1$-bundle.
\end{remark}

We are now ready to define Bott-Samelson resolutions. As it has been mentioned, there can be more resolutions associated to the same Schubert variety; this is reflected by the fact that Bott-Samelson resolutions are not indexed by permutations but by decompositions of permutations. In other words we will associate a scheme $r_I:R_I\rightarrow \flag (V)$ to every $l$-tuple $I$. The definition is done recursively on the size of $I$. 

\begin{definition}\label{def Bott-Samelson}
Let $I$ be the $l$-tuple $(i_1,i_2,\tred,i_l)$ with $i_k \in \{1,\tred,n-1 \}$.

 If $l=0$, then  $I=\emptyset$ and one sets $R_\emptyset:=X$, $r_\emptyset=i_{V_{\textbf{\textbullet}}}$. 

If $l>0$, then it is possible to write $I=(I',j)$ and, thanks to  the inductive hypothesis, $r_{I'}:R_{I'}\rightarrow \flag (V)$ has already been defined. 
One then can consider the following fiber diagram

\begin{eqnarray}\label{diagram flag}
\xymatrix{
  R_{I'}\times_{Y_{j}} \flag(V) \ar[rr]^{pr_2} \ar[d]_{pr_1}& &\flag(V)\ar[d]^{\varphi_j} \\
  R_{I'}\ar[r]^{r_{I'}}& \flag(V) \ar[r]^{\varphi_j} &Y_j     }
\end{eqnarray}
and set $R_I:=R_{I'}\times_{Y_j} \flag(V)$ and $r_I:=pr_2$.
\end{definition}

\vspace{0.2 cm}

\begin{remark} \label{rem smooth morph}
Since $\varphi_i$ is a smooth morphism, then the projection on the first factor $R_I\rightarrow R_{I'}$ has to be smooth as well. This fact, together with our assumption of $X$ being a smooth scheme over $k$, proves by induction that $R_I\in\SM$.
 %By induction on the length of $I$ this fact, together with our assumption of $X$ being a smooth scheme over $k$, ensures that $R_I\in\SM$ for any choice of $I$.
% Since $i_{V_{\textbf{\textbullet}}}$ is a smooth morphism, $r_I: R_I\rightarrow\flag (V)$ is also smooth for every choice of $I$.  
\end{remark}

The relationship existing between Bott-Samelson resolutions and Schubert varieties is made explicit by the following results.

\begin{proposition}\label{prop resolution}

Let $I=(i_1,\tred, i_l)$ be a minimal decomposition and set $\omega=\omega_0 s_I$. Then 

1) $r_I(R_I)=\Omega_{\omega}$ and  the resulting map $R_I\rightarrow \Omega_{\omega}$ is a projective birational morphism. $R_I$ is therefore a resolution of singularities of $\Omega_{\omega}$;

2) i) $r_{I*}\mathcal{O}_{R_I}=\mathcal{O}_{\Omega_\omega}$  as coherent sheaves and therefore $\Omega_\omega$ is a normal scheme;

\ \ \ ii) $R^q f_*\mathcal{O}_{R_I}=0$ for q>0, hence $\Omega_\omega$ has at worst rational singularities. 
\end{proposition}

\begin{proof}
For part (1) see \cite[Appendix C]{SchubertFulton}. For part (2) see \cite[Theorem 4]{SchubertRamanathan}.
 %Fulton-Pragacz, \textit{Schubert Varieties and Degeneracy Loci},  appendix C.
\end{proof}

\begin{remark}
The importance of the previous proposition is better understood when one relates it to the push-forward morphisms of $CH_*$ and $G_0$: it guarantees that in both theories the push-forward morphisms maps the fundamental class of $R_I$ to the one of $\Omega_\omega$.  
\end{remark}

\begin{remark}
If $I$ is a minimal decomposition, then its size $l$ describes the relative dimension of the associated Schubert variety $\Omega_{\omega_0 s_I}$ as a scheme over $X$. This can be easily seen for $X=\spec k$, from which the general case is derived. If $X=\spec k$, $l$ actually describes the dimension of $\Omega_{\omega_0 s_I}$: since $I$ is minimal,  one has 
$$l(\omega_0\cdot s_I)=l(\omega_0)-l(s_I)=\frac{n(n-1)}{2}-l$$  and therefore $l=\frac{n(n-1)}{2}-l(\omega_0 \cdot s_I)$. In view of proposition \ref{prop schubert} we know that for any permutation $\omega\in S_n$ the codimension of $\Omega_\omega$ in $\flag (n)$ is given by $l(\omega)$. Since we know that ${\rm dim}_k\flag(n)=\frac{n(n-1)}{2}$, we are able to conclude that ${\rm dim}_k\Omega_{\omega_0 s_I}=l$.
\end{remark}

%The previous proposition immediately implies an analogous statement for general degeneracy loci.

%\begin{proposition}

%Let $V$ a vector bundle of rank $n$ over a scheme $X$, $V_{\textbf{\textbullet}}$ a full flag of subbundles of $V$ and  $W_{\textbf{\textbullet}}$ a full flag of quotient bundles. Let $\omega\in S_n$. For every tuple $I$ corresponding to a minimal decomposition of $\omega_0\cdot \omega$, $i_{W\unddot}^{-1}(R_I)$  is birationally isomorphic to $\Omega_{r_\omega}(V\unddot,W\unddot, id_V)$.
%\end{proposition}

%\begin{proof}
%By the preceding proposition we know that for every minimal decomposition $I$ of $\omega_0\cdot\omega$ the Bott-Samelson resolution $R_I$ is birationally isomorphic to $\Omega_\omega$. On the other hand, thanks to lemma \ref{lem preimage}, we also have that $i_{F\unddot}^{-1}(\Omega_\omega)=\Omega_{r_\omega}(E\unddot,F\unddot,id_V)$ and therefore the proposition follows since birational isomorphisms are preserved under pull-backs.  
%\end{proof}

% Since by remark \ref{rem smooth morph} we know that Bott-Samelson resolutions are smooth schemes over $k$, it is possible to consider the cobordism classes $[r_I:R_I\rightarrow \flag (V)]$. These classes will be denoted by $\mathcal{R}_I$. The following lemma shows how the recursive definition of the Bott-Samelson resolution reflects on these classes.

%\section{Schubert, Grothendieck and $\beta$-polynomials} \label{sec polynomials}

\subsection{Schubert, Grothendieck and $\beta$-polynomials} \label{sec polynomials}

We begin this subsection by illustrating the definition of double Schubert and Grothendieck polynomials. These two families of polynomials over $\ZZ$ are both indexed by permutations and are defined using essentially the same procedure, based on the ordering of $S_n$ given by the  length function. We will write $R[\bf{x},\bf{y}]$ for $R\variables$.

\begin{definition}
Fix $n\in\NN$.
For each $i\in\{1,\tred,n-1\}$ we define the divided difference operators $\partial_i$ and the isobaric divided difference operators $\pi_i$ on $\ZZ[\bf{x},\bf{y}]$ by setting

\begin{eqnarray}\label{def SG1}
i)\ \partial_i P= \frac{P-\sigma_i (P)}{x_i-x_{i+1}}\quad ;\quad ii)\ \pi_i P=\frac{(1-x_{i+1}) P-(1-x_{i})\sigma_i (P)}{x_i-x_{i+1}}\ ,
\end{eqnarray}
where $\sigma_i$ is the operator exchanging $x_i$ and $x_{i+1}$.

For $\omega\in S_n$ we define the double Schubert polynomial $\mathfrak{S}_\omega$ and the double Grothendieck polynomial $\mathfrak{G}_\omega$ as follows:

if $\omega=\omega_0$ then  
\begin{eqnarray} \label{def SG2}
i)\ \mathfrak{S}_\omega:=\displaystyle{\prod_{i+j\leq k}}(x_i-y_j)\quad ;\quad ii)\  \mathfrak{G}_\omega:=\displaystyle{\prod_{i+j\leq k}}(x_i+y_j-x_i y_j)\ ;
\end{eqnarray}

if $\omega\not= \omega_0$ then there exist an elementary transposition $s_i$ such that $l(\omega)<l(\omega s_i)$: one then sets
\begin{eqnarray}\label{def SG3}
i)\ \mathfrak{S}_\omega:=\partial_i \mathfrak{S}_{\omega s_i}\quad ;\quad ii)\ 
\mathfrak{G}_\omega:=\pi_i \mathfrak{G}_{\omega s_i}\ .
\end{eqnarray}
\end{definition}

\begin{remark} \label{rem definition 1}

A priori the polynomials $\mathfrak{S}_{\omega}$ and $\mathfrak{G}_{\omega}$ are not associated to the permutation $\omega$ but to one of the many minimal decompositions of $\omega_0 \omega$. One therefore has to show that the definition  is independent of the choice of minimal decomposition. The inspection of the relations satisfied by the elementary transposition shows that they are generated by three types of relations: $s_i^2=id_{S_n}$ for every $i\in\{1,\tred,n-1\}$,  
 $s_i s_j=s_j s_i$ if $|i-j|\geq 2$ and $s_i s_j s_i=s_j s_i s_j $ if $|i-j|=1$. 

As we are only interested in minimal decompositions, the relations relevant for us are the ones that do not alter the size of a decomposition: for this reason we can disregard the first set of relations. On the other hand the remaining ones, which are a particular instance of the so-called \textit{braid relations}, turn minimal decompositions into minimal decompositions and could therefore give rise to different polynomials. One way to ensure that this cannot happen is to show that the divided difference operators themselves satisfy the braid relations.

%However, since the divided difference operators satisfy the braid relations, different minimal decompositions give rise to the same polynomial.
\end{remark}

\begin{remark} \label{rem definition 2}
From the way they have been defined, the polynomials $\mathfrak{S}_\omega$ and $\mathfrak{G}_\omega$ should depend on the choice of $n\in\NN$ and therefore on the ambient symmetric group $\omega$ lives in. This is not actually the case since one can show that $\mathfrak{S}_{\omega_0}$ and $\mathfrak{G}_{\omega_0}$ do not change if one views $\omega_0$ as an element of $S_{n+1}$. Since the definition has $w_0$ as a base case and the recursive steps are not affected by the choice of $n$, the equality for this particular case implies the invariance of the definition for any permutation.
\end{remark}

In \cite{GrothendieckFomin} Fomin and Kirillov unified Schubert and Grothendieck polynomials by defining the double $\beta$-polynomials: this is a family of polynomials over $\ZZ[\beta]$ which specializes to Schubert polynomials when $\beta$ is set to be equal to 0 and to Grothendieck polynomials when $\beta$ equals $-1$. The definition follows the same pattern: one only needs to give an analogue of the divided difference operators and to fix the polynomial associated to the longest permutation $\omega_0$.

%We now start the construction of family a of polynomials over $\ZZ[\beta]$ which will generalize both Schubert and Grothendieck polynomials. We first define the analogues of the divided diffence operators. 

% Consider the assignment $$P\mapsto \frac{(1+\beta x_{i+1}) P-(1+\beta x_{i})}{}$$

\begin{definition}\label{def beta difference}
Fix $n\in \NN$. For each $i\in\{1,\tred,n-1\}$ we define the $\beta$-divided difference operator $\phi_i$ on $\ZZ[\beta][\bf{x},\bf{y}]$ by setting 
\begin{eqnarray}\label{def H1}
\phi_i P=(1+\sigma_i)\frac{(1+\beta x_{i+1})P}{x_i-x_{i+1}}=\frac{(1+\beta x_{i+1}) P-(1+\beta x_{i})\sigma_i (P)}{x_i-x_{i+1}}\ ,
\end{eqnarray}
where $\sigma_i$ is the operator exchanging $x_i$ and $x_{i+1}$ and $1$ represents the identity operator.  

\end{definition}

For these operators to be well-defined, we need the following lemma.

\begin{lemma}
Let $P\in\ZZ[\beta][\bf{x},\bf{y}]$. Then $(x_i-x_{i+1})$ divides $(1+\beta x_{i+1}) P-(1+\beta x_{i})\sigma_i (P)$.
\end{lemma}

\begin{proof}
First of all let us observe that, since the operators are additive, it is sufficient to restrict to monomials. A futher reduction can be made by noticing that each operator  $\phi_i$ is linear with respect to polynomials which are symmetric in $x_i$ and $x_{i+1}$. It therefore suffices to consider only monomials of the shape $x_j^k$, with $j\in\{i,i+1\}$ and $k$ strictly positive. Since the two cases are essentially the same, we will only prove the case $j=i$. One then has
$$(1+\beta x_{i+1}) x_i^k-(1+\beta x_{i})\sigma_i(x_i^k)=(1+\beta x_{i+1}) x_i^k-(1+\beta x_{i})x_{i+1}^k=(x_i^k-x_{i+1}^k)+\beta x_i x_{i+1}(x_i^{k-1}-x_{i+1}^{k-1})\ ,$$
which is clearly divisible by $(x_i-x_{i+1})$.   
\end{proof}

We now prove a result concerning the relations existing between products of divided difference operators.

\begin{proposition}\label{prop independence}
The operators $\phi_i$ satisfy the braid realtions. More precisely, the following equations holds:
\begin{align*}
i)\qquad \ \phi_i\circ\phi_j &=\phi_j\circ\phi_i\qquad \ \text{ if }|i-j|\geq 2 \  ;\\
ii)\ \phi_i\circ\phi_j\circ\phi_i &=\phi_j\circ\phi_i\circ\phi_j\  \text{ if }|i-j|=1 \ .
\end{align*}
\end{proposition}

\begin{proof}
In the course of the proof, in order to simplify the notation, we will write $B_{ij}$ for $\frac{1+\beta x_{i+1}}{x_i-x_j}$ and we will therefore have $\phi_i=(1+\sigma_i) B_{i i+1}$. Moreover, since it does not alter the proof, instead of $i$ and $j$ will write 1 and 3 in $(i)$ and 1 and 2 in $(ii)$.

The proof of the two equalities essentially consists of expressing the different operators as linear combinations of products of $\sigma_i$'s. With this goal in mind, it is useful to notice that a product of operators $\sigma_i$ acts on polynomials by exchanging the variables according to some permutation $\omega$ and therefore one can reasonably denote such a product as $\sigma_\omega$. For instance, using this notation, one would write $\sigma_{(12)}$ for $\sigma_1$.

Now, in order to rewrite the given operators in the needed form, one needs to extract all coefficients $B_{ij}$ from the operators $\sigma_i$.
Let us consider for example the case of $\phi_1\circ\phi_3$: one can modify it as follows
$$\phi_1\circ\phi_3=(1+\sigma_1) B_{12} (1+\sigma_3)B_{34}=
(1+\sigma_1)(B_{12}B_{34}\cdot 1+B_{12}B_{43}\cdot \sigma_3)=$$
$$=B_{12}B_{34}\cdot 1 +B_{12}B_{43}\cdot \sigma_3+
B_{21}B_{34}\cdot \sigma_1+B_{21}B_{43}\cdot \sigma_{(12)(34)}\ .$$

If the same procedure is carried out on the other operators one obtains the following expressions

$$\phi_3\circ\phi_1=B_{34}B_{12}\cdot 1 +B_{34}B_{21}\cdot \sigma_1+
B_{43}B_{12}\cdot \sigma_3+B_{43}B_{21}\cdot \sigma_{(12)(34)}\ ,$$
$$\phi_1\circ\phi_2\circ\phi_1=
(B_{12}B_{23}B_{12}+B_{21}B_{13}B_{12})\cdot 1+
(B_{12}B_{23}B_{21}+B_{21}B_{13}B_{21})\cdot\sigma_1
+B_{12}B_{32}B_{13}\cdot\sigma_2+$$
$$+B_{12}B_{32}B_{31}\cdot\sigma_{(132)}+
B_{21}B_{31}B_{23}\cdot\sigma_{(123)}+
B_{21}B_{31}B_{32}\cdot\sigma_{(13)}\ ,$$ 
$$\phi_2\circ\phi_1\circ\phi_2=
(B_{23}B_{12}B_{23}+B_{32}B_{13}B_{23})\cdot 1+
(B_{23}B_{12}B_{32}+B_{32}B_{13}B_{32})\cdot\sigma_2
+B_{23}B_{21}B_{13}\cdot\sigma_1+$$
$$+B_{32}B_{31}B_{12}\cdot\sigma_{(132)}+
B_{23}B_{21}B_{31}\cdot\sigma_{(123)}+
B_{32}B_{31}B_{21}\cdot\sigma_{(13)}\ .$$ 

When one finally compares the results, it becomes evident that (i) holds and that to prove (ii) it remains to show that the coefficients of $1$, $\sigma_1$ and $\sigma_2$  are actually equal. Since this is achieved by explicit computations we will work out, as an example, the one associated to $1$. After the expressions for $B_{ij}$ have been substituted and the two quantities have been factored, one has the following:
\begin{align*}
B_{12}B_{23}B_{12}+B_{21}B_{13}B_{12}&=\frac{(1+\beta x_2)(1+\beta x_3)}{(x_1-x_2)^2}\left[\frac{1+\beta x_2}{x_2-x_3}-\frac{1+\beta x_1}{x_1-x_3}\right]\\
%&=& \frac{(1+\beta x_2)(1+\beta x_3)}{(x_1-x_2)^2}\left[\right]
B_{23}B_{12}B_{23}+B_{32}B_{13}B_{23}&=
\frac{(1+\beta x_2)(1+\beta x_3)^2}{(x_2-x_3)^2}\left[\frac{1}{(x_1-x_2)}-\frac{1}{(x_1-x_3)}\right]\\
%&=\frac{(1+\beta x_3)(1+\beta x_2)(1+\beta x_3)}{(x_2-x_3)(x_1-x_2)(x_1-x_3)}
\end{align*}

To prove the equality it now suffices to compute explicitly the terms inside the square brackets 
\begin{align*}
\frac{1+\beta x_2}{x_2-x_3}-\frac{1+\beta x_1}{x_1-x_3}&=
\frac{(x_1-x_2)(1+\beta x_3)}{(x_2-x_3)(x_1-x_3)}\ , \\\frac{1}{(x_1-x_2)}-\frac{1}{(x_1-x_3)}&=\frac{x_2-x_3}{(x_1-x_2)(x_1-x_3)}\ .\ \qedhere
\end{align*}
\end{proof}

We are now in the position to introduce the $\beta$-polynomials $\mathfrak{H}_{\omega}$.
%a generalization of Schubert and Grothendieck polynomials.%???

\begin{definition} 
Fix $n\in \ZZ$ and let $\omega\in S_n$.
If $\omega=\omega_0$ then 
\begin{eqnarray}\label{def H2}
\mathfrak{H}_{\omega_0}:=\prod_{i+j\leq k}(x_i+y_j+\beta x_i y_j).
\end{eqnarray}
If $\omega\not= \omega_0$ then there exists an elementary transposition $s_i$ such that $l(\omega)<l(\omega s_i)$ and one sets 
\begin{eqnarray}\label{def H3}
\mathfrak{H}_\omega:=\phi_i \mathfrak{H}_{\omega s_i}\ .
\end{eqnarray} 

\end{definition}

 Exactly as for $\mathfrak{S}$ and $\mathfrak{G}$ (see remarks \ref{rem definition 1}-\ref{rem definition 2}) one has to show that the definition of $\mathfrak{H}_\omega$ does not depend on the choice of a minimal decomposition of $\omega_0\omega$ and on the choice of the symmetric group $S_n$. Thanks to proposition \ref{prop independence} we already know that $\mathfrak{H}_\omega$ is independent of the choice of minimal decomposition. 

We now prove two lemmas that will be used in the proof of the independence of the polynomials from the choice of $n$.

\begin{lemma}\label{lem Fij}
Let $P=x_i+y_j+\beta x_iy_j$. Then $\phi_i P=1$.
\end{lemma}

\begin{proof}
Through easy computations based on the definition of $\phi_i$, one obtains $\phi_i 1=-\beta$ and $\phi_i x_i=1$. This two expression are sufficient to finish the proof: thanks to the linearity of $\phi_i$ with respect to polynomials symmetric in $x_i$ and $x_{i+1}$ and to its additivity, one has
$$\phi_i P=\phi_i(x_i+y_j+\beta x_i y_j)=(1+\beta y_j)\cdot\phi_i x_i+y_j\cdot \phi_i 1=(1+\beta y_j)\cdot 1-y_j\cdot \beta= 1\ . \qedhere$$
\end{proof}

\begin{lemma} \label{lem polynomial}
Fix $n\in \NN$. For every $m\in \NN$ with $1\leq m\leq n+1$, set
$$H_m:=\prod_{i+j\leq n} (x_i+y_j+\beta x_iy_j)\prod_{k=m}^n (x_k+y_{n+1-k}+\beta x_k y_{n+1-k})\ .$$
Then for $1\leq m'\leq n$ one has
$$\phi_{m'} H_{m'}=H_{m'+1}\ .$$
\end{lemma}

\begin{proof}
First of all one rewrites $H_{m'}$ as $H_{m'+1}\cdot(x_{m'}+y_{n+1-{m'}}+\beta x_{m'} y_{n+1-m'})$ and observes that, since in $H_{m'+1}$ the terms $(x_{m'}+y_j+\beta x_{m'}y_j)$ and $(x_{m'+1}+y_j+\beta x_{m'+1}y_j)$ appear in pairs, $H_{m'+1}$  is symmetric in $x_{m'}$ and $x_{m'+1}$. To finish the proof it is now sufficient to use the linearity of $\phi_{m'}$ with respect to symmetric functions and lemma \ref{lem Fij}.
\end{proof}

\begin{proposition}
The polynomials $\mathfrak{H}_\omega$ are independent of the choice of symmetric group $S_n$ to which $\omega$ belongs.
\end{proposition}

\begin{proof}
Let us denote by $\omega_{0,n}$ the longest element of $S_n$ viewed as an element of $S_{n+1}$. As it was observed in remark \ref{rem definition 2}, the proof of the proposition can be reduced to showing that

$$\mathfrak{H}_{\omega_{0,n}}=\prod_{i+j\leq n}(x_i+y_j+\beta x_i y_j).$$

To prove this one first needs to factor $\omega_0$ as  a product of elementary transpositions multiplied by $\omega_{0,n}$: $\omega_0=\omega_{0,n}s_n\trecd s_1$. Then, one recalls the recursive definition of $\mathfrak{H}_{\omega_{0,n}}$ and finishes the proof by applying $n$ times lemma \ref{lem polynomial}: 
$$\mathfrak{H}_{\omega_{0,n}}=\phi_n \trecd \phi_1 \mathfrak{H}_{\omega_0}=\phi_n \trecd \phi_1 H_1=\phi_n\trecd \phi_2 H_2=\trecd=H_{n+1}=\prod_{i+j\leq n}(x_i+y_j+\beta x_i y_j) . \qedhere$$ 
\end{proof}

Let us now denote by $\mathfrak{H}^{(b)}_\omega$ and $\phi^{(b)}_i$ the polynomial and the operators one obtains from $\mathfrak{H}_\omega$ and $\phi^{(b)}$ when $\beta$ is set equal to $b$. Using this notation we can make clear what we mean when we say that the $\beta$-polynomials represent a generalization of both Schubert and Grothendieck polynomials.    

\begin{proposition}\label{prop special}
Fix $n\in\NN$. For every $\omega\in S_n$ one has  
$$i)\ \  \mathfrak{H}^{(0)}_\omega(x_1,\tred,x_n,-y_1,\tred,-y_n)=\mathfrak{S}_\omega \quad ;\quad ii)\ \  \mathfrak{H}^{(-1)}_\omega=\mathfrak{G}_\omega\ .$$ 
\end{proposition}

\begin{proof}
In order to verify the two statements one only has to check that they hold for the special case  $\omega=\omega_0$ and that the $\beta$-divided difference operators $\phi_i$ specialize respectively to $\partial_i$ and $\pi_i$. For this it is sufficient to compare the definitions of the polynomials and of the operators. 
\end{proof}

\begin{remark}\label{rem dualdef}
In the last proposition there is an evident asymmetry between the two equalities, given by the fact that, in order to recover the Schubert polynomials, one has to change the sign of the $y_i$'s in the $\beta$-polynomials. As it will become evident when we will deal with the algebraic cobordism analogue of these concepts, in some sense the problem lies in the definition of the double Schubert polynomial and more specifically in the expression for  $\mathfrak{S}_{\omega_0}$. The choice of setting $\mathfrak{S}_{\omega_0}$ equal to $\prod_{i+j\leq n}x_i-y_j$ instead of $\prod_{i+j\leq n}x_i+y_j$  was probably motivated by the observation that in this way one obtains an easier expression for the Chow ring-valued fundamental classes of Schubert varieties, in which one simply substitutes the Chern roots of the  bundles which are involved. In this way the definition of the double Schubert polynomials already takes into account that it is necessary to take the dual of the second  family of line bundles and this is reflected in the (relatively harmless) sign change.

Unfortunately performing the same operations on double Grothendieck polynomials has a far stronger impact on their expression: one would have to replace $y_j$ with $-\frac{y_j}{1-y_j}$. It is most likely for this reason that in this case it has been decided not to encode in the definition the effects of taking the dual on the second family, creating a gap between the two families of polynomials.  
\end{remark}

%\section{The description of the fundamental classes in the Chow ring }

\subsection{The description of the fundamental classes in the Chow ring }

In this subsection we present the results  which allow to express the Chow ring fundamental classes of both Schubert varieties and degeneracy loci by means of Schubert polynomials. Throughout the subsection $p:V\rightarrow X$ will be a vector bundle of rank $n$, with $\pi:\flag (V)\rightarrow X$ as the associated full flag bundle and  $V_{\textbf{\textbullet}}=(V_1\subset V_2\subset ... \subset V_n=V)$ will be a fixed full flag of subbundles. Let us moreover recall that $\flag (V)$ comes equipped with $Q\unddot$, the universal full flag of quotient bundles of $\pi^* V$.

 We begin our presentation by providing a description of the Chow ring of the flag bundle.

\begin{proposition}\label{prop flag CH}
Let $V$ be a vector bundle over $X\in\SM$ and let $J$ be the ideal of $CH^*(X)[X_1,\tred,X_n]$ generated by the elements $e_i-c_i(V)$ where $e_i$ is the $i$-th elementary symmetric function and $c_i(V)$ is the $i$-th Chern class of $V$. Then the Chow ring of the flag bundle can be described as follows:  
$$CH^*(\flag (V))\simeq CH^*(X)[X_1,\tred,X_n]/J\ .$$ 
\end{proposition}

\begin{proof}
See \cite[Lemma 5.3]{FlagsFulton}.
%The proof is based on the construction of the flag bundle as an iterated $\Proj^1$-bundle and on the projective bundle formula.   
\end{proof}

\begin{remark}
In the proof of the previous lemma the isomorphism is constructed by mapping the variables $X_i$'s to the \textit{Chern roots} of $V$ associated to the universal full flag of quotient bundles $Q\unddot$. For any full flag of quotient bundles $W\unddot=(\pi^*V=W_n\srarrow W_{n-1}\srarrow ... \srarrow W_1)$ the Chern roots are the first Chern classes  $c_1(L^{W\unddot}_i)\in CH^*(\flag(V))$ with $\{1,\tred,n\}$. This notion can as well be defined for full flags of subbundles and in this case the Chern roots associated to the flag $U\unddot=(U_1\subset U_2\subset ... \subset U_n=\pi^*V)$ are the elements $c_1(L_i^{U\unddot})\in CH^*(\flag(V))$ with $i\in \{1,\tred,n\}$. 
\end{remark}

Since every divided difference operator $\partial_i$ is linear with respect to polynomials symmetric in $X_i$ and $X_{i+1}$, it follows that the ideal $J$ is preserved under their action on $CH^*(X)[X_1,\tred,X_n]$. As a consequence one obtains operators $\overline{\partial_i}$ over $CH^*(\flag (V))$ which, as we will see in the next lemma, can be described in terms of pull-back and push-forward morphisms. With this goal in mind let us apply the functor $CH^*$ to diagram (\ref{diagram flag}) and observe that, since $pr_1$ and $\varphi_i$ are smooth morphisms, we obtain 
$$
\xymatrix{
  CH^*(R_{I}) \ar[rr]^{{r_I}_*}& &CH^*(\flag(V)) \\
  CH^*(R_{I'})\ar[r]^{{r_{I'}}_*} \ar[u]_{pr_1^*}& CH^*(\flag(V)) \ar[r]^{{\varphi_j}_*} &CH^*(Y_j)   \ar[u]^{\varphi_j^*} } 
$$

\begin{lemma}\label{lemma operator}
Following the notation from the preceding diagram one has $$\overline{\partial_j}=\varphi_j^*{\varphi_j}_*\ .$$
\end{lemma}

\begin{proof}
See  \cite[Lemma 7.2]{FlagsFulton}.
\end{proof}

This lemma yields the following corollary which, since it relates one with the other the push-forward classes of the Bott-Samelson resolutions, represents the first step towards the description of the fundamental classes of Schubert varieties.    

\begin{corollary} \label{cor BottChow}
Let $I=(i_1,\tred,i_l)$ be an l-tuple with $i_j\in\{1,\tred, n-1\}$ and let $R_{I}$ be the corresponding Bott-Samelson resolution. Then in $CH^*(\flag (V))$ we have the equality
$$\overline{\partial_{i_1}}\trecd\overline{\partial_{i_l}}({r_\emptyset}_*[R_\emptyset]_{CH^*})={r_I}_*[R_I]_{CH^*}\ .$$
\end{corollary}

\begin{proof}
The proof is by induction on the lenght $I$ and the base of the induction is tautologically true as $l=0$ implies $I=\emptyset$. For the inductive step since $l>0$ one can write $I=(I',I_l)$ and, in view of the definition of the Bott-Samelson resolutions, one has $R_I=pr_1^{-1}(R_{I'})$. Therefore, thanks to the functorial compatibilities in the Chow ring between the proper push-forwards and the flat pull-backs, one can write
$$\varphi_{i_l}^*{\varphi_{i_l}}_*{r_{I'}}_*[R_{I'}]_{CH^*}={r_I}_*pr_1^*[R_{I'}]_{CH^*}=
{r_I}_*[R_I]_{CH^*}\ .$$
The statement then follows once both lemma \ref{lemma operator} and the inductive hypothesis are applied to the left hand side.
\end{proof}

Let us recall that by definition the Bott-Samelson resolution $R_\emptyset$ is just the Schubert variety $\Omega_{\omega_0}$. It immediately follows that this last corollary can be used to obtain explicit expressions for the classes $r_{I*}[R_I]_{CH^*}$, provided one has such an expression for $[\Omega_{\omega_0}]_{CH^*}\in CH^*(\flag(V))$.

\begin{lemma}
Let $V\rightarrow X$ be a vector bundle and let $x_i$ and $y_i$ denote respectively the Chern roots associated to the full flags $Q\unddot$ and $\pi^*(V\unddot)$. Then in $CH^*(\flag(V))$ one has
$$[\Omega_{\omega_0}]_{CH^*}=\prod_{i+j\leq n}(x_i-y_j)\ .$$
\end{lemma}

\begin{proof}
See \cite[Section 2.3, Lemma 1]{SchubertFulton}.
\end{proof}

\begin{corollary}\label{cor BottSchubert}
Let $I=(i_1,\tred,i_l)$ be a minimal decomposition and set $\omega=\omega_0 s_I$. Then in $CH^*(\flag(V))$ one has 
$$r_{I*}[R_I]_{CH^*}=\mathfrak{S}_\omega(x_1,\tred,x_n,y_1,\tred,y_n)\ .$$
\end{corollary}

\begin{proof}
For $I=\emptyset$ the statement is just the preceding lemma. For the general case one only has to apply corollary \ref{cor BottChow} and to recall the recursive definition of Schubert polynomials.
\end{proof}

\begin{remark}
It is worth mentioning that corollary \ref{cor BottSchubert} implies that all tuples which are minimal decompositions of the same permutation $\omega$ give rise to Bott-Samelson resolutions whose push-forward classes all coincide as elements of $CH^*(\flag(V))$. 
\end{remark}

\begin{remark}
Even though, as we will see, our interest in Schubert polynomials is due to their ability of describing the fundamental class of Schubert varieties (and more in general degeneracy loci), their definition is a priori only linked to the push-forward classes of Bott-Samelson resolutions. It is the birational invariance of the Chow ring which enables to bridge the gap between these two notions, allowing to describe Schubert varieties by means of the more easily computable classes associated to Bott-Samelson resolutions.  
\end{remark}

The next step is to relate the push-forward classes of Bott-Samelson resolutions to the fundamental classes of Schubert varieties.  

\begin{theorem}\label{th flagbundle}
Let $V\rightarrow X$ be a vector bundle and let $\omega\in S_n$. Denote by $x_i$ and $y_i$ the Chern roots associated to the full flag bundles $Q\unddot$ and $\pi^*V\unddot$.
 In $CH^*(\flag(V))$ one has
$$[\Omega_\omega]_{CH^*}=\mathfrak{S}_\omega(x_1,\tred,x_n,y_1,\tred,y_n)\ .$$
\end{theorem}

\begin{proof}
The statement follows directly from corollary \ref{cor BottSchubert} and proposition \ref{prop resolution}. In fact for every Schubert variety $\Omega_\omega$ one can consider the Bott-Samelson resolution $R_I$, associated to any of the minimal decompositions of $\omega$: in view of part (1) of proposition \ref{prop resolution} $r_I$ is a birational isomorphism and therefore $[\Omega_\omega]_{CH}=r_{I*}[R_I]_{CH}$.  
\end{proof}

We are finally in the position to express the fundamental class of a degeneracy loci, provided this has the expected codimension. This is achieved by pulling back to the base the fundamental class of a suitably constructed Schubert variety. 

\begin{lemma} \label{lem chowultimate}
Given a pure dimensional Cohen-Macaulay scheme $X$, let $V\rightarrow X$ be a vector bundle of rank $n$ with $F\unddot$ and $E\unddot$ full flags respectively of quotient bundles and of subbundles. Let $\omega\in S_n $ and assume that the degeneracy locus $\Omega_{r_\omega}(E\unddot,F\unddot,id_V)$ has codimension $l(\omega)$ in $X$. Then as an element of $CH_*(X)$ the fundamental class of the degeneracy locus is given by the formula  
$$[\Omega_{r_\omega}(E\unddot,F\unddot,id_V)]_{CH_*}=\mathfrak{S}_\omega(x_1,\tred,x_n,y_1,\tred,y_n)\ ,$$
where we denote by $x_i$ the Chern roots associated to $F\unddot$ and by $y_i$ the Chern roots associated to $E\unddot$.
\end{lemma}

\begin{proof}
First of all one should observe that in view of lemma \ref{lem preimage} we have that $i_{F\unddot}^{-1}(\Omega_{\omega})=\Omega_{r_\omega}(E\unddot,F\unddot,id_V)$ where $i_{F\unddot}:X\rightarrow \flag (V)$ is the morphism associated to the flag $F\unddot$.

$$
\xymatrix{
\Omega_{r_\omega}(E\unddot,F\unddot,id_V)  \ar@{^{(}->}[r] \ar@{^{(}->}[d]& \Omega_\omega   \ar@{^{(}->}[d] \\
X   \ar@{^{(}->}[r]^{i_{F\unddot}} & \flag (V)
}
$$
The assumption on the codimension $\Omega_{r_\omega}(E\unddot,F\unddot,id_V)$ in $X$, together with the fact that $\Omega_\omega$ is a Cohen-Macaulay scheme, implies that the embedding $\Omega_{r_\omega}(E\unddot,F\unddot,id_V)\subset \Omega_\omega$ is regular and therefore one has that the Gysin morphism $i_{F\unddot}^*$ maps the fundamental class of $\Omega_\omega$  onto the fundamental class of $i^{-1}_{F\unddot}(\Omega_\omega)$.% which, by lemma $\ref{lem preimage}$, coicides with $\Omega_{r_\omega}(E\unddot,F\unddot,id_V)$.

 To proceed in the proof one now has to apply theorem \ref{th flagbundle} so to be able to express the fundamental class of the Schubert variety $\Omega_\omega$ as the Schubert polynomial $\mathfrak{S}_\omega$ evaluated at the two families of Chern roots $\{x'_i\}$ and $\{y'_i\}$, which are associated to the full flags $Q\unddot$ and $\pi^*(E\unddot)$.  The final step consists in applying to this polynomial $i_{F\unddot}$: since $\mathfrak{S}_\omega$ has coefficients in $\ZZ=CH^*(\spec k)$, one only has to worry about the effect of the Gysin morphism on the Chern roots. These are mapped onto the Chern roots of the pull-back of the respective flags which are just $F\unddot$ (by the universal property of $\flag (V)$) and $E\unddot$ (since $i_{F\unddot}\pi=id_V$). One therefore has  
$$[\Omega_{r_\omega}(E\unddot,F\unddot,id_V)]_{CH}=i_{F\unddot}^*[\Omega_\omega]_{CH}=
i_{F\unddot}^*(\mathfrak{S}_\omega(x'_1,\tred,x'_n,y'_1,\tred,y'_n))=
\mathfrak{S}_\omega(x_1,\tred,x_n,y_1,\tred,y_n)\ .
\qedhere $$
\end{proof}

\begin{theorem}\label{th chowultimate}
Let $h:E\rightarrow F$ be a morphism of vector bundles of rank $n$ over a pure dimensional Cohen-Macaulay scheme $X$. Let $E_{\textbf{\textbullet}}$ and $F_{\textbf{\textbullet}}$ be full flags of $E$ and $F$ respectively.  Let $\omega\in S_n $ and assume that the degeneracy locus $\Omega_{r_\omega}(E\unddot,F\unddot,h)$ has codimension $l(\omega)$ in $X$. Then as an element of $CH_*(X)$ the fundamental class of the degeneracy locus is given by the formula  
$$[\Omega_{r_\omega}(E\unddot,F\unddot,h)]_{CH_*}=\mathfrak{S}_\omega(x_1,\tred,x_n,y_1,\tred,y_n)\ ,$$
where we denote by $x_i$ the Chern roots associated to $F\unddot$ and by $y_i$ the Chern roots associated to $E\unddot$.
\end{theorem}

\begin{proof}
One first uses lemma \ref{lem id} and then applies theorem \ref{lem chowultimate} to the locus $\Omega_{r'_{\omega'}}(E'\unddot,F'\unddot,id_V)$. To conclude the proof it suffices to observe that, as $\omega'$ is nothing but $\omega$ viewed as an element of $S_{2n}$, one has $\mathfrak{S}_{\omega'}\in\ZZ[X_1,\tred,X_n,Y_1,\tred,Y_n]$, while, by construction, the first $n$ Chern roots of $E'\unddot$ and $F'\unddot$ (which we denote by $y'_i$ and $x'_i$) coincide with the Chern roots of $E\unddot$ and $F\unddot$. Summing up, one gets the following chain of equalities:
$$[\Omega_{r_\omega}(E\unddot,F\unddot,h)]_{CH_*}=[\Omega_{r_{\omega'}}(E'\unddot,F'\unddot,id_V)]_{CH_*}=\mathfrak{S}_{\omega'}(x'_1,\tred,x'_{2n},y'_1,\tred,y'_{2n})=\mathfrak{S}_\omega(x_1,\tred,x_n,y_1,\tred,y_n)\ .\qedhere$$   
\end{proof}

%\section{The description of the fundamental classes in the Grothendieck ring}

\subsection{The description of the fundamental classes in the Grothendieck ring}
We will now give an illustration of the results that can been obtained when the Chow ring is replaced with the Grothendieck ring of vector bundles. As we will see both theorem \ref{th flagbundle} and theorem \ref{th chowultimate} have an exact counterpart in this setting. In \cite{PieriFulton} Fulton and Lascoux proved that the fundamental classes of Schubert varieties can be expressed by means of Grothendieck polynomials, exactly as in theorem \ref{th flagbundle}, while in \cite{GrothendieckBuch} Buch proved an analogue of  theorem \ref{th chowultimate} which extends the result to degeneracy loci of the right codimension. In stating the theorems we will follow the notations used by Buch. 

Before we state the theorems, it is worth  recalling that also in the Grothendieck ring of vector bundles one can define first Chern classes for line bundles. For a line bundle $L$ one sets 
\begin{align}
c_1(L):= 1-[L^\vee]\ . \label{eq chern}
\end{align}

\begin{theorem} \label{th flagbundle G_0}
Let $V\rightarrow X$ be a vector bundle and let $\omega\in S_n$. In $K^0(\flag(V))$ one has
\begin{align*}
[\mathcal{O}_{\Omega_\omega}]_{K^0}&=\mathfrak{G}_\omega(1-[M_1^\vee],\tred,1-[M_n^\vee],1-[N_1],\tred,1-[N_n])=\\
&=\mathfrak{G}_\omega(c_1(M_1),\tred,c_1(M_n),c_1(N^\vee_1),\tred,c_1(N_n^\vee))\ , 
\end{align*}

where for $i\in\{1,\tred, n\}$ we set  $M_i:=L_i^{Q\unddot}$ and $N_i:=L_i^{\pi^* V\unddot}$.  
\end{theorem}

\begin{proof}
See  \cite[Theorem 3]{PieriFulton}.
\end{proof}

\begin{remark}
It can be worth to point out that in the proof of the previous theorem it is necessary to make use of both parts of proposition \ref{prop resolution}. In fact, it not sufficient to know that $r_I:R_I\rightarrow \Omega_\omega$ is a birational isomorphism, one also needs to know that $\Omega_\omega$ is normal and that it has at most rational singularities to be able to conclude that $r_{I*}[\mathcal{O}_{R_I}]_{K^0}=[\mathcal{O}_{\Omega_\omega}]_{K^0}$.
\end{remark}

In view of (\ref{eq chern}) and of remark \ref{rem dualdef} the parallelism with the Chow ring case becomes evident: in both cases the fundamental classes of Schubert varieties are written by means of two families of polynomials in Chern roots which are defined by the same exact inductive procedure.  Of course the similarities are not limited to the statement: the main structure of the proof itself is untouched. Again one first establishes a connection between double Grothendieck polynomials and the push-forward classes of Bott-Samelson resolutions by reducing everything to the special case of the longest permutation $w_0$ and successively one is left to show that each of these classes actually coincides with the fundamental class of the corresponding Schubert variety.

Starting from this result one can proceed further and obtain the following statement which covers the more general case of a degeneracy loci of a morphism between vector bundles. Also in this case the proof is essentially unchanged.

\begin{theorem}
Let $h:E\rightarrow F$ be a morphism of vector bundles of rank $n$ over a smooth scheme $X$. Let $E_{\textbf{\textbullet}}$ and $F_{\textbf{\textbullet}}$ be full flags of $E$ and $F$ respectively.  Let $\omega\in S_n $ and assume that the degeneracy locus $\Omega_{r_\omega}(E\unddot,F\unddot,h)$ has codimension $l(\omega)$ in $X$. Then as an element of $K^0(X)$ the fundamental class of the degeneracy locus is given by 
\begin{align*}
[\mathcal{O}_{\Omega_{r_\omega}(E\unddot,F\unddot,h)}]_{K^0}&=\mathfrak{G}_\omega(1-[M_1^\vee],\tred,1-[M_n^\vee],1-[N_1],\tred,1-[N_n])= \\
&=\mathfrak{G}_\omega(c_1(M_1),\tred,c_1(M_n),c_1(N_1^\vee),\tred,c_1(N_n^\vee))\ ,
\end{align*}

where for $i\in\{1,\tred, n\}$ we set  $M_i:=L_i^{F\unddot}$ and $N_i:=L_i^{E\unddot}$. 
\end{theorem}

\begin{proof}
See \cite[Theorem 2.1]{GrothendieckBuch}.
\end{proof}
 
 %To conclude the proof it now suffices to point out that the construction of the bundle $V$ implies that the Chern roots associated to $E'\unddot$ and $F'\unddot$ contain those of $E\unddot$ and $F\unddot$ and all the remaining ones are zero. This is due to the fact that $V$ is 
%the extra line bundles arising from the two filtrations are all trivial:
%$$[\Omega_{r_\omega}(E\unddot,F\unddot,h)]_{CH}=[\Omega_{r_{\omega'}}(E'\unddot,F'\unddot,id_V)_{CH}]_{CH}=\mathfrak{S}_{\omega'}(x'_1,\tred,x'_{2n},y'_1,\tred,y'_{2n})=\mathfrak{S}_\omega$$ 
 %and to notice that the costruction of $V$ is such as to m  

%\chapter{Cobordism classes of Bott-Samelson resolutions and application to connected $K$-theory} \label{ch 3}
\section{Cobordism classes of Bott-Samelson resolutions and application to connected $K$-theory} \label{ch 3}

In this section we illustrate how the method used by Fulton for the Chow ring can be applied also to algebraic cobordism. We first present the analogue of the divided difference operators and then we compute the cobordism class of $\Omega_{\omega_0}$ as an element of $\Omega^*(\flag (V))$. In this way we achieve the description of the push-forward classes of the Bott-Samelson resolution in $\Omega^*(\flag (V))$ and afterwards we specialize it to connected $K$-theory, giving a geometric interpretation to the double $\beta$-polynomials of subsection \ref{sec polynomials}.

Throughout this section we will assume the base field $k$ to have characteristic 0.

%\section{A formula for the push-forward of $\mathbb{P}^1$-bundles}

\subsection{A formula for the push-forward of $\mathbb{P}^1$-bundles}

In \cite{SchubertHornbostel} Hornbostel and Kiritchenko specialize the results of a theorem by Vishik (\cite[Theorem 5.30]{SymmetricVishik}) and give an explicit formula for the push-forward map along a $\Proj ^1$-bundle $\varphi :\mathbb{P}(E)\rightarrow X$. They then use this formula to  build an operator $A_\varphi: \Omega ^*(\mathbb{P}(E))\rightarrow \Omega ^*(\mathbb{P}(E))$ which they prove to coincide with $\varphi ^*\varphi_*$. This is achieved in the following way. First of all they define an operator $A:\Omega ^*(X)[[y_1,y_2]]\rightarrow\Omega ^*(X)[[y_1,y_2]]$ by setting 
$$A(f)= (1+\sigma)\frac{f}{F(y_1,\chi (y_2))}$$
where $[\sigma (f)](y_1,y_2)=f(y_2,y_1)$ and they show that it is well-defined. They then substitute the Chern roots of $E$ (denoted by $\alpha_1$ and $\alpha_2$)  for $y_1$ and $y_2$. If one examines more in detail what it means to substitute the Chern roots, one notices that from $\Omega^* (X)[[y_1,y_2]]$ one actually recovers $\Omega ^*(\mathbb{P}(E))$. More precisely one has  

$$\Omega ^*(\mathbb{P}(E))\simeq \frac{\Omega ^*(X)[[y_1,y_2]]}{(y_1+y_2-c_1(E),y_1y_2-c_2(E))}\ .$$

With this description the embedding of $\Omega ^*(X)$ into $\Omega ^*(\mathbb{P}(E))$ (given by the pull-back along $\varphi$) turns $\Omega ^*(X)$ into the subring of symmetric power series in $\alpha_1$ and $\alpha_2$. In fact, every symmetric power series in Chern roots can be written as a power series in Chern classes and therefore, since the Chern classes are all nilpotents, as an element of $\Omega ^*(X)$. 

 It can be easily checked that the image of $A$ consists of symmetric power series and as a consequence the composition $\Omega ^*(X)[[y_1,y_2]]\rightarrow\Omega ^*(X)[[y_1,y_2]]\rightarrow \Omega ^*(\mathbb{P}(E))$ factors through $\Omega ^*(X)$. Moreover, since $A$ maps the ideal $(y_1+y_2-c_1(E),y_1y_2-c_2(E))$ into itsefl, it is possible to define a new operator $A_\varphi:\Omega ^*(\mathbb{P}(E))\rightarrow \Omega ^*(\mathbb{P}(E))$ and this again factors through $\Omega ^*(X)$.

It actually turns out that the first map of this factorization is $\varphi_*$. To prove this, it is sufficient to show that the two maps coincide on the generators of $\Omega ^*(\mathbb{P}(E))$ as an $\Omega ^*(X)$-module. This is precisely what Hornbostel and Kiritchenko prove: 

$$\varphi _*(1_y)=[A(1)](\alpha_1,\alpha_2)\ ,$$ 
$$\varphi _*(\xi)=[A(y_1)] (\alpha_1,\alpha_2)\ .$$ 
These two equalities imply that the two maps are equal since, by the projective bundle formula, 
$\Omega(\mathbb{P}(E))\simeq 1_{\mathbb{P}(E)}\,\Omega ^*(X)\oplus \xi \,\Omega ^*(X) $ (here $\xi=c_1(\mathcal{O}_E(1))$). Finally, by composing with $\varphi ^*$ one is able to conclude that $A_\varphi=\varphi ^*\varphi_*$. Summarizing we have the following proposition (\cite[Proposition 2.1 and corollary 2.3]{SchubertHornbostel}).

\begin{proposition} \label{prop operator}
Let $\varphi: \mathbb{P}(E)\rightarrow X$ be a $\mathbb{P}^1$-bundle and $A_\varphi:\Omega ^*(\mathbb{P}(E))\rightarrow \Omega ^*(\mathbb{P}(E))$ be  the operator obtained from  
\begin{align*}
\Omega ^*(X)[[y_1,y_2]]&\stackrel{A}{\longrightarrow}\hspace{0.3 cm}\Omega ^*(X)[[y_1,y_2]] \\
f\hspace{0.8 cm}&\longmapsto (1+\sigma)\frac{f}{F(y_1,\chi (y_2))}\ ,
\end{align*}
 by substituting the Chern roots of $E$ for $y_1,y_2$.
Then  $A_\varphi=\varphi ^*\varphi_*$. 

\end{proposition}

Once this result has been established one can use it, as we will see in the next subsection, to compute recursively the cobordism classes associated to the Bott-Samelson resolutions. 

%\newpage

%\section{Operators on $\flag (V)$ and the classes $\mathcal{R}_I$}

\subsection{Operators on $\flag (V)$ and the classes $\mathcal{R}_I$}

We now turn our attention to the flag bundle and we provide a description of the algebraic cobordism ring $\Omega^*(\flag(V))$ which mirrors the one we gave in proposition \ref{prop flag CH} for the Chow ring. 

\begin{proposition}
Let $V$ be a vector bundle over $X\in\SM$ and let $J$ be the ideal of $\Omega^*(X)[X_1,\tred,X_n]$ generated by the elements $e_i-c_i(V)$ where $e_i$ is the $i$-th elementary symmetric function and $c_i(V)$ is the $i$-th Chern class of $V$. Then the algebraic cobordism ring of the flag bundle can be described as follows:  
$$\Omega^*(\flag (V))\simeq \Omega^*(X)[X_1,\tred,X_n]/J\ .$$ 
\end{proposition}

\begin{proof}
See \cite[Theorem 2.6]{SchubertHornbostel}.
\end{proof}

 Since by remark \ref{rem smooth morph} Bott-Samelson resolutions are smooth schemes over $k$, it follows that every morphism $r_I$ defines a cobordism class $[r_I:R_I\rightarrow \flag (V)]\in\Omega^*(\flag(V))$. We will denote this class by $\mathcal{R}_I$. The following lemma shows how the recursive definition of the Bott-Samelson resolution reflects on these classes. 

\begin{lemma} \label{lem Bott}
 Let $I$ be an $l$-tuple with $I=(I',i_{l})$. Then $\mathcal{R}_I={\varphi_{i_l}} ^* {\varphi_{i_l}}_*\mathcal{R}_{I'}$.  

\end{lemma}

\begin{proof}
The stament follows directly from definition \ref{def Bott-Samelson}. More in detail one has 
$${\varphi_{i_{l}}}_* \mathcal{R}_{I'}={\varphi_{i_{l}}}_*[r_{I'}:R_{I'}\rightarrow \flag (V)]=[\varphi_{i_{l}} \circ r_{I'}: R_{I'}\rightarrow Y_{i_l}]\ ,$$
 so taking the pull-back along $\varphi_{i_l}$ gives exactly $$[p_2:R_{I'}\times_{Y_{i_l}}\flag (V)  \rightarrow \flag (V)]=\mathcal{R}_I\ .\qedhere$$
\end{proof}

 Let us now recall remark \ref{rem P^1 bundle} and denote by $A_i: \Omega^*(\flag (V))\rightarrow \Omega^*(\flag (V))$ the operators arising from the $\mathbb{P}^1$-bundles $\varphi_i:\flag (V)\rightarrow Y_i$. Exactly as for the Chow ring, by means of these operators one can relate the classes of any Bott-Samelson resolution to the initial class $\mathcal{R}_\emptyset$.  It is therefore central to have an explicit description of this particular class as this will allow us to compute all the other ones.

Let us recall that a Schubert variety $\Omega_\omega$ was defined as $\Omega_{r_\omega}(\pi^* V\unddot,Q\unddot, h=id_{\pi^* V})$ and that the morphism  $h_{lk}:\pi^*(V_l)\rightarrow Q_k$ is given by the composition of the restriction of $h$ to $V_l$ with $\pi^*(V)\srarrow Q_k$. 
% Since it will be necessary for computing $\mathcal{R}_\emptyset$, we first recall how the Chern polynomial of a bundle $\text{Hom}(E,F)$ can be expressed in terms of the Chern roots of the arguments.   

\begin{proposition} \label{prop initial class}
Let $V\unddot=(V_1\subset V_2\subset ... \subset V_n=V)$ be a full flag of subbundles of $V$ and $Q\unddot=(\pi^*V=Q_n\srarrow Q_{n-1}\srarrow ... \srarrow Q_1)$ be the universal full flag of quotient bundles of $\pi^* V$. Denote by $x_i$ and $y_i$ the Chern roots associated to the full flags $Q\unddot$ and $\pi^*(V\unddot)$.
%For $k\in\{1,\tred,n\}$ set $x_k=c_1({\rm Ker}(Q_k\srarrow Q_{k-1}))$ and  $y_k=c_1(\pi^*V_k/\pi^*V_{k-1})$.
Then  
$$\mathcal{R}_\emptyset=\prod_{k+j\leq n} F(x_k,\chi (y_j))\ .$$
\end{proposition}

\begin{proof}
The strategy of the proof is to construct a bundle $K$ together with a section $s$, such that the zero scheme $Z(s)$ will coincide with $\Omega_{\omega_0}$. To do so, first of all let us consider the morphism of vector bundles 
$$\psi:M=\bigoplus_{l=1}^{n-1} \text{Hom}(\pi^*V_l,Q_{n-l})\longrightarrow \bigoplus_{l=1}^{n-2}\text{Hom}(\pi^*V_l,Q_{n-l-1})=M'$$
which assigns to the family 
$\{g_l\}_{l\in\{1,\tred, n-1\}}$ the family $\{g_{l+1}\circ i_l-p_{n-l}\circ g_l\}_{l\in\{1,\tred, n-2\}}$. Here 
$i_l:\pi^*V_l\irarrow \pi^*V_{l+1}$ and $p_l: Q_l\srarrow Q_{l-1}$ are respectevely the injections and the projections within the two flags. As it is easy to check that $\psi$ is surjective, we have the following exact sequence of bundles:
$$0\longrightarrow {\rm Ker\,}\psi\longrightarrow M\stackrel{\psi}\longrightarrow M'\longrightarrow 0\ .$$
Since we know the ranks of $M$ and $M'$, this sequence allows us to compute the rank of $K:=\text{Ker\,}\psi$. We will denote this rank by $N$. 
\begin{align*}
{\rm rank\,} K&={\rm rank\,}M-{\rm rank\,}M'=\sum_{l=1}^{n-1}l(n-l)-\sum_{l=1}^{n-2}l(n-l-1)=(n-1)+\sum_{l=1}^{n-2}[l(n-l)-l(n-l-1)]=\\
&=(n-1)+\sum_{l=1}^{n-2}l=\sum_{l=1}^{n-1}l=\frac{n(n-1)}{2}
\end{align*}
Moreover, thanks to the Whitney formula, we have 
$$c_t(M)=c_t(K)c_t(M')$$
 and, when one looks at the leading coefficients of both sides, this implies that 
$$c_{{\rm rank\,} M}(M)=c_N(K)c_{{\rm rank\,} M'}(M')\ .$$ 
It therefore follows that we can compute the top Chern class of $K$ by taking the ratio of the top Chern classes of $M$ and $M'$. It is worth noticing that the previous equality guarantees that this division is well defined. Now, in order to compute these top Chern classes, we again make use of the Whitney formula: this time we successively remove all direct summands. In this way we obtain the following expressions for the Chern polynomials 
$$c_t(M)=\prod_{l=1}^{n-1}c_t (\text{Hom}(\pi^*V_l,Q_{n-l}))\quad ,\quad c_t(M')=\prod_{l=1}^{n-2}c_t (\text{Hom}(\pi^*V_l,Q_{n-l-1}))\ ,$$
each of which, exactly as before, provides us with an expression for the top Chern class
$$c_{{\rm rank\,} M}(M)=\prod_{l=1}^{n-1}c_{l(n-l)} (\text{Hom}(\pi^*V_l,Q_{n-l}))\quad ,\quad c_{{\rm rank\,} M'}(M')=\prod_{l=1}^{n-2}c_{l(n-l-1)} (\text{Hom}(\pi^*V_l,Q_{n-l-1}))\ .$$
At this point the last missing piece of information is a formula for the top Chern class of a bundle of the form $\text{Hom}(\pi^* V_{m_1},Q_{m_2})$.
We will achieve this by computing the top Chern class of another bundle, isomorphic to the given one:  $(\pi^* V_{m_1})^\vee\otimes Q_{m_2}$.   
Since $\pi^* V_{m_1}$ has a full flag of subbundles and $Q_{m_2}$ has a full flag of quotient bundles, we can apply corollary \ref{cor Chern} which returns us
$$c_{m_1m_2}((\pi^* V_{m_1})^\vee\otimes Q_{m_2})=
\prod_{l=1}^{m_1}\prod_{k=1}^{m2} F(c_1({\rm Ker}(Q_k\srarrow Q_{k-1})),\chi (c_1(\pi^*V_l/\pi^*V_{l-1})))=\prod_{l=1}^{m_1}\prod_{k=1}^{m2} F(x_k,\chi (y_l))\ .$$
 We are finally able to compute $c_N(K)$. 

%$$
%c_N(K)=\frac{\prod_{l=1}^{n-1}c_{l(n-l)}( \text{Hom}(\pi^*V_l,Q_{n-l}))}{\prod_{l=1}^{n-2}c_{l(n-1-l)}( \text{Hom}(\pi^*V_l,Q_{n-l-1}))}=c_{n-1}( \text{Hom}(\pi^*V_{n-1},Q_1))\cdot\prod_{l=1}^{n-2}\frac{c_{l(n-l)}( \text{Hom}(\pi^*V_l,Q_{n-l}))}{c_{l(n-1-l)}( \text{Hom}(\pi^*V_l,Q_{n-l-1}))}=$$
%$$=\prod_{j=1}^{n-1}F(x_1,\chi (y_j))\cdot
%\prod_{l=1}^{n-2} \prod_{j=1}^{l}
%\frac{\prod_{k=1}^{n-l}F(x_k,\chi (y_j))}{\prod_{k=1}^{n-1-l}F(x_k,\chi (y_j))}=\prod_{j=1}^{n-1}F(x_1,\chi (y_j))\cdot
%\prod_{l=1}^{n-2} \prod_{j=1}^{l} F(x_{n-l},\chi(y_j))=$$
%$$=\prod_{l=1}^{n-1} \prod_{j=1}^{l} F(x_{n-l},\chi(y_j))=
%\prod_{k=1}^{n-1} \prod_{j=1}^{n-k} F(x_{k},\chi(y_j))=
%\prod_{k+j\leq n}F(x_{k},\chi(y_j))
%$$

%\begin{align*}
%c_N(K)&=\frac{\prod_{l=1}^{n-1}c_{l(n-l)}( \text{Hom}(\pi^*V_l,Q_{n-l}))}{\prod_{l=1}^{n-2}c_{l(n-1-l)}( \text{Hom}(\pi^*V_l,Q_{n-l-1}))}=\\
%&=c_{n-1}( \text{Hom}(\pi^*V_{n-1},Q_1))\cdot\prod_{l=1}^{n-2}\frac{c_{l(n-l)}( \text{Hom}(\pi^*V_l,Q_{n-l}))}{c_{l(n-1-l)}( \text{Hom}(\pi^*V_l,Q_{n-l-1}))}
%=\\
%&=\prod_{j=1}^{n-1}F(x_1,\chi (y_j))\cdot
%\prod_{l=1}^{n-2} \prod_{j=1}^{l}
%\frac{\prod_{k=1}^{n-l}F(x_k,\chi (y_j))}{\prod_{k=1}^{n-1-l}F(x_k,\chi (y_j))}=\\
%&=\prod_{j=1}^{n-1}F(x_1,\chi (y_j))\cdot
%\prod_{l=1}^{n-2} \prod_{j=1}^{l} F(x_{n-l},\chi(y_j))=\\
%&=\prod_{l=1}^{n-1} \prod_{j=1}^{l} F(x_{n-l},\chi(y_j))=
%\prod_{k=1}^{n-1} \prod_{j=1}^{n-k} F(x_{k},\chi(y_j))=
%\prod_{k+j\leq n}F(x_{k},\chi(y_j))
%\end{align*}

\begin{align*}
c_N(K)&=\frac{\prod_{l=1}^{n-1}c_{l(n-l)}( \text{Hom}(\pi^*V_l,Q_{n-l}))}{\prod_{l=1}^{n-2}c_{l(n-1-l)}( \text{Hom}(\pi^*V_l,Q_{n-l-1}))}=c_{n-1}( \text{Hom}(\pi^*V_{n-1},Q_1))\cdot\prod_{l=1}^{n-2}\frac{c_{l(n-l)}( \text{Hom}(\pi^*V_l,Q_{n-l}))}{c_{l(n-1-l)}( \text{Hom}(\pi^*V_l,Q_{n-l-1}))}
=\\
&=\prod_{j=1}^{n-1}F(x_1,\chi (y_j))\cdot
\prod_{l=1}^{n-2} \prod_{j=1}^{l}
\frac{\prod_{k=1}^{n-l}F(x_k,\chi (y_j))}{\prod_{k=1}^{n-1-l}F(x_k,\chi (y_j))}=\prod_{j=1}^{n-1}F(x_1,\chi (y_j))\cdot
\prod_{l=1}^{n-2} \prod_{j=1}^{l} F(x_{n-l},\chi(y_j))=\\
&=\prod_{l=1}^{n-1} \prod_{j=1}^{l} F(x_{n-l},\chi(y_j))=
\prod_{k=1}^{n-1} \prod_{j=1}^{n-k} F(x_{k},\chi(y_j))=
\prod_{k+j\leq n}F(x_{k},\chi(y_j))
\end{align*}

Now that we have computed the top Chern class of $K$, we still need to provide a section such that its zero scheme coincide with $\Omega_{\omega_0}$.  For this reason, let us consider the family of morphisms $h_{l,n-l}:\pi^*V_l\irarrow \pi^*V\srarrow Q_{n-l}$. It is clearly sent to $0$ by $\psi$ and, as consequence, it defines a section of $K$, which we will denote~$s$. The isomorphism of $Z(s)$ and $\Omega_{\omega_0}$ then follows from lemma \ref{lem Omega0}
$$Z(s)=\bigcap_{l=1}^{n-1} Z(h_{l,n-l})=\Omega_{\omega_0}\ .$$
In order to conlcude the proof it is now sufficient to observe that, by lemma \ref{lem regular Schubert},  $\Omega_{\omega_0}$ is smooth, is regularly embedded in $\flag (V)$ and has codimension $l(\omega_0)=\frac{n(n-1)}{2}=N $: this allows to apply part (2) of lemma \ref{lem top Chern}. One then has
$$\mathcal{R}_\emptyset=[\Omega_{\omega_0}\irarrow \flag (V)]=[Z(s)\irarrow \flag (V)]=c_N(K)=\prod_{k+j\leq n} F(x_k,\chi (y_j))\ . \qedhere$$
\end{proof}  

% Both $\pi^*V_{m_1}$ and $Q_{m_2}$ are completely filtered, therefore one has 
%$$c_{m_1m_2}(\text{Hom}(\pi^*V_{m_1},Q_{m_2}))=c_{m_1m_2}((\pi^*V_{m_1})^\vee\otimes Q_{m_2})=\prod_{j=1}^{m_1}\prod_{k=1}^{m2} F(x_k,\chi (y_j))\ .$$
%Since the set of factors $F(x_k,\chi (y_j))$ of the numerator contains the one for the denominator, it is sufficient to check which factors appears only in the numerator to conclude that
%$c_N(K)=\prod_{k+j\leq n} F(x_k,\chi (y_j))\ .$

It now remains to express the relationship between $\mathcal{R}_\emptyset$ and the other classes.

\begin{theorem}  \label{prop Bott-Samelson}
 For $I=(i_1,...,i_l)$, $\mathcal{R}_I=A_{i_l}\trecd A_{i_1} \mathcal{R}_\emptyset$. 
\end{theorem}

\begin{proof}
The proof is by induction on the number of elements in the $l$-tuple $I$. While for $l=0$  the statement is trivial, the inductive step can be proved by combining lemma \ref{lem Bott}
 and proposition \ref{prop operator}:
$$\mathcal{R}_I=\mathcal{R}_{(I',i_l)}={\varphi _{i_l}}^* {\varphi _{i_l}}_* \mathcal{R}_{I'}=A_{i_l} \mathcal{R}_{I'}=
A_{i_l}A_{i_{l-1}}\trecd A_{i_1} \mathcal{R}_\emptyset\ . \qedhere$$   
\end{proof}

\begin{remark}
The previous result represents the extension of theorem 3.2 in \cite{SchubertHornbostel} from the case of the flag manifold (in which the base scheme is $\spec k$) to a general flag bundle with smooth base $X$. 
\end{remark}

We end this subsection by pulling back the classes $\mathcal{R_I}$ to the base. 

\begin{definition}
Let $V\rightarrow X$ be a vector bundle with $V\unddot$ and $W\unddot$ full flags  of respectively subbundles and quotient bundles. Let $i_{W\unddot}:X\rightarrow\flag(V) $ be the section associated to $W\unddot$ by the universal property of $\flag (V)$.
To every degeneracy locus $\Omega_{r_{\omega}}(V\unddot,W\unddot,id_V)$  we can associate a class 
$$\mathbf{\Omega}_I:=i_W^*(\mathcal{R}_I)\in\Omega^*(X)$$ 
which depends on the choice of $R_I$, one of the Bott-Samelson resolutions birationally isomorphic to the Schubert variety $\Omega_\omega$. Here $I$ represents any of the minimal decompositions of $\omega_0\omega$.  
\end{definition}

%\section{Specialization to connected $K$-theory}

\subsection{Specialization to connected $K$-theory}

In this subsection we are going to state the conclusions that can be drawn for $CK^*(\flag (V))$ from  the results we have obtained in $\Omega^*(\flag (V))$. As before $V$ will be a vector bundle of rank $n$ over $X\in\SM$, equipped with a full flag $V\unddot$, by means of which all Schubert varieties are meant to be defined. The universal full flag of quotient bundles over $\flag (V)$ will be denoted $Q\unddot$.
% All the Schubert varieties we will consider are defined by means of $V\unddot$.

\begin{proposition}
Let $\Omega_\omega$ be the Schubert variety associated to $\omega\in S_n$. Then $\eta_{\Omega_\omega}=\vartheta_{CK_*}([R_I\rightarrow \Omega_\omega])\in CK_*(\Omega_\omega)$ for every Bott-Samelson resolution $r_I:R_I\rightarrow \flag (V)$ associated to $I$, a minimal decomposition of $\omega_0\omega$ .  
\end{proposition}

\begin{proof}
We know from part (1) of proposition \ref{prop resolution} that Bott-Samelson resolutions arising from minimal decompositions actually map onto the corresponding Schubert variety and therefore it makes sense to talk about the cobordism classes $[R_I\rightarrow \Omega_\omega]$. Moreover, again by proposition \ref{prop resolution}, each $R_I$ of the given kind is a resolution of singularities of $\Omega_\omega$, so we can finish the proof by applying $\vartheta_{CK}$ and recalling the definition of $\eta_{\Omega_\omega}$. 
\end{proof}

An immediate corollary of this result is that all the classes $\vartheta_{CK^*}([r_I:R_I\rightarrow \flag (V)])$ related to the same Schubert variety coincide in $CK^*(\flag (V))$.

%An immediate corollary of this result is the coincidence in $CK^*(\flag (V))$ of all classes $[r_I:R_I\rightarrow \flag (V)]_{CK}$ related to the same Schubert variety.

\begin{corollary} \label{cor class}
With the same notations as in the previous proposition, let $j$ be the inclusion of $\Omega_\omega$ into $\flag (V)$. Then 
$j_*\eta_{\Omega_\omega}=\vartheta_{CK}(\mathcal{R}_I)$.

\end{corollary}

\begin{remark}
Another relevant difference when one considers connected $K$-theory as opposed to algebraic cobordism, is a considerable simplification in the expressions describing the different operations. For instance, as  
\begin{equation} \label{CK}
\vartheta_{CK}(F(u,\chi(v)))=u+\vartheta_{CK}(\chi(v))-\beta u\cdot\vartheta_{CK} (\chi(v)) =u-\frac{v}{1-\beta v}+\frac{\beta u v}{1-\beta v}=\frac{u-v}{1-\beta v}\ ,
\end{equation}
we will be able to write out explicit formulas for the operators linked to $\Proj ^1$-bundles and the class $\vartheta_{CK}(\mathcal{R}_\emptyset)$. 
\end{remark}

Let us recall that for a $\Proj ^1$-bundle $\varphi:\Proj(E)\rightarrow X$ the operator $A_\varphi:\Omega^*(\Proj (E))\rightarrow \Omega^*(\Proj (E))$ had been defined from 
$$A:\Omega^*(X)[[y_1,y_2]]\rightarrow \Omega^*(X)[[y_1,y_2]]\ ,\  f\mapsto (1+\sigma)\frac{f}{F(y_1,\chi(y_2))}$$
 by substituting the Chern roots of $E$ for $y_1$ and $y_2$. Using (\ref{CK}), we can now rewrite $A^{CK}=A\otimes_{\mathbb{L}^*}\ZZ [\beta]:CK^*(X)[[y_1,y_2]]\rightarrow CK^*(X)[[y_1,y_2]]$ as follows:
 $$A^{CK}(f)= (1+\sigma)\left[\frac{(1-\beta y_2)f}{y_1-y_2}\right]= \frac{(1-\beta y_2)f}{y_1-y_2}+\frac{\sigma((1-\beta y_2)f )}{y_2-y_1}= \frac{(1-\beta y_2)f-(1-\beta y_1)\sigma(f)}{y_1-y_2}\ .$$   
%=(1+\sigma)\frac{f}{\vartheta_{CK}(F(y_1,\chi(y_2)))}

\begin{remark} \label{rem operators}
It is important to point out that the previous equality shows that the operator $A^{CK}_\varphi$ can be expressed in terms of the $\beta$-divided difference operators of definition \ref{def beta difference}: one only needs to change the sign of $\beta$ and consider $\phi^{(-\beta)}$.  
\end{remark}

It is now worth restating the content of proposition \ref{prop operator} after one has applied the functor $- \otimes_{\Laz^*}\ZZ[\beta]$. $A_\varphi\otimes_{\Laz^*}\ZZ[\beta]$ will be denoted as $A^{CK}_\varphi$. 

\begin{proposition} \label{prop operator CK}
Let $\varphi: \mathbb{P}(E)\rightarrow X$ be a $\mathbb{P}^1$-bundle and $A^{CK}_\varphi:CK ^*(\mathbb{P}(E))\rightarrow CK ^*(\mathbb{P}(E))$ be  the operator obtained from  
\begin{eqnarray*}
CK ^*(X)[[y_1,y_2]]&\stackrel{A^{CK}}{\longrightarrow}&\hspace{0.6 cm}CK ^*(X)[[y_1,y_2]] \\
f\hspace{0.8 cm}&\longmapsto &\frac{(1-\beta y_2)f-(1-\beta y_1)\sigma(f)}{y_1-y_2}\ ,
\end{eqnarray*}
 by substituting the Chern roots of $E$ for $y_1,y_2$.
Then  $A^{CK}_\varphi=\varphi ^*\varphi_*$. 
\end{proposition}

As it has been mentioned earlier, by means of (\ref{CK}) it is possible to write an explicit expression for the fundamental class of  $\Omega_{\omega_0}$ in $CK^*(\flag (V))$.
%Again using (\ref{CK}) it is possible to write a more explicit expression for the fundamental class of  $\Omega_{\omega_0}$ in $CK^*{\flag (V)}$.

\begin{proposition} \label{prop initial class CK}
Denote by $x_i$ and $y_i$ the Chern roots associated to the full flags $Q\unddot$ and $\pi^*(V\unddot)$. Then  
$$\vartheta_{CK^*}(\mathcal{R}_\emptyset)=\prod_{k+l\leq n}\frac{x_k-y_l}{1-\beta y_l}=\mathfrak{H}_{\omega_0}^{(-\beta)}(x_1,\tred,x_n,\chi_{F_m}(y_1),\tred,\chi_{F_m}(y_n)) \ .$$
\end{proposition}

\begin{proof}
The first equality follows immediately once one applies $\vartheta_{CK}$ to proposition \ref{prop initial class} and uses (\ref{CK}). For the second equality one only needs to recall the definition of $\beta$-polynomials and again use (\ref{CK}).
\end{proof}

We are now ready to express the fundamental class of any Schubert variety $\Omega_\omega$ as a rational function in the Chern roots arising from the flags $Q\unddot$ and $\pi^*V$.

\begin{theorem} \label{prop Bott-Samelson CK}
Let $\omega\in S_n$ and $I=(i_1,\tred,i_l)$ be any minimal decomposition of $\omega_0\omega$. Let $X\in \SM$. Denote by $j$ the inclusion of the Schubert variety $\Omega_\omega$ into $\flag (V)$ and by $x_i$ and $y_j$ the Chern roots associated to the full flags $Q\unddot$ and $\pi^*(V\unddot)$. Then the class $j_*\eta_{\Omega_\omega}\in CK^*(\flag (V))$ is given by
$$j_* \eta_{\Omega_\omega}=\mathfrak{H}^{(-\beta)}_\omega(x_1,\tred,x_n,\chi_{F_m}(y_1),\tred,\chi_{F_m}(y_n)) \ .$$
%$$j_* \eta_{\Omega_\omega}= A^{CK}_{i_l}\trecd A^{CK}_{i_1} \vartheta_{CK^*} (\mathcal{R}_\emptyset)\ , \text{ where }\vartheta_{CK^*}(\mathcal{R}_\emptyset)=\prod_{k+l\leq n}\frac{x_k-y_l}{1-\beta y_l}\ .$$ 
\end{theorem}

\begin{proof}
From corollary \ref{cor class} we know that the class $j_*\eta_{\Omega_\omega}$ coincides with $\vartheta_{CK}(\mathcal{R}_I)$ provived that $I$ is a minimal decomposition of $\omega_0\omega$. Moreover, thanks to theorem \ref{prop Bott-Samelson} we can express $\mathcal{R}_I$ by means of $\mathcal{R_\emptyset}$ and the operators $A_{i_j}$. Therefore, by functoriality, $\vartheta_{CK}(\mathcal{R}_I)$ can be expressed in terms of the operators $A^{CK}_{i_j}$ and of $\vartheta_{CK}(\mathcal{R}_\emptyset)$. More precisely we have
$$j_* \Omega_\omega=\vartheta_{CK^*}(\mathcal{R}_I)=
\vartheta_{CK}( A_{i_l}\trecd A_{i_1} (\mathcal{R}_\emptyset))=
A^{CK}_{i_l}\vartheta_{CK}( A_{i_{l-1}}\trecd A_{i_1} (\mathcal{R}_\emptyset))=\trecd=
 A^{CK}_{i_l}\trecd A^{CK}_{i_1} \vartheta_{CK} (\mathcal{R}_\emptyset)\ .$$
To finish the proof it is  now sufficient to invoke proposition \ref{prop initial class CK} and to observe that, as it was pointed out in remark \ref{rem operators}, the operators $A_{i_j}^{CK}$ coincide with the $\beta$-divided difference operators $\phi^{(-\beta)}$.
\end{proof}

\begin{remark}
It directly follows from proposition \ref{prop special} that the previous theorem specializes to theorems \ref{th flagbundle} and \ref{th flagbundle G_0}. One only has to apply the canonical natural transformations $CK^*\rightarrow CH^* $ and $CK^*\rightarrow K_0[\beta,\beta^{-1}]$ to the equality. This recovers immediately the result for the Chow ring, while for the Grothendieck ring it is still necessary to set $\beta$ equal to 1.   
\end{remark}

%{\large\textbf{Comparison with $CH^*$ and $K^0$}

 %that the equality clearly holds for $\omega_0$ as it suffices to confront (\ref{def SG2}) with (\ref{def H2}). Then one should notice
%or, in other words, that the statement holds for $\omega_0$ special polynomials $\mathfrak{H}_{\omega_0}$
%\end{proof}
%$\displaystyle{\frac{1-\beta x_2}{x_2-x_3}-\frac{1-\beta x_1}{x_1-x_3}=
%(x_1-x_2)(1-\beta x_3)}$ and that\\ $\displaystyle{\frac{1}{(x_1-x_2)}-\frac{1}{(x_1-x_3)}=\frac{x_2-x_3}{(x_1-x_2)(x_1-x_3)}}$

%This should remind the reader of the way the Bott-Samelson resolution were defined.   

%We now turn our attention to another oriented cohomology theory: connected K-theory. This theory is obtained from algebraic cobordism by tensoring over $\Laz^*$ with $\mathbb{Z}[\beta]$ and has the universal property described by the following theorem.
%We will now recall a few facts about connected $K$-theory. 

%\begin{thebibliography}{99}

%\bibitem{Fulton Flags}
%W. Fulton
%\emph{Flags, Schubert polynomials, degeneracy loci and determinantal formulas},
%Duke Math. J. \textbf{65} (1991), 381-420.  

%\bibitem{Levine Morel}
 %M. Levine, F. Morel
 %\emph{Algebraic cobordism}, Springer Monographs in Mathematics. Springer, Berlin, 2007

%\end{thebibliography}

\bibliographystyle{siam}
\bibliography{biblio}

\end{document}